\documentclass[11pt,letterpaper]{amsart}

\usepackage{amssymb, manfnt}
\usepackage{hyperref}
\usepackage[normalem]{ulem}
\usepackage{tikz}
\usepackage{booktabs}
\usepackage{multirow}
\usepackage{MnSymbol}
\usepackage{graphicx}
\usepackage{tikz-cd}
\usepackage{longtable}
\usepackage{amsmath,amscd}
\usepackage[all]{xy}
\graphicspath{ {graphs/} }
\usetikzlibrary{decorations.text}
\usepackage{graphicx}
\usepackage[outdir=./pictures]{epstopdf}
\graphicspath{ {pictures/} }

\newcommand{\arxiv}[1]{\href{http://arxiv.org/abs/#1}{\tt arXiv:\nolinkurl{#1}}}
\newcommand{\doi}[1]{\href{http://dx.doi.org/#1}{{\tt DOI:#1}}}

\newcommand{\googlebooks}[1]{(preview at \href{http://books.google.com/books?id=#1}{google books})}

\newcommand\dimHom{\operatorname{dimHom} }

\newcommand\Out{\operatorname{Out}}

\newcommand\Fun{\operatorname{Fun}}

\newcommand\id{\operatorname{id}}
\newcommand\Aut{\operatorname{Aut}}
\newcommand\TenAut{\operatorname{Eq}}
\newcommand\Rep{\operatorname{Rep}}

\newcommand\Hom{\operatorname{Hom}}

\newcommand\ad{\operatorname{Ad}}
\newcommand\BrAut{\operatorname{EqBr}}

\newcommand\Z[1]{ \mathbb{Z}_{#1}}

\newcommand\cC{ \mathcal{C}}
\newcommand\cF{ \mathcal{F}}
\newcommand\cat[3]{ \cC(   \mathfrak{#1}_{#2} , #3)    }
\newcommand\ssin[2]{ \sin\left(  \frac{#1}{#2}\right)   }

\newcommand\dcat[3]{ \cC(   \mathfrak{#1}_{#2} , #3)^0_{\operatorname{Rep}(\Z{m}) }   }
\newcommand\ecat[4]{ \cC(   \mathfrak{#1}_{#2} , #3)^0_{\operatorname{Rep}(\Z{#4}) }   }
\newcommand\fcat{ \cat{sl}{r+1}{k}_{\Rep(\Z{m'})}^\text{ad}   }
\newcommand\Stab{ \operatorname{Stab}_{\Z{r+1}} }

\def\altdb{\vadjust{\vbox to 0pt{\vss\hbox{\kern \hsize
\quad{\dbend}}\kern\baselineskip\kern-10pt}}}

\theoremstyle{plain}
\newtheorem{theorem}{Theorem}[section]
\newtheorem*{theorem*}{Theorem}
\newtheorem*{prop*}{Proposition}
\newtheorem{cor}[theorem]{Corollary}
\newtheorem{lemma}[theorem]{Lemma}
\newtheorem{prop}[theorem]{Proposition}
\newtheorem{conj}[theorem]{Conjecture}
\newtheorem{rmk}[theorem]{Remark}
\theoremstyle{remark}

\theoremstyle{definition}
\newtheorem{dfn}[theorem]{Definition}

\newcommand\Inv{\operatorname{Inv}}

\newcommand{\changeX}[1]{#1}
\newcommand{\changeY}[1]{#1}

\title{Type $II$ quantum subgroups of $\mathfrak{sl}_N$. $I$: Symmetries of local modules }
\author{Cain Edie-Michell}
\address{Cain Edie-Michell\\
University of New Hampshire\\
Durham, 
New Hampshire}
\email{cain.edie-michell@unh.edu}

\allowdisplaybreaks
\begin{document}

\begin{abstract}
This paper is the first of a pair that aims to classify a large number of the type $II$ quantum subgroups of the categories $\cat{sl}{r+1}{k}$. In this work we classify the braided auto-equivalences of the categories of local modules for all known type $I$ quantum subgroups of $\cat{sl}{r+1}{k}$. We find that the symmetries are all non-exceptional except for four cases (up to level-rank duality). These exceptional cases are the orbifolds $\ecat{sl}{2}{16}{2}$, $\ecat{sl}{3}{9}{3}$, $\ecat{sl}{4}{8}{4}$, and $\ecat{sl}{5}{5}{5}$.

We develop several technical tools in this work. We give a skein theoretic description of the orbifold quantum subgroups of $\cat{sl}{r+1}{k}$. Our methods here are general, and the techniques developed will generalise to give skein theory for any orbifold of a braided tensor category. We also give a formulation of orthogonal level-rank duality in the type $D$-$D$ case, which is used to construct one of the exceptionals. We uncover an unexpected connection between quadratic categories and exceptional braided auto-equivalences of the orbifolds. We use this connection to construct two of the four exceptionals.

In the sequel to this paper we will use the classified braided auto-equivalences to construct the corresponding type $II$ quantum subgroups of the categories $\cat{sl}{r+1}{k}$. This will essentially finish the type $II$ classification for $\mathfrak{sl}_n$ modulo type $I$ classification. When paired with Gannon's type $I$ classification for $r\leq 6$, our results will complete the type $II$ classification for these same ranks.

This paper includes an appendix by Terry Gannon, which provides useful results on the dimensions of objects in the categories $\cat{sl}{r+1}{k}$.
\end{abstract}

\maketitle

\section{Introduction}
 One of the oldest open problems in quantum algebra has been the program to classify the \textit{quantum subgroups} (or module categories, or Morita equivalence classes of algebra objects) of the categories $\cat{g}{}{k}$. This program was initially investigated in the language of \textit{conformal field theory} by Cappelli, Itzykson, and Zuber \cite{MR918402}. They used physical reasoning to argue that a quantum subgroup of $\cat{g}{}{k}$ is precisely the data needed to extend a \text{chiral Wess-Zumino-Witten} conformal field theory (constructed from $\mathfrak{g}$ and $k$) up to a \textit{full conformal field theory}. With this motivation in hand they were then able to give a combinatorial classification of the quantum subgroups of $\cat{sl}{2}{k}$. Their results were unexpected and exciting, falling into an $A-D-E$ pattern. The two infinite families $A$ and $D$ were expected, but far more intriguing were the three exceptional examples $E_6$, $E_7$, and $E_8$.

 Inspired by the richness of the $\mathfrak{sl}_2$ classification, there was a flurry of activity to give classification results for the higher rank Lie algebras \cite{MR1865095,MR2742806,MR2506168}. However this proved far more difficult than the rank one case. Despite the intense research activity directed towards the problem, very few new classification results were achieved. Once the dust had settled, a combinatorial classification for $\mathfrak{sl}_3$ had been given by Gannon \cite{MR1266482}, and $\mathfrak{sl}_4$ had been claimed by Ocneanu \cite{MR1907188}, but without supplied proof. It was here that the project stagnated, with many considering it to be intractable.

 In a more general setting, the problem of extending chiral conformal field theory up to full conformal field theory was studied rigorously by Fuchs, Runkel, and Schweigert \cite{MR1940282,MR2026879,MR2076134,MR2137114,MR2259258}. They were able to mathematically confirm the physical arguments of Cappelli, Itzykson, and Zuber. It was proven that the data to extend a chiral conformal field theory is precisely a module category over the representation category of the chiral theory. However, a module category is more than just its combinatorics, which is what was classified in \cite{MR918402} and \cite{MR1266482}. There is also the categorical data of the module category, which is captured by the 6-j symbols, or equivalently the associator, of the module. Thus classification for $\mathfrak{sl}_2$ and $\mathfrak{sl}_3$ was incomplete. The categorical data for the $\mathfrak{sl}_2$ case was worked out in the subfactor language in \cite{MR1193933,MR1145672,MR1313457,MR1929335,MR1308617,MR1617550,MR1976459}, and in the categorical language in \cite{MR1976459}. For the $\mathfrak{sl}_3$ case the categorical data was worked out in \cite{MR2553429}.
 
 There is a fundamental bifurification in classification program of quantum subgroups for any modular tensor category. This split occurs between the \textit{type $I$ quantum subgroups}, and the \textit{type $II$ quantum subgroups}. \changeY{These sub-classes of quantum subgroups are most easily defined using the Morita equivalence classes of algebra objects formalism. A quantum subgroup is called type $I$ if the Morita equivalence class of algebra objects contains a commutative representative, and it is called type $II$ if there is no such commutative representative}. The differences between these two cases means that different classification techniques are needed for each case. \changeY{There is also the distinction between \textit{non-exceptional} quantum subgroups, and \textit{exceptional} quantum subgroups. We say a quantum subgroup of $\cat{sl}{r+1}{k}$ is non-exceptional if it can be obtained as the category of modules of an algebra of the form $\operatorname{Fun}(G) \in \cat{sl}{r+1}{k}$, where $G$ is a finite group (necessarily a subgroup of $\mathbb{Z}_{r+1}$). A quantum subgroup is then exceptional if it is not non-exceptional.}

Recently there has been a massive revitalisation in the program to classifying quantum subgroups of the higher rank Lie algebras. This began with work of Schopieray \cite{MR3808050}, which gave level bounds on which categories $\cat{g}{}{k}$ could have exceptional type $I$ quantum subgroups for the rank two Lie algebras. These techniques were then drastically improved upon by Gannon \cite{Level-Bounds}, where effective level bounds were determined for all Lie algebras. In short, this allowed for a computer search to find all type $I$ quantum subgroups for any Lie algebra. These computer searches were performed by Gannon, and type $I$ classification was given for all ranks less than 7, a dramatic improvement on the state of knowledge. For these examples it was found that there are the expected infinite families of de-equivariantisation (or orbifold) type $I$ quantum subgroups, a finite number of type $I$ quantum subgroups coming from conformal inclusions of Lie groups \cite{MR1617550}, and four new examples not related to conformal inclusions of Lie groups. \changeY{We will refer to these latter four quantum subgroups as the \textit{truly exceptional} quantum subgroups.}

Thus the type $I$ case has essentially been solved, and classification up to higher ranks is now a matter of computer power, rather than mathematical insight. However, the type $II$ case (which comprise all remaining examples) still remains entirely open. This paper is the first in a pair to classify the type $II$ quantum subgroups for $\mathfrak{sl}_n$. The techniques developed in these papers will generalise to the other classical algebras. However we restrict our attention now to the type $A$ case for three reasons. First is that the details of working through the generalisation will require substantial effort that would push the length of these papers beyond a readable limit. Second is that combinatorial evidence suggests that type $A$ has the richest behaviour with type $II$ quantum subgroups, so we can expect to find the most interesting results by studying this case. Finally, historically the type $A$ case had received the most attention, and thus results in type $A$ will attract more interest than the other classical Lie algebras.

Our main tool to classify type $II$ quantum subgroups of the categories $\cat{sl}{r+1}{k}$ is the following theorem due to Davydov, Nikshych, and Ostrik, which gives a bijective correspondence between all quantum subgroups, and pairs of type $I$ quantum subgroups, and a braided equivalence between their categories of \textit{local modules}.
\begin{theorem}\cite{MR3022755}\label{thm:dno}
Let $\cC$ be a modular category. There is a bijective correspondence
\[    \{  \text{Irreducible modules over $\cC$} \} \leftrightarrow  \left\{  \parbox{7.8cm}{ Triples $( \mathcal{M}_1, \mathcal{M}_2, \mathcal{F}     )$, where \\ $\mathcal{M}_1$ and $\mathcal{M}_2$ are type $I$ module categories, and\\$\mathcal{F}: \mathcal{M}_1^0 \to \mathcal{M}^0_2$ is a braided equivalence. }  \right   \}                          \]
\end{theorem}
The work of Gannon has classified the type $I$ modules of $\mathfrak{sl}_n$ for $n \leq 7$. Thus to give the classification of the type $II$ modules, and hence complete the classification of all quantum subgroups, we need to determine all braided equivalences between their local modules. Gannon finds that there are three kinds of type $I$ modules \cite{Level-Bounds}. The first class (and most exciting as type $I$ modules) are the four truly exceptional examples, with two occurring at $\cat{sl}{6}{6}$ and two at $\cat{sl}{7}{7}$. These quantum subgroups have categories of local modules equivalent to:
\[
 \cat{so}{35}{1},\qquad \cat{sl}{2}{10}^\zeta, \qquad \operatorname{Vec},\quad \text{ and }\quad \operatorname{Vec},
\]
where $\cat{sl}{2}{10}^\zeta$ is a Galois conjugate of the category $\cat{sl}{2}{10}$. For all of these examples, the categories of local modules are completely understood.

\changeY{
\begin{rmk}
We wish to point out that the paper \cite{Level-Bounds} is unpublished as of the time of publication of this article, and the statements of the previous paragraph were provided to the author by Gannon in private communication. All of the theorems in this paper are independent from the results of \cite{Level-Bounds}, and the implicit claims of existence of certain exceptional type $II$ module categories over $\cat{sl}{r+1}{k}$ are rigorous. In the sequel to this paper, we classify all module categories over $\cat{sl}{r+1}{k}$ for $r\leq6$, which will require the results of \cite{Level-Bounds} to be rigorous.
\end{rmk}

}

 The second class consists of the module categories constructed from conformal inclusions of Lie groups. These can be found in \cite{MR3039775}, and for the type $A$ case they are:
\begin{align*}
A_{1,10} & \subset B_{2,1},\quad  &&A_{1,28}  \subset G_{2,1},\quad &&A_{2,9}  \subset E_{6,1},\\
A_{2,21} & \subset E_{7,1},\quad  &&A_{3,8}  \subset D_{10,1}, \quad &&A_{5,6}  \subset C_{10,1},\\
A_{7,10} & \subset D_{35,1},\quad &&A_{n, n-1}  \subset A_{\frac{(n-1)(n-2)}{2},1}, \quad &&A_{n, n+3}  \subset A_{\frac{n(n+3)}{2},1}, \\
A_{2n+1, 2n+2} & \subset B_{2n^2+ 4n + 1,1}, \quad && A_{2n, 2n+1} \subset D_{2n(n+1),1},\quad &&A_{2n+1, 4n+5}  \subset B_{4n^2 + 7n + 2,1}.
\end{align*}
 \changeY{Note that these conformal inclusions are embeddings of Wess-Zumino-Witten VOA's.} For all of these examples, the category of local modules is braided equivalent to $\changeY{\operatorname{Rep}(G_1) \simeq }\cat{g}{}{1}$, where $\mathfrak{g}$ is the corresponding Lie algebra of \changeY{the Lie group $G$ \cite[Theorem 5.2]{MR1936496}}. The third class consists of the infinite number of orbifold modules, constructed via de-equivariantisation. For $\cat{sl}{r+1}{k}$ these are parametrised by $m$ a divisor of $r+1$, such that $m^2 \divides k(r+1)$ is $r$ is even, or $2m^2 \divides k(r+1)$ is $r$ is odd. \changeY{For $m$ satisfying these conditions, we have that $\operatorname{Rep}(\mathbb{Z}_m)$ is a braided subcategory of $\cat{sl}{r+1}{k}$. This can be verified using the known formulas for the twists in $\cat{sl}{r+1}{k}$ \cite[Section 3.1]{MR1887583}.} We write $\cat{sl}{r+1}{k}_{\Rep(\Z{m})}$ for the type $I$ modules \changeY{coming from de-equivariantization by these Tannakian subcategories}. The category of local modules $\dcat{sl}{r+1}{k}$ for these examples is described in the bulk of the paper.

It is extremely rare that the categories of local modules for any of these type $I$ modules coincide. Thus the interesting type $II$ module categories of $\cat{sl}{r+1}{k}$ come from exceptional braided auto-equivalences of these categories of local modules. The goal of this paper is to determine the braided auto-equivalences of the categories of local modules for all known type $I$ quantum subgroups. In the sequel to this paper we will identify the small number of exceptions where the categories of local modules coincide, and work through the details of Theorem~\ref{thm:dno} in order to explicitly construct and classify the corresponding type $II$ quantum subgroups. Paired with Gannon's classification of type $I$ quantum subgroups, this will give type $II$ classification for $n\leq 7$. Further, our results of the sequel will show that for each $\mathfrak{sl}_{r+1}$, there is an effective bound on $k$ for which exceptional type $II$ quantum subgroups of $\cat{sl}{r+1}{k}$ can occur. These results will put us in a strong position to classify type $II$ modules for larger $n \geq 8$, once the type $I$ classification has been sorted for these $n$.

Let us examine the braided auto-equivalences of the local modules for the known type $I$ quantum subgroups. For the four truly exceptional examples found by Gannon we can quickly compute that the auto-equivalence groups are all trivial, except for the Galois conjugate of $\cat{sl}{2}{10}$ which has auto-equivalence group $\Z{2}$ \cite[Theorem 1.2]{Cain-normal}. For the type $I$ quantum subgroups coming from conformal inclusions of Lie groups, the group of braided auto-equivalences has been computed in earlier works of the author \cite[Theorem 1.1]{ABCG}. For completeness, we collect the results here.
\begin{theorem}
We have
\begin{align*}
&\BrAut( \cat{so}{5}{1})=&& \BrAut( \cat{g}{2}{1})=&&\BrAut( \cat{e}{7}{1}) = \\
                            &\BrAut( \cat{so}{2(2n^2 + 4n + 1)+1}{1})=&&\BrAut( \cat{so}{2(4n^2 + 7n + 2)+1}{1})=&&  \{e\} ,\\[1em]
&\BrAut( \cat{e}{6}{1})=&& \BrAut( \cat{so}{20}{1})=&&\BrAut( \cat{sp}{20}{1})= \\
& &&\BrAut( \cat{so}{70}{1}) =  &&  \Z{2} ,\\[1em]
&\BrAut( \cat{sl}{\frac{(n-1)(n-2)}{2}\changeY{+1}}{1})=&& \BrAut( \cat{sl}{\frac{n(n+3)}{2}\changeY{+1}}{1})=&&\Z{2}^{p+t} ,\\[1em]
&\BrAut( \cat{so}{4n(n+1)}{1})=&&\begin{cases}
\Z{2} \text{ if } n\equiv \{0,3\} \pmod 4\\
S_3 \text{ if } n\equiv \{1,2\} \pmod 4
\end{cases}
\end{align*}
where $p$ is the number of distinct odd primes that divide the rank plus one, and $t$ is equal to 1 if the rank is equivalent to $3$ mod $4$, and 0 otherwise.
\end{theorem}

Finally we have the orbifold type $I$ quantum subgroups. Somewhat paradoxically these have the most interesting categories of local modules, and hence determining their group of braided auto-equivalences is highly non-trivial. The remainder of this paper will be devoted to proving the following theorem, which determines the braided auto-equivalences groups in question. Excitingly we find a finite number of cases where the braided auto-equivalence group is exceptional, which corresponds to the existence of exceptional type $II$ quantum subgroups. These exceptional type $II$ quantum subgroups will be explicitly constructed in the sequel.

\begin{theorem}\label{thm:main}
Let $r \geq 1$ and $k\geq 2$ and $m$ a divisor of $r+1$ satisfying $m^2 \divides k(r+1)$ if $r$ is even, and  $2m^2 \divides k(r+1)$ if $r$ is odd. Then except for the cases
\begin{align*}
 &\ecat{sl}{2}{16}{2},\quad & \ecat{sl}{3}{9}{3},\quad &\ecat{sl}{4}{8}{4},\quad &\ecat{sl}{5}{5}{5},\\
 & \ecat{sl}{8}{4}{4},\quad &  \ecat{sl}{9}{3}{3},\quad & \ecat{sl}{16}{2}{2}, \text{ and } &\ecat{sl}{16}{2}{4}
 \end{align*}
we have that
\[   \BrAut(    \dcat{sl}{r+1}{k}   ) =   \begin{cases}
\{e\} \text{ if $k=2$ and $r=1$}\\
\Z{m'} \times  \Z{2}^{p+t} \text{ if $k=2$ or $r=1$} \\
D_{m'} \times \Z{2}^{p+t} \text{ otherwise }
\end{cases}\]
where
\begin{itemize}
\item $m' = \gcd(m,k)$,
\item $m'' = \frac{m}{m'}$,
\item $p$ is the number of distinct odd primes dividing $\frac{r+1}{mm''}$ but not $\frac{k}{m'}$, and
\item $t = \begin{cases} 0 \text{ if $\frac{r+1}{mm''}$ is odd, or if $\frac{k}{m'} \equiv 0 \pmod 4$, or if both $\frac{k}{m'}$ is odd, and $\frac{r+1}{mm''} \equiv 2 \pmod 4$} \\
 1 \text{ otherwise.}
\end{cases}$
\end{itemize}

For the remaining exceptional cases we have that

\begin{align*}
\BrAut(    \ecat{sl}{2}{16}{2}   ) & = S_3  ,\qquad  && \BrAut(    \ecat{sl}{3}{9}{3}   )  = S_4\\
\BrAut(    \ecat{sl}{4}{8}{4}   ) & =  S_4,\qquad &&\BrAut(    \ecat{sl}{5}{5}{5}   )  = A_5, \\
\BrAut(    \ecat{sl}{8}{4}{4}   ) & =  S_4, \qquad   &&\BrAut(  \changeY{  \ecat{sl}{9}{3}{3}}   )  = S_4\times \Z{2}\\
 \BrAut(    \ecat{sl}{16}{2}{2}   ) & = S_3  \times \Z{2}, \text{ and }  && \BrAut(    \ecat{sl}{16}{2}{4}   ) = S_3.
\end{align*}

\end{theorem}

With this theorem in hand, we are now placed to classify all type $II$ quantum subgroups whose type $I$ parents are in the known list. In particular this will allow us to classify all type $II$ quantum subgroups of $\mathfrak{sl}_n$ for $n\leq 7$. This will complete the classification of all quantum subgroups for these examples. As mentioned earlier, this type $II$ classification will be dealt with in the sequel to this paper. Extrapolating from the work of Gannon, we can expect the truly exceptional type $I$ quantum subgroups of the higher rank $\mathfrak{sl}_n$ to be exceedingly rare, and when they do occur, we can expect their categories of local modules to be somewhat trivial. This means that when the type $I$ classification has been extended to higher rank, the results of this paper will allow the type $II$ classification to nearly immediately follow.

With the motivation and main theorem of this paper described, let us move on to describing the structure of the article.

In Section~\ref{sec:prelims} we introduce the background required to begin this paper. We introduce the combinatorics of the categories $\cat{sl}{r+1}{k}$. In particular we give the formula for the dimensions of the simples, and prove useful inequalities which they obey. We describe the structure of the orbifold $\cat{sl}{r+1}{k}_{\Rep(\Z{m})}$, and of the local modules $  \dcat{sl}{r+1}{k}$. We explicitly determine useful structure of the category $  \dcat{sl}{r+1}{k}$, including the parametrisation of the simples, the group of invertibles, and the adjoint subcategory.

In Section~\ref{sec:nonexcep} we determine the so called \text{non-exceptional} braided auto-equivalences of $  \dcat{sl}{r+1}{k}$. These are the braided auto-equivalences which fix the image of the adjoint representation under the free module functor. The end result is the expected one, i.e. we show all non-exceptional braided auto-equivalences are either charge conjugation, simple current auto-equivalence, or come from the canonical $\Z{m}$-action. \changeY{Here a simple current auto-equivalence is a symmetry of the category constructed via the action of invertible elements, see \cite[Lemma 2.4]{ABCG} for additional details)}. Proving this result is highly technical, and requires several powerful techniques. The difficulty here is not surprising, as determining the non-exceptional braided auto-equivalences has troubled researchers working on this same problem in the past. To begin we develop skein theory for the adjoint subcategory of $ \dcat{sl}{r+1}{k}$. Our methods here are general, and will allow one to find skein theory for de-equivariantisation by an abelian group of any braided category, given that skein theory of the original category is known. With this skein theory in hand we can then use standard planar algebra techniques to find the non-exceptional braided auto-equivalence group of the adjoint subcategory. To extend these auto-equivalences to the entire category we use the techniques developed by the author in \cite{MR4192836}. These techniques give an upper bound on the number of auto-equivalences which may extend an auto-equivalence on the adjoint subcategory. By a happy coincidence, this upper bound is precisely realised by simple current auto-equivalences, introduced in the authors work \cite{MR4069181}, which was inspired by combinatorics from conformal field theory. This happy coincidence suggests the potential for a general theorem.
\begin{conj}
Let $\cC$ be a modular tensor category, $\cC^\text{ad}$ its adjoint subcategory, and $\cF$ an auto-equivalence of $\cC$ which restricts to the identity on $\cC^\text{ad}$. Then $\cF$ is isomorphic to a simple current auto-equivalence.
\end{conj}
The validity of this general conjecture remains to be investigated. All together, the results of this section fully classify all non-exceptional braided auto-equivalences of the categories $\dcat{sl}{r+1}{k}$.

In Section~\ref{sec:excep} we investigate the combinatorics of the exceptional braided auto-equivalences of $\dcat{sl}{r+1}{k}$. We show that with a finite number of exceptions, that every braided auto-equivalence of $\dcat{sl}{r+1}{k}$ is non-exceptional, and hence is covered by the results of the previous section. Our main observation here is simple. If there were an exceptional braided auto-equivalence of $\dcat{sl}{r+1}{k}$, then its image of the adjoint would have the same dimension and twist. This puts massive combinatorial restrictions on the objects of the category $\cat{sl}{r+1}{k}$. By studying these restrictions in a case by case analysis, we are able to obtain a series of inequalities which imply that both the rank and level must be small. From here we can computer search to find the finite cases where $\dcat{sl}{r+1}{k}$ has an exceptional braided auto-equivalence, at the level of the fusion ring and twists. Up to level-rank duality we find four possible candidates for exceptional braided auto-equivalences. These are $\ecat{sl}{2}{16}{2}$, $\ecat{sl}{3}{9}{3}$, $\ecat{sl}{4}{8}{4}$, and $\ecat{sl}{5}{5}{5}$. While the case by case analysis is messy, and a uniform approach to this section would be desired, the exceptional examples which are discovered mean that such a uniform approach is unlikely to exist.

In Section~\ref{sec:real} we finish up by realising all of the exceptional braided auto-equivalences of $\dcat{sl}{r+1}{k}$ for the finite number of remaining cases identified in the previous section. We see two situations at hand. The first has already been observed in the literature \cite{MR2559686} in the $\mathfrak{sl}_2$ case, and concerns the categories $\ecat{sl}{2}{16}{2}$ and $\ecat{sl}{4}{8}{4}$. Here the exceptional braided auto-equivalences exist due to coincidences of categories connecting them to the Lie algebra $\mathfrak{so}_8$ and hence triality. The second situation is much more interesting and exotic. We show a connection between the two remaining examples and $\ecat{sl}{4}{8}{4}$, and three explicit \textit{quadratic categories}. This connection is sufficiently explicit, so that having a construction of the quadratic categories allows us the construction of the exceptional braided auto-equivalences. For the two remaining cases, we have that the corresponding quadratic categories have been constructed by Izumi \cite{MR3635673,MR3827808}, which allows these cases to be resolved.

The connection between type $II$ quantum subgroups, and quadratic categories appears to be more than just a convenient coincidence. It occurs for other Lie algebras outside the $A$ series, and the author will weakly conjecture that every exceptional type $II$ quantum subgroup for a simple Lie algebra comes from either a coincidence of categories, or from a connection to a quadratic category. We will not say much more on this to avoid spoiling future work.

This paper also includes an appendix authored by Terry Gannon which contains some results on the combinatorics of the categories $\cat{sl}{r+1}{k}$. The results of this appendix are a key ingredient for the computations of Section~\ref{sec:excep}.

\subsection*{Acknowledgements}
We would like to thank Dietmar Bisch for many clarifying conversations throughout the duration of this work. We would also like to thank Scott Morrison for helpful conversations on skein theory for orbifolds, Pinhas Grossman and Masaki Izumi for helpful conversations on the Cuntz algebra construction of quadratic categories, Victor Ostrik for pointing out to us that the exceptional braided auto-equivalence of $\ecat{sl}{4}{8}{4}$ exists due to a coincidence of categories, and finally Terry Gannon for helpful conversations on the combinatorics of the categories $\cat{sl}{r+1}{k}$, sharing his results in the type $I$ case, and for authoring the appendix. Finally, we would like to thank the referees for their many useful suggestions.

This material is based work supported by the National Science Foundation under Grant No. DMS-1440140 while the author was in residence at the Mathematical Sciences Research Institute in Berkeley, California, during the Spring 2020 semester. In addition, the author was supported by NSF grant DMS 2137775 and a AMS-Simons Travel Grant.

\section{Preliminaries}\label{sec:prelims}

We refer the reader to \cite{MR3242743} for the basics of fusion categories.

\subsection{Quantum Integers, Dimensions, and Inequalities}

The main object of study in this paper will be the modular tensor categories $\cat{sl}{r+1}{k}$, the category of level $k$ integrable representations of $\hat{\mathfrak{sl}_n}$. For an overview us these categories see \cite{MR4079742}. For our purposes we will only require some basic combinatorics of these categories. The simple objects of $\cat{sl}{r+1}{k}$ are parametrised by
\[           \sum_{i= 0}^r \lambda_i \Lambda_i   \quad \text{ where } \lambda_i \in \mathbb{N} \quad \text{ and } \quad \sum_{i=0}^r \lambda_i = k.           \]
Often we will omit the $\lambda_0$ term of a simple object, as its value can be deduced from the remaining $\lambda_i$'s. For example, the vector representation $(k-1)\Lambda_0 + \Lambda_1$ will usually be written simply as $\Lambda_1$. A special subset of these simples are the $r+1$ invertibles (or simple currents), which are the objects
\[        \{    k\Lambda_i : i \in  \Z{r+1}\}.\]

To describe the quantum dimensions of the simple objects of $\cat{sl}{r+1}{k}$ we will need two ingredients. The first are the quantum integers.
\begin{dfn}
We define the $n$-th quantum integer (as a function of $r$ and $k$) as
\[        [n]_{r,k} := \frac{     q ^n - q  ^{-n}}{q - q^{-1}} \quad \text{ and } \quad q  =e^{2\pi i \frac{  1}{2 (1 + k + r)}   }.\]
\end{dfn}

The second ingredient is the \text{hook formula}, which gives the quantum dimension of a simple of $\cat{sl}{r+1}{k}$ in terms of quantum integers. To describe this formula, we have to introduce the tableaux of a simple object. Let $ X = \sum_{i= 0}^r \lambda_i \Lambda_i $ be a simple object, and define a $r\times k$ tableaux $T(X)$ whose $j$-th row contains $\sum_{i = j}^r \lambda_i$ boxes. For each box $(x,y)$ in the tableaux $T(X)$ we can define the content, which is the quantum integer $[r+1 + - x + y]_{r,k}$, and the hook length, which is the quantum integer $[h]_{r,k}$, where $h$ is the number of boxes with the same $x$ or $y$ coordinate. The quantum dimension of $X$ is the product over all the boxes of $T(X)$ of the contents divided by the hooks. For a quick example, we have that the tableaux for the object $\Lambda_1 + \Lambda_2 \in  \cat{sl}{r+1}{k}$ has two boxes in row one, and one box in row two. Thus the contents are
\[    [r+1]_{r,k}, \quad   [r]_{r,k}, \quad \text{and} \quad  [r+2]_{r,k},\]
and the hooks are
\[    [3]_{r,k}, \quad   [1]_{r,k}, \quad \text{and} \quad  [1]_{r,k}.\]
Therefore the hook formula tells us that the quantum dimension of $\Lambda_1 + \Lambda_2$ is $\frac{[r]_{r,k}[r+1]_{r,k}[r+2]_{r,k}}{[3]_{r,k}}$.

There are two natural actions of the simples of $\cat{sl}{r+1}{k}$ that preserve the dimensions. These are \textit{charge-conjugation} which sends
\[       \sum_{i= 0}^r \lambda_i \Lambda_i   \mapsto  \sum_{i= 0}^r \lambda_i \Lambda_{-i}  .\]
The fact that this map preserves dimensions can be deduced from the hook formula.  The other action comes from simple currents, which sends
\[   \sum_{i= 0}^r \lambda_i \Lambda_i   \mapsto \sum_{i= 0}^r \lambda_i \Lambda_{i+a}    \text{ for }  a \in \Z{r+1}.  \]
This map preserves dimensions as it is simply tensoring by the invertible $k\Lambda_a$. For an object $X\in \cat{sl}{r+1}{k}$ we write $[X]$ for its orbit under the action of simple currents.

For a given object $X\in \cat{sl}{r+1}{k}$ it will be useful to know which subgroup of invertibles fix $X$. To that end we introduce the following notation.
\begin{dfn}
Let $X \in \cat{sl}{r+1}{k}$ a simple object. Given $\Z{d}$ a subgroup of the invertibles of $\cat{sl}{r+1}{k}$, we define
\[   \operatorname{Stab}_{\Z{d}}(X):= \{ g \in \Z{d} : g\otimes X \cong X\}.   \]
\end{dfn}

The quantum dimensions of the simple objects of $\cat{sl}{r+1}{k}$ satisfy a variety of useful equalities and inequalities.

Our main tool is the fact that the dimensions of the simples of $\cat{sl}{r+1}{k}$ respect the geometry of the truncated Weyl chamber in a nice manner. Namely if one draws a convex hull in the truncated Weyl chamber, then the minimum of the dimensions in this hull will occur at the corners.
\begin{lemma}\label{lem:convex}\cite{MR1887583}
For $1 \leq i \leq N$, let $X_i \in \cat{sl}{r+1}{k}$ simple objects, and $t_i \in [0,1]$ such that $\sum_{i=1}^N t_i = 1$. Then
\[  \dim\left( \sum_{i=1}^N t_i X_i  \right) \geq \min\left\{ \dim( X_i ) : 1 \leq i \leq N\right\},\]
with equality occurring exactly at the corners of the convex hull.
\end{lemma}

We also have the following inequalities of quantum integers which occur due to the cut-off of the level $k$ in the truncated Weyl chamber.

\begin{lemma}\label{lem:quanInc}
For all $1 \leq n\leq r+ k$ we have
\[   [n]_{r,k}  < [n]_{r+1,k} \quad \text{ and } \quad  [n]_{r,k}  < [n]_{r,k+1}.\]
\end{lemma}
\begin{proof}
The second inequality holds as $ [n]_{r,k}$ is equal to $[n]_{n-1,r + k-n+1}$. The value $[n]_{n-1,r + k-n+1}$ is precisely the graph norm of the fusion graph for $\Lambda_1 \in \cat{sl}{n}{r+k-n+1}$. This fusion graph embeds in the fusion graph for $\Lambda_1 \in \cat{sl}{n}{r+k-n+2}$, which has graph norm $[n]_{n-1,r + k-n+2}$. As graph norms respect inclusions, we get that
\[    [n]_{n-1,r + k-n+1} <[n]_{n-1,r + k-n+2},   \]
which is equivalent to
\[    [n]_{r, k} <[n]_{r,k+1}.   \]

The first inequality now holds as $ [n]_{r,k} = [n]_{k,r}$ .
\end{proof}

Often it will be useful to bound a quantum integer by a simpler function of $n$. The following inequalities allow us exactly that. The first bounds the quantum integer above.

\begin{lemma}\label{lem:above}\cite{MR3808050}
For all $n \geq 1$ we have
\[[n]_{r,k} \leq n.\]
\end{lemma}

The second bounds the quantum integer below.

\begin{lemma}\label{lem:below}\cite{MR3808050}
Suppose that $1\leq n \leq \frac{c-1}{c}(1 + r + k)$ for some $c \in \mathbb{N}$, then
	\[ [n]_{r,k} \geq \frac{1}{c}.\]
\end{lemma}

With these general inequalities in hand, we can now prove a collection of useful inequalities on the dimensions of the simples of $\cat{sl}{r+1}{k}$.

\begin{lemma}\label{lem:gsym}
For all $0\leq h \leq k$ and $0\leq a \leq r$ we have
\[  \dim\left(  h\Lambda_a \right) = \dim\left(  (k-h)\Lambda_a \right).\]
\end{lemma}
\begin{proof}
By applying a simple current symmetry, we see that the objects $( h\Lambda_a)$ and $((k- h)\Lambda_{-a})$ have the same dimension. Applying charge conjugation then gives the result.
\end{proof}

Our first true inequality gives bounds on the symmetric powers of the fundamental representations.

\begin{lemma}\label{lem:orderL}
Let $1 \leq a \leq r+1$ and $1\leq \lambda_a \leq \frac{k-1}{2}$. If $ \lambda_a \leq j \leq k - \lambda_a$ then
\[ \dim(  j\Lambda_a)\geq  \dim(   \lambda_a\Lambda_a) .\]
\end{lemma}

\begin{proof}
From Lemma~\ref{lem:gsym} we have
\[   \dim\left(  \lambda_a\Lambda_a   \right) =  \dim\left(  (k - \lambda_a)\Lambda_a ) \right).\]

We can write
\[   j\Lambda_a = \left(1 - \frac{j-\lambda_a }{ k-2\lambda_a }\right)\cdot \lambda_a\Lambda_a + \frac{j-\lambda_a }{k-2\lambda_a }\cdot     (k- \lambda_a)\Lambda_a. \]
Thus the result follows from Lemma~\ref{lem:convex}.
\end{proof}

Applying level-rank duality to the above bound, we can also obtain the following. Together, these bounds allow us to understand the ordering on the dimensions of the symmetric powers of the fundamental representations.

\begin{lemma}\label{lem:orderR}
Let $1 \leq a \leq \frac{r}{2}$ and $1 \leq \lambda_a \leq k$. If $a \leq j \leq r+1 - a$ then
\[   \dim(  \lambda_a\Lambda_{j}    ) \geq  \dim(  \lambda_a\Lambda_{a}    ) . \]
\end{lemma}
\begin{proof}
Via a level-rank duality, we have that the dimension of $(\lambda_a\Lambda_{j} )$ in $\cat{sl}{r+1}{k}$ is equal to the dimension of $(j\Lambda_{\lambda_a})$ in $\cat{sl}{k}{r+1}$. The result then follows from Lemma~\ref{lem:orderL}.
\end{proof}

The last bound we will give applies to objects that are fixed by the invertible objects of $\cat{sl}{r+1}{k}$. If this stabaliser subgroup of an object is non-trivial, then the following lemma gives strong restrictions on the dimension of that object.

\begin{lemma}\label{lem:boundfix}
Suppose $X \in \cat{sl}{r+1}{k}$ with $\Stab(X) = \Z{d}$. Let $m \in \mathbb{N}$, and $0\leq a \leq k$. Then
\[ \dim(X) \geq \dim\left( a\Lambda_{m\frac{r+1}{d}}\right).   \]
\end{lemma}
\begin{proof}
As $\Stab(X) = \Z{d}$ we have that $X$ is fixed by $k \Lambda_{\frac{r+1}{d}}$. Thus $X$ is of the form
\[       \sum_{i =1}^{\frac{r+1}{d}} \lambda_i \sum_{j = 1}^d \Lambda_{i + j\frac{r+1}{d}},   \]
with
\[   \sum_{i =1}^{\frac{r+1}{d}} \lambda_i  = \frac{k}{d}.  \]
 Let $m \in \mathbb{N}$, and $0\leq a \leq k$ and define for $1 \leq i \leq r+1$ the objects
 \[  P_i := (k-a)\Lambda_i + a\Lambda_{i + m\frac{r+1}{d}}.\]
 As $k\Lambda_{1} \otimes P_i = P_{i+1}$, we have that $\dim(P_i) = \dim(P_j)$ for all $1\leq i,j\leq r+1$.

We claim that
 \[  X = \sum_{i =1}^{\frac{r+1}{d}} \frac{\lambda_i}{k} \sum_{j = 1}^d P_{i + j\frac{r+1}{d}}.\]
 To see this we count the multiplicity of an arbitary $\Lambda_{\ell}$ in both sides of the above equation. In the object $X$, the multiplicity of $\ell$ is equal to $\lambda_{ \ell \pmod {\frac{r+1}{d}}}$. On the right hand side, $\Lambda_{\ell}$ will appear in the $P_{\ell}$ term where it appears with multiplicity $a\frac{\lambda_{\ell  \pmod {\frac{r+1}{d}}}}{k}$, and in the $P_{\ell -  m\frac{r+1}{d}}$ term where it appears with multiplicity $(k-a)\frac{\lambda_{\ell -  m\frac{r+1}{d} \pmod {\frac{r+1}{d}}}}{k} = (k-a)\frac{\lambda_{\ell  \pmod {\frac{r+1}{d}}}}{k}$. Thus in the entire right hand side, $\Lambda_{\ell}$ appears with multiplicity $\lambda_{ \ell \pmod {\frac{r+1}{d}}}$, and so the claim is valid.

 As $\sum_{i =1}^{\frac{r+1}{d}} \frac{\lambda_i}{k} \sum_{j = 1}^d 1 = 1$, we can use Lemma~\ref{lem:convex} to see that
 \[  \dim(X) \geq \min(  \dim(P_i ) : 1\leq i \leq r+1 ) = \dim(  P_0) = \dim( a\Lambda_{ m\frac{r+1}{d}  }   )\]
 as desired.

\end{proof}

\subsection{De-equivariantization}

Our main focus of study in this paper will be the orbifold type $I$ quantum subgroups, and their local modules. These are constructed as \textit{de-equivariantisations} of the modular categories $\cat{sl}{r+1}{k}$.

In general let $\cC$ be a braided tensor category, and choose a distinguished subcategory braided equivalent to $\Rep(G)$ for $G$ a finite group. We can consider the function algebra $\Fun(G) \subset \Rep(G) \to \cC$, which lifts to a commutative algebra in $\mathcal{Z}(\cC)$ via the braiding. We write $(\Fun(G), \sigma)$ for this commutative central algebra object.

\begin{dfn}
The \textit{de-equivariantisation} of $\cC$ by $\Rep(G)$ is defined as the category of $\Fun(G)$ modules, which can be endowed with the structure of a $G$-crossed braided category via $\sigma$. We write $\cC_{\Rep(G)}$ for this de-equivariantisation.
\end{dfn}

The category $\cC_{\Rep(G)}$ has the canonical structure of a $G$-crossed braided category (see \cite{MR2059630,MR1923177,0005291}). The $G$-action is given by left translation of the algebra $\Fun(G)$ i.e. multiplication by group elements. The category of local modules of the commutative central algebra object $(\Fun(G), \sigma)$ is the trivially graded subcategory of $\cC_{\Rep(G)}$ with respect to the $G$-crossed structure. \changeY{We write $\cC_{\Rep(G)}^0$ for this category of local modules.}

With the generalities out of the way, let us now focus on the specific de-equivariantisations of interest for this paper. The Tannakian subcategories of $\cat{sl}{r+1}{k}$ are completely understood and classified. These subcategories of $\cat{sl}{r+1}{k}$ are parametrised by $m$ a divisor of $r+1$ satisfying $m^2 \divides k(r+1)$ if $r$ is even, and $2m^2 \divides k(r+1)$ if $r$ is odd. The corresponding Tannakian subcategory is equivalent to $\Rep(\Z{m})$, and is generated by the invertible object $k\Lambda_{\frac{r+1}{m}}$.

Our goal is to describe the basic structure of the modular tensor category $\dcat{sl}{r+1}{k}$. We begin by looking at the $\Z{m}$-crossed braided category $\cat{sl}{r+1}{k}_{\Rep(\Z{m})}$. The objects of $\cat{sl}{r+1}{k}_{\Rep(\Z{m})}$ are of the form $(Y,\rho)$, where $Y$ is an object which is fixed by tensoring by $k\Lambda_{\frac{r+1}{m}}$, and $\rho$ is a choice of isomorphism $Y \to k\Lambda_{\frac{r+1}{m}} \otimes Y$ satisfying a standard coherence condition. We have the free module functor
\[ \cF_{\Z{m}} :    \cat{sl}{r+1}{k} \to \cat{sl}{r+1}{k}_{\Rep(\Z{m})}\]
given by tensoring with the algebra object $\Fun(\Z{m}) = \oplus_i  k\Lambda_{i\frac{r+1}{m}}$. It is known that the functor $\cF_{\Z{m}}$ is dominant \cite[Proposition 5.5]{MR2863377}. The adjoint to $\cF_{\Z{m}}$ is the lax monoidal functor given by forgetting the isomorphism $\rho$. That is
\[   \cF^*_{\Z{m}} ( (Y, \rho)) = Y.\]

In order to simplify our proofs and computations later, it is necessary to give a more elementary description of the simple objects of $\cat{sl}{r+1}{k}_{\Rep(\Z{m})}$. \changeY{Our skein theoretic version of the proof can be found in Lemma~\ref{lem:simps}.}

\begin{lemma}\changeY{\cite[Corollary 5.3]{MR1749250}}
The simple objects of $\cat{sl}{r+1}{k}_{\Rep(\Z{m})}$ are parameterised by pairs $(X, \chi_X)$, where $X$ is a simple object of $\cat{sl}{r+1}{k}$ considered up to action by $k\Lambda_{\frac{r+1}{m}}$, and $\chi_X$ is a character of the group $\operatorname{Stab}_{\Z{m}}(X)$.

The dimension of the simple object $(X, \chi_X)$ is given by
\[ \dim( X,\chi_X ) = \frac{\dim(X)}{|\operatorname{Stab}_{\Z{m}}(X)|}  \]

The canonical $\Z{m}$-action on these simples is given by multiplication of $\chi_X$ by the standard character $g^j \mapsto e^{2 \pi i \frac{j}{  |\operatorname{Stab}_{\Z{m}}(X)|}}$.
\end{lemma}

Under this parametrisation, we can explicitly describe the free module functor $\cF_{\Z{m}}$. We have
\[  \cF_{\Z{m}}(X) = \bigoplus_{\chi \in \hat{G} } (X, \chi): \changeY{\cat{sl}{r+1}{k}\to \cat{sl}{r+1}{k}_{\Rep(\Z{m})}}.\]
\changeY{Note that the restriction of the adjoint $\cF_{\Z{m}}^*:  \cat{sl}{r+1}{k}_{\Rep(\Z{m})} \to \cat{sl}{r+1}{k}$ to the subcategory $  \cat{sl}{r+1}{k}_{\Rep(\Z{m})}^0$ is a ribbon lax monoidal functor \cite[Lemma 3.10]{MR1749250} or \cite[Lemme 3.3]{MR1741269}. In practice, this means that the twist of $(X, \chi) \in   \cat{sl}{r+1}{k}_{\Rep(\Z{m})}^0$ is equal to the twist on $X \in\cat{sl}{r+1}{k}$. }

To obtain the simple objects of the category $\dcat{sl}{r+1}{k}$ we must take the objects which are $0$-graded in the $\Z{m}$-graded category $\cat{sl}{r+1}{k}_{\Rep(\Z{m})}$. The $\Z{m}$-grading on the category $\cat{sl}{r+1}{k}_{\Rep(\Z{m})}$ is inherited from the $\Z{r+1}$-grading on the category $\cat{sl}{r+1}{k}$. Thus a simple object $(   \sum_{i=1}^r \lambda_i \Lambda_i , \chi)$ will live in the $\sum_{i=1}^r i \lambda_i$ graded component of $\cat{sl}{r+1}{k}_{\Rep(\Z{m})}$, taken modulo $m$. This gives the following.

\begin{lemma}
The simple objects of $\dcat{sl}{r+1}{k}$ are parametrised by pairs $(   \sum_{i=1}^r \lambda_i \Lambda_i , \chi)$ where $\sum_{i=1}^r i \lambda_i \equiv 0 \pmod m$, and $\chi$ is a character of $\operatorname{Stab}_{\Z{m}}( \sum_{i=1}^r \lambda_i \Lambda_i)$.
\end{lemma}

\changeX{Let us single out a distinguished object of $\dcat{sl}{r+1}{k}$. Consider the object $(\Lambda_1 + \Lambda_r) \in \cat{sl}{r+1}{k}^{\text{ad}}$. We can take the image of this object under the functor $\cF_{\Z{m}}$ to obtain an object in $ \left(\cat{sl}{r+1}{k})_{\operatorname{Rep}(\mathbb{Z}_m)}\right)^{ad} \subset \dcat{sl}{r+1}{k}$.

\begin{dfn}
We define the object
\[  \Omega :=  \cF_{\Z{m'}}(\Lambda_1 + \Lambda_r) \in \fcat .\]
\end{dfn}

The distinguished simple object $\Omega \in \dcat{sl}{r+1}{k}$ satisfies several nice properties that will make it useful for us in our computations later. Immediately we have that $\Omega$ is self-dual, its dimension is $[r]_{r,k}[r+2]_{r,k}$, and there exists a map $\Omega \otimes \Omega \to \Omega$. 
}

We now compute some useful information about the categories $\dcat{sl}{r+1}{k}$. Let us define $m' = \gcd(m, k)$ and $m'' := \frac{m}{m'}$.

For the remainder of the paper we will constantly encounter three exceptions in nearly all of our lemmas and proofs. These are the categories $\ecat{sl}{2}{4}{2}$, $\ecat{sl}{3}{3}{3}$, and $\ecat{sl}{4}{2}{2}$. To put them to rest, we deal with them now.
\begin{lemma}
The claims of theorem~\ref{thm:main} hold for the categories $\ecat{sl}{2}{4}{2}$, $\ecat{sl}{3}{3}{3}$, and $\ecat{sl}{4}{2}{2}$.
\end{lemma}
\begin{proof}
From the formula of the dimensions of the simples of $\dcat{sl}{r+1}{k}$ we immediately see that each of these cases is pointed. By considering twists we find that
\begin{align*}
\ecat{sl}{2}{4}{2} &\simeq   \operatorname{Vec}( \Z{3}, \{1, e^{2 \pi i \frac{1}{3}},e^{2 \pi i \frac{1}{3}}) \\
\ecat{sl}{3}{3}{3} &\simeq   \operatorname{Vec}( \Z{2}\times \Z{2}, \{1, -1,-1,-1) \\
\ecat{sl}{4}{2}{2} &\simeq   \operatorname{Vec}( \Z{6}, \{1, e^{2 \pi i \frac{5}{12}}, e^{2 \pi i \frac{2}{3}},e^{2 \pi i \frac{3}{4}},e^{2 \pi i \frac{2}{3}},e^{2 \pi i \frac{5}{12}}\},
\end{align*}
where the second argument describes the non-degenerate quadratic form on $\Z{n}$. With these explicit presentations, it is straight-forward to verify that they satisfy the claims of Theorem~\ref{thm:main}.
\end{proof}

\begin{rmk}\label{rmk:special}
In order to keep the statements of various lemma tidy, for the remainder of this paper we will implicitly assume that the three cases $\ecat{sl}{2}{4}{2}$, $\ecat{sl}{3}{3}{3}$, and $\ecat{sl}{4}{2}{2}$ are ignored.
\end{rmk}

\changeX{With the special cases mentioned in the above remark excluded, we can show that the object $\Omega$ is simple.

\begin{lemma}
The object $\Omega$ is a simple object in $\dcat{sl}{r+1}{k}$.
\end{lemma}
\begin{proof}
This lemma is equivalent to showing that $\operatorname{Stab}_{\Z{m}}( \Lambda_1 + \Lambda_r)$ is trivial. Let $j \in \Z{m}$, then the corresponding invertible object of $\cat{sl}{r+1}{k}$ is $k\Lambda_{j\frac{r+1}{m}}$. We compute
\[  k\Lambda_{j\frac{r+1}{m}} \otimes   ( \Lambda_1 + \Lambda_r) \cong \left(  (k-2)\Lambda_{j\frac{r+1}{m}} + \Lambda_{j\frac{r+1}{m}+1} + \Lambda_{j\frac{r+1}{m}-1}\right).  \]
A case by case analysis, where we consider $k\geq 4$, $k=3$, and $k=2$ gives the desired result.
%
%
\end{proof}}

Let us study the group of invertibles of the modular category $\dcat{sl}{r+1}{k}$.

\begin{lemma}\label{lem:invfull}
We have
\[  \Inv(   \dcat{sl}{r+1}{k}   )  =  \left\{  (k\Lambda_{\ell m''},1) : \ell \in  \Z{\frac{r+1}{mm''}} \right \}   \cong \Z{\frac{r+1}{mm''}}.\]
\end{lemma}
\begin{proof}
From the formula for the dimensions of the simples of $\dcat{sl}{r+1}{k}$, it is clear that $(X,\chi_X) \in \dcat{sl}{r+1}{k}$ will be invertible if and only if $X$ has integer dimension, and lives in a graded component of $\cat{sl}{r+1}{k}$ which is a multiple of $m$. The objects with integer dimension in $\cat{sl}{r+1}{k}$ have been classified \cite{MR2274523}, and aside from the special cases we have discarded, the only such objects are the invertibles.

 The invertible objects of $\cat{sl}{r+1}{k}$ are of the form $k\Lambda_i$ for $i \in \Z{r+1}$. These invertible objects live in the graded component $ki$ of $\cat{sl}{r+1}{k}$. Hence, to find the invertible objects of $\dcat{sl}{r+1}{k}$, we need to see which $i\in \Z{r+1}$ satisfy the equation $ki = N m$ for some $N\in \mathbb{N}$. We can write this equation as $\frac{k}{m'} i = N m''$. As a consequence of the definition of $m'$ and $m''$, we have that $\frac{k}{m'} $ and $m''$ are coprime. Thus, $m''$ divides $i$, and so $i$ is a multiple of $m''$. This tells us the group of invertibles of $\dcat{sl}{r+1}{k}$ is generated by the object $(k\Lambda_{m''},1)$, and hence form a group isomorphic to $\Z{\frac{r+1}{mm''}}$.
\end{proof}

The category $\dcat{sl}{r+1}{k}$ is modular, which implies the universal grading group is isomorphic to the group of invertibles. Hence we get the following corollary.

\begin{cor}
The universal grading group of $\dcat{sl}{r+1}{k}$ is the group $\Z{\frac{r+1}{mm''}}$.
\end{cor}

Let us now identify the adjoint subcategory of $\dcat{sl}{r+1}{k}$. Knowing this subcategory will allow us to use powerful graded category techniques. A natural guess would be that $\left({\dcat{sl}{r+1}{k}}\right)^\text{ad} \simeq \left( \cat{sl}{r+1}{k}^{\text{ad}}\right)_{\Rep(\Z{m})}$. However this doesn't even typecheck, as $\Rep(\Z{m})$ is not necessarily always a subcategory of $ \left( \cat{sl}{r+1}{k}^{\text{ad}}\right)$ i.e. consider $\cat{sl}{8}{3}_{\Rep(\Z{2})}$. Instead we find that the adjoint subcategory is equivalent to $\fcat$. As $m'$ divides $k$, we have that $\Rep(\Z{m'})$ is a subcategory of $\cat{sl}{r+1}{k}^{\text{ad}}$.

%
%
%

\begin{lemma}\label{lem:adsub}
We have
\[    \left({\dcat{sl}{r+1}{k}}\right)^\text{ad}\simeq \fcat\]
\end{lemma}

\begin{proof}

The category $\fcat$ naturally embeds in ${\dcat{sl}{r+1}{k}}$ via the identity functor. As $\fcat$ is generated by the simple object $\Omega$, and $\dimHom(\Omega \otimes \Omega \to \Omega) \geq 1$, the adjoint subcategory of $\fcat$ is itself. Hence the adjoint subcategory of ${\dcat{sl}{r+1}{k}}$ contains $\fcat$. The global dimension of $\left({\dcat{sl}{r+1}{k}}\right)^\text{ad}$ and $\fcat$ are the same, thus we have an equivalence.
\end{proof}

Finally, we study the invertible objects of the adjoint subcategory $\fcat$.

\begin{lemma}
We have
\[  \Inv(\ad( \cat{sl}{r+1}{k})_{\Rep(\Z{m'})})\cong    \Z{\frac{n'}{m'}},\]
where $n' := \gcd(r+1,k)$.
\end{lemma}
\begin{proof}
This proof is fairly similar to the proof of Lemma~\ref{lem:invfull}. The same idea shows that any invertible of $\fcat$ will be of the form $(  k\Lambda_i,1)$ where $i \in \Z{ \frac{r+1}{m'}  }$ and $ki \equiv 0 \pmod {r+1}$. This implies that $i$ has to be a multiple of $\frac{r+1}{n'}$. Thus the invertible objects of $\fcat$ are of the form $( k\Lambda_{ j\frac{r+1}{n'}     },1)$, where $j\in \Z{\frac{n'}{m'}}$.
\end{proof}


The invertible objects of $\dcat{sl}{r+1}{k}$ act transitively on the simple object $\Omega$ in all but one special case.
\changeX{
\begin{lemma}\label{lem:tran}
Suppose $(r,k,m) \nin  \{(1,4,1), (2,3,1), (3,2,1)\}$, and let $g\in  \Inv(\dcat{sl}{r+1}{k})$. Then 
\[ g \otimes \Omega \cong \Omega \implies g \cong \mathbf{1}.\]
\end{lemma}
\begin{proof}
By Lemma~\ref{lem:invfull} we have that $g \cong (k\Lambda_{\ell m''},1)$ for some $\ell \in \mathbb{Z}_{\frac{r+1}{mm''}}$. As $g\otimes \Omega \cong \Omega$, we get
\[    k\Lambda_{\ell m''}\otimes (\Lambda_1 + \Lambda_r)  \quad \text{ and } \quad  (\Lambda_1 + \Lambda_r) \]
live in the same orbit under the action of $\Z{m}$ in $\cat{sl}{r+1}{k}$. Thus there exists a $j$ such that
\begin{equation}\label{eq:brainhurts}      \Lambda_{-1 + \ell m''} + (k-2)\Lambda_{\ell m''} + \Lambda_{\ell m'' + 1} =     \Lambda_{-1 + j\frac{r+1}{m}} + (k-2)\Lambda_{j \frac{r+1}{m}} + \Lambda_{1 + j\frac{r+1}{m}}.\end{equation}

We first deal with the special case of $r=1$.  If $m = 2$, then we have that $\ell \equiv 0 \pmod {\frac{r+1}{mm''}}$, and so $g\cong \mathbf{1}$. Otherwise $m=1$, and Equation~\eqref{eq:brainhurts} becomes
\[     2\Lambda_{\ell  + 1} + (k-2)\Lambda_{\ell } = 2\Lambda_{1} + (k-2)\Lambda_{0} .    \]
Either $\ell=0$ and the desired result is immediate, or $\ell=1$, which forces $k=4$ (an excluded case in the statement of the lemma).

With the case of $r=1$ dealt with, we can assume that 
\[\Lambda_{-1 + \ell m''} \neq  \Lambda_{\ell m'' + 1}\quad \text{and } \quad  \Lambda_{-1 + j\frac{r+1}{m}}\neq  \Lambda_{1 + j\frac{r+1}{m}}.\]
We now break into cases depending on $k$.

If $k > 3$, then Equation~\eqref{eq:brainhurts} gives that $\ell m'' \equiv j \frac{r+1}{m} \pmod {r+1}$. Hence $\ell \equiv 0 \pmod {\frac{r+1}{mm''}}$ and so $g\cong \mathbf{1}$ as desired.

If $k = 3$, then Equation~\eqref{eq:brainhurts} gives that $\ell m'' \equiv \{-1 + j\frac{r+1}{m}, j\frac{r+1}{m},1 + j\frac{r+1}{m}\} \pmod {r+1}     \}$. If $\ell m'' \equiv j\frac{r+1}{m}\pmod {r+1}$, then $\ell \equiv 0 \pmod {\frac{r+1}{mm''}}$, and we have $g \cong \mathbf{1}$. If $\ell m'' \equiv \pm 1+ j\frac{r+1}{m}\pmod {r+1}$, then Equation~\eqref{eq:brainhurts} gives $3\equiv 0 \pmod {r+1}$, and so $r=2$. Either we have $m=1$, in which case we are excluded by the statement of the lemma, or $m=3$, in which case we are excluded by Remark~\ref{rmk:special}.

If $k = 2$ then Equation~\eqref{eq:brainhurts} gives that either $\ell m'' \equiv j \frac{r+1}{m} \pmod {r+1}$ or $\ell m'' \equiv 2+ j \frac{r+1}{m}\equiv -2+ j \frac{r+1}{m} \pmod {r+1}$. In the first case, we have that $\ell \equiv 0 \pmod {\frac{r+1}{mm''}}$ and we are done. Inthe latter case, we have that $4\equiv 0 \pmod {r+1}$, and so $r=3$. We either have $m=2$, in which case we are excluded by Remark~\ref{rmk:special}, or $m=1$, in which case we are excluded by the statement of the lemma.
\end{proof}
}

We will make $\Omega$ the base point of our auto-equivalence computations, and distinguish auto-equivalences based on whether they fix or move this object.

\begin{dfn}
We say an auto-equivalence of $\dcat{sl}{r+1}{k}$ is \textit{non-exceptional} if it maps $\Omega$ to an image of $\Omega$ under simple currents. We will say the auto-equivalence is \textit{exceptional} if it is not non-exceptional.
\end{dfn}

We will see in the bulk of this paper the surprising result that only a finite number of auto-equivalences are exceptional, and that every non-exceptional auto-equivalence comes from either a simple current auto-equivalence, charge conjugation, or from the canonical $\Z{m}$-action.

While the definition of a non-exceptional auto-equivalence allows for the object $\Omega$ to be moved, the following lemma shows this is not the case.

\begin{lemma}
A non-exceptional auto-equivalence of $\dcat{sl}{r+1}{k}$ must fix $\Omega$.
\end{lemma}
\begin{proof}
Let $\cF$ a non-exceptional auto-equivalence of $\dcat{sl}{r+1}{k}$. Then there exists an invertible element $g$ of $\dcat{sl}{r+1}{k}$ such that $\cF(\Omega) \cong g \otimes \Omega$. As $g$ is of the form $(k\Lambda_{nm''},1)$ for $n \in \mathbb{N}$, we compute that
\[ g \otimes \Omega \cong ( (k-2)\Lambda_{nm''} +\Lambda_{nm''+1} + \Lambda_{nm''-1}    ,1).\]

The object $\Omega$ is self-dual, and so $g \otimes \Omega$ must be as well, thus $g \otimes \Omega \cong g^* \otimes \Omega$. \changeX{Assuming that we are not in the excluded cases of Lemma~\ref{lem:tran}, we can apply this lemma to obtain $g^{\otimes 2}\cong \mathbf{1}$, and so $g \cong (k\Lambda_{j\frac{r+1}{2m}}, 1)$ for $j \in \{0,1\}$. If we are in one of the three excluded cases, then a direct calculation shows that $g\otimes \Omega\cong \Omega$ for all $g \in \operatorname{Inv}\left(\dcat{sl}{r+1}{k}\right)$ and hence $\mathcal{F}$ fixes $\Omega$.

}

For the generic case of $r\geq 2$ and $k \geq 3$, we know that $\dim \hom (\Omega \otimes \Omega \to \Omega) = 2$, and thus $\dim \hom (g\otimes \Omega \otimes g\otimes \Omega \to g\otimes \Omega) = 2$. Using the braiding on the category, along with the fact that $g$ has order two, we see that $\dim \hom (\Omega \otimes \Omega \to g\otimes \Omega) = 2$. We explicitly compute the simple decomposition of $\Omega \otimes \Omega$ as
\[  \mathbf{1} \oplus 2\Omega \oplus  ( \Lambda_2 + 2\Lambda_r , 1) \oplus  ( \Lambda_2 + \Lambda_{r-1} , 1)  \oplus  ( 2\Lambda_1 + \Lambda_{r-1} , 1) \oplus  ( 2\Lambda_1 + 2\Lambda_{r} , 1).  \]
As $g\otimes \Omega =  ( (k-2)\Lambda_{j\frac{r+1}{2m}} +\Lambda_{j\frac{r+1}{2m} + 1} + \Lambda_{j\frac{r+1}{2m}-1}    ,1)$ must appear in this decomposition, we can immediately deduce that $j = 0$, i.e. $g$ must be the identity.

For the remaining cases, the proof is almost identical, except the decomposition of $\Omega \otimes \Omega$ is smaller, and in some cases the stabaliser subgroup of the simples in the decomposition is non-trivial, so the characters of the stabaliser groups must be changed.
\end{proof}

In light of the above result we make the following definition.

\begin{dfn}
We write $\BrAut( \dcat{sl}{r+1}{k} ; \Omega)$ for the group of braided auto-equivalences of $\dcat{sl}{r+1}{k}$ which fix $\Omega$, or equivalently, the group of non-exceptional auto-equivalences.
\end{dfn}

\subsection{Planar Algebras}

A key tool for the results of this paper are \textit{planar algebras}. Roughly speaking a planar algebra $\mathcal{P}$ is a collection of vector spaces $\{ \mathcal{P}_n : n \in \mathbb{N}\}$, along with a multi-linear action of planar tangles. The full definition can be found in \cite{math.QA/9909027}, and illuminating examples in \cite{MR2559686}.

We will be interested in planar algebras constructed from symmetrically self-dual objects in pivotal fusion categories. Let $X \in \cC$ be such an object. Then we can define a planar algebra $\mathcal{P}_X$ by
\[ (\mathcal{P}_X)_n := \Hom(\mathbf{1} \to X^{\otimes n}).\]
Supposing the object $X$ generated $\cC$, then we can recover $\cC$ by taking the idempotent completion of $\mathcal{P}_X$. Here the objects are idempotents in the algebras $(\mathcal{P}_X)_{2n}$ (where we have $n$ legs pointing up, and $n$ legs pointing down) with vertical stacking as the multiplication. The morphisms between two idempotents are elements of the planar algebra which intertwine the two idempotents. The tensor product is given by horizontal juxtaposition, and direct sums are added formally. Additional information on these two constructions can be found in \cite{ MR2559686}.

It is proven in \cite[Theorem A]{1607.06041} that the above bijection between planar algebras and symmetrically self-dual objects $X \in \cC$ is functorial. That is there is an isomorphism between automorphisms of the planar algebra $\mathcal{P}_X$, and pivotal auto-equivalences of the category $\cC$ which fix $X$.

\changeY{\subsection{Simple Current Auto-equivalences}

A useful class of auto-equivalences of $\ecat{sl}{r+1}{k}{m}$ are given by \textit{simple current auto-equivalences}. These are graded auto-equivalences which permute the simple objects by tensoring with certain invertible objects in $\ecat{sl}{r+1}{k}{m}$. The precise definition is as follows.

\begin{lemma}\cite[Lemma 2.4]{ABCG}
Let $\cC$ be a modular tensor category, and $g$ an invertible object of order $M$. Set $q$ equal to the unique integer (modulo $2M$) such that
\[    \sigma_{g,g} = e^{2\pi i \frac{q}{2M}}\id_{g\otimes g},\] (see \cite[Proposition 2.5.1]{MR1734419})
and choose $a \in \{0,1,\cdots , M-1\}$ such that
\[  1 + aq \quad \text{ is coprime to } M.\]
Then there exists a monoidal auto-equivalence $\cF_{g,a}$ of $\cC$ defined on objects by
\[   \cF_{g,a}(X) = g^{-an}\otimes X , \]
where $n$ is the unique integer (modulo $M$) such that $\sigma_{X,g}\sigma_{g,X} = e^{2\pi i \frac{n}{M}}\id_{g\otimes X}$. The monoidal auto-equivalence $\cF_{g,a}$ is braided if and only if
\[   a + \frac{a^2q}{2} \equiv 0 \pmod M.\]
\end{lemma}

As $\Omega$ is in the adjoint subcategory of $\ecat{sl}{r+1}{k}{m}$, we have that any simple current auto-equivalence fixes $\Omega$, and hence is non-exceptional.

}

\section{Non-exceptional auto-equivalences of $\dcat{sl}{r+1}{k}$}\label{sec:nonexcep}

In this section we will determine the braided auto-equivalences of $\dcat{sl}{r+1}{k}$ that fix the distinguished object $\Omega$. In terms of the notation introduced in this paper, we will determine the group $\BrAut( \dcat{sl}{r+1}{k}; \Omega)$. We show that non-exceptional auto-equivalences (in the formal definition of this paper) are non-exceptional (in the layman terms). That is, every non-exceptional braided auto-equivalence is either charge conjugation, simple current, or comes from the canonical $\Z{m}$-action on $\dcat{sl}{r+1}{k}$.

Let us outline the arguments of this section. To begin, we initially focus our attention on the distinguished subcategory $\fcat$. The subcategory $\fcat$ has two nice features that will assist with the results of this section. First is that it has trivial universal grading group, and hence has a unique pivotal structure, and second the category $\cat{sl}{r+1}{k}^\text{ad}_{\Rep(\Z{m'})}$ is generated by the distinguished object $\Omega$. Together these facts will allow us powerful planar algebra techniques to determine the non-exceptional symmetries.

With the above in mind, we give a presentation of the planar algebra $\mathcal{P}_\Omega$, i.e. the planar algebra generated by the object $\Omega \in \fcat$. To achieve this, we observe that $\mathcal{P}_{\Omega}$ contains the planar algebra $\mathcal{P}_{\Lambda_1 + \Lambda_r}$, i.e. the planar algebra generated by the object $\Lambda_1 + \Lambda_r\in   \cat{sl}{r+1}{k}^{\text{ad}}$. The planar algebra $\mathcal{P}_{\Lambda_1 + \Lambda_r}$ is well understood, and is known to be generated by two trivalent vertices. We can then find an additional generator in $\mathcal{P}_\Omega$, which together with the two trivalent vertices, generate all of $\mathcal{P}_\Omega$. The idea here is that the group $\Z{m'}$ is singly generated, which allows us to understand skein theory for de-equivariantisation in terms of the addition of one additional generator. With the generators of $\mathcal{P}_\Omega$ identified, we can then find relations that these generators satisfy.

\begin{rmk}
While it is not explicit in this paper, the techniques we have briefly described above (and will explain in detail in the remainder of this section), can be used to give skein theory for any de-equivariantisation by an abelian group.
\end{rmk}

With the presentation of the planar algebra $\mathcal{P}_{\Omega}$ in hand, we can use it to give an upper bound for the group of braided auto-equivalences of $\fcat$ which fix $\Omega$. We find that there are at most $2m'$ of these auto-equivalences, which compose to form a group isomorphic to $D_{m'}$. Further, we explicitly identify how these potential auto-equivalences act on the simples of $\fcat$. We then construct these $2m'$ potential auto-equivalences by the charge conjugation auto-equivalence, which gives us a $\Z{2}$ subgroup, and by the canonical $\Z{m'}$-action on $\fcat$ which comes from de-equivariantisation.

To obtain the auto-equivalences of $\dcat{sl}{r+1}{k}$ which fix $\Omega$, we appeal to the techniques developed in \cite{MR4192836}. These techniques allow us to give an upper bound for $\BrAut( \dcat{sl}{r+1}{k}; \Omega)$ in terms of $\BrAut( \fcat; \Omega)$ and some cohomogical data. While there is no reason that this bound should be sharp (the techniques involve verifying that certain obstructions vanish in order to show that auto-equivalences lift) we are able to show that the theoretical upper bound is realised by simple current auto-equivalences.

All together we prove the following theorem.
\begin{theorem}\label{thm:mainnon}
Let $r,k \in  \mathbb{N}$, and $m$ a divisor of $r+1$ such that $m^2 \divides k(r+1)$. Set $m' = \gcd(m,k)$ and $m'' = \frac{m}{m'}$. Then we have the following isomorphism of groups
\[  \BrAut(  \dcat{sl}{r+1}{k}; \Omega) \cong      \begin{cases}
\{e\} \text{ if $k=2$ and $r=1$}\\
\Z{m'} \times  \Z{2}^{p+t} \text{ if $k=2$ or $r=1$} \\
D_{m'} \times \Z{2}^{p+t} \text{ otherwise }
\end{cases}\]
where
\begin{itemize}
\item $p$ is the number of distinct odd primes dividing $\frac{r+1}{mm''}$ but not $\frac{k}{m'}$, and
\item $t = \begin{cases} 0 \text{ if $\frac{r+1}{mm''}$ is odd, or if $\frac{k}{m'} \equiv 0 \pmod 4$, or if both $\frac{k}{m'}$ is odd, and $\frac{r+1}{mm''} \equiv 2 \pmod 4$} \\
 1 \text{ otherwise.}
\end{cases}$
\end{itemize}
\end{theorem}

With the high-level arguments in mind, let us begin with the details of proving the above theorem.

Consider the planar algebra $ \mathcal{P}_{ \Omega}$. As $\Omega$ generates, and their exists a map $\Omega \otimes \Omega \to \Omega$, we have that $\fcat$ has trivial universal grading group, and thus also has a unique pivotal structure. Therefore we have that $\BrAut( \fcat, \Omega)$ is isomorphic to the group of braided planar algebra automorphisms of $ \mathcal{P}_{ \Omega}$. Our goal is thus to specify as much of the structure of this planar algebra $ \mathcal{P}_{ \Omega}$ as possible in order to understand its auto-equivalence group.

As the free module functor $\cat{sl}{r+1}{k}^{\text{ad}} \to \fcat$ is dominant, and maps $ \Lambda_1 + \Lambda_r$ to $\Omega$, we obtain a planar algebra embedding
\[ \mathcal{P}_{ \Lambda_1 + \Lambda_r } \to \mathcal{P}_{ \Omega}.\]

The planar algebra $\mathcal{P}_{ \Lambda_1 + \Lambda_r }$ is well understood \cite{MR4002229,ABCG}. It is generated by two trivalent vertices satisfying the Thurston relations (see \cite[Lemma 3.2]{MR4002229}). Hence the planar algebra $\mathcal{P}_{ \Omega}$ also contains two trivalent vertices
\[   \raisebox{-.5\height}{ \includegraphics[scale = .06]{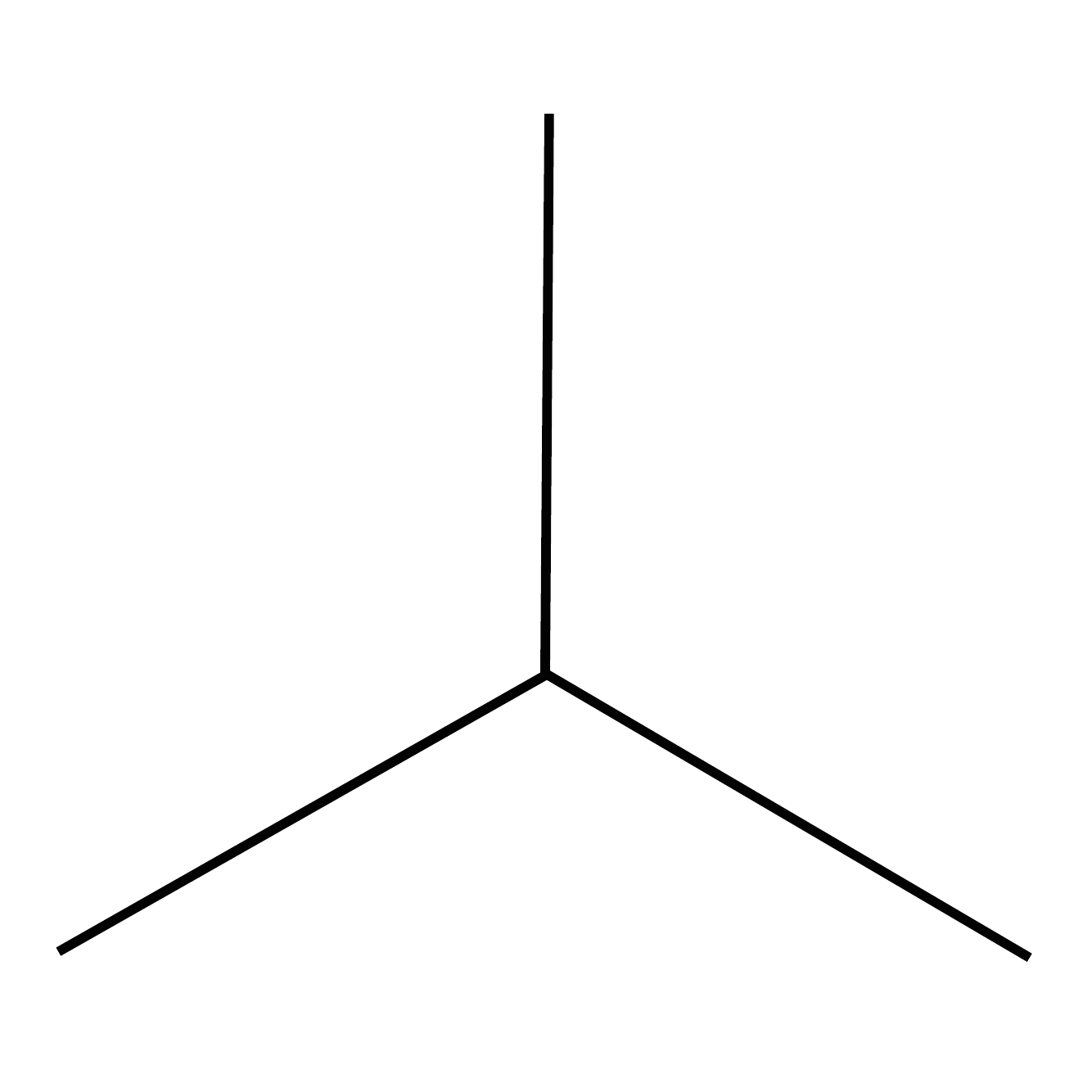}}  \quad \text{ and } \quad \raisebox{-.5\height}{ \includegraphics[scale = .08]{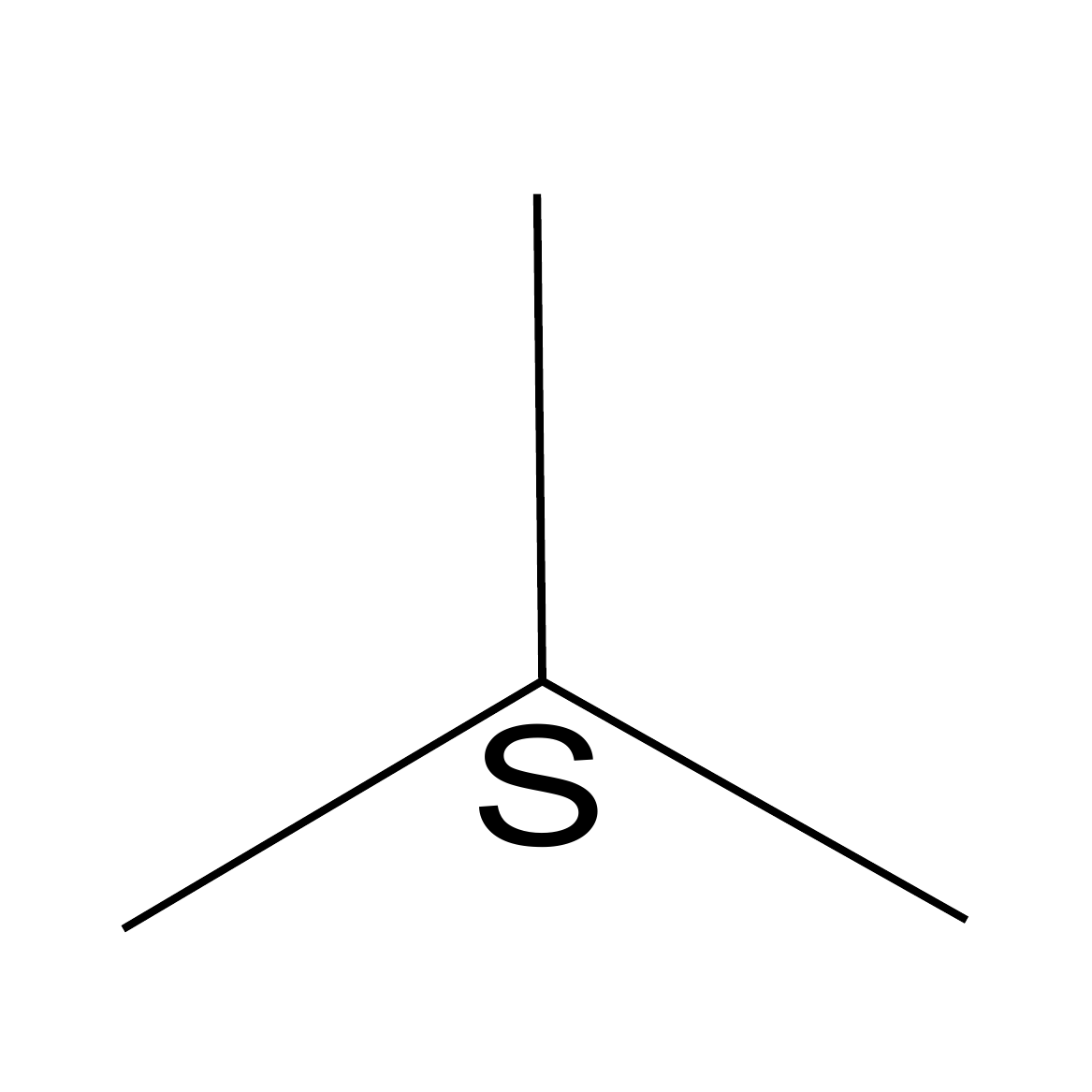}}\]
 satisfying these same Thurston relations. However, there are going to be additional generators in this planar algebra. These additional generators come from the de-equivariantization by $\Rep(\Z{m'})$.

\begin{rmk}
For the remainder of this section we will identify $ \cat{sl}{r+1}{k}^{\text{ad}}$ as the idempotent completion of the planar algebra $\mathcal{P}_{ \Lambda_1 + \Lambda_r }$, and $\fcat$ as the idempotent completion of the planar algebra $\mathcal{P}_{ \Omega}$. This means that we regard simple objects of these categories as minimal idempotents of the planar algebras, and morphisms as elements of the planar algebra which commute with the idempotents.
\end{rmk}

Let us write $p_{k\Lambda_{\frac{r+1}{m'}}}$ for the minimal idempotent of $\cat{sl}{r+1}{k}^{\text{ad}}$ corresponding to the simple object $k\Lambda_{\frac{r+1}{m'}}$. From the inclusion of planar algebras $\mathcal{P}_{ \Lambda_1 + \Lambda_r }   \to \mathcal{P}_{ \Omega }$, we have that this idempotent $p_{k\Lambda_{\frac{r+1}{m'}}}$ also exists in $\mathcal{P}_{ \Omega}$.

The free module functor $\mathcal{F}_{\Z{m'}}: \cat{sl}{r+1}{k}^\text{ad} \to\fcat $ sends $k\Lambda_{\frac{r+1}{m'}}$ to the tensor unit. Therefore in the planar algebra $\mathcal{P}_{ \Omega}$, the trivial idempotent and $p_{k\Lambda_{\frac{r+1}{m'}}}$ are isomorphic. Thus there exists an invertible element $S \in \mathcal{P}_{ \Omega}$ (which we draw as a circle to differentiate it from the other planar algebra elements) satisfying
\[  \raisebox{-.5\height}{ \includegraphics[scale = .6]{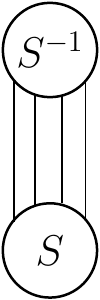}}=1 \quad \text{ and} \quad   \raisebox{-.5\height}{ \includegraphics[scale = .6]{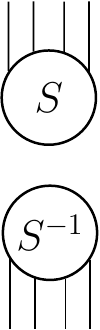}} =   \raisebox{-.5\height}{ \includegraphics[scale = .5]{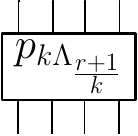}}.\]

The element $S$ lives in the $n$-box space of $\mathcal{P}_{ \Omega}$, where $n$ is the smallest $n$ such that $k\Lambda_{\frac{r+1}{m'}}$ appears in the decomposition of $(\Lambda_1 + \Lambda_r)^{\otimes n }$

We claim that $\mathcal{P}_{ \Omega}$ is generated by the two trivalent vertices, along with the new element $S$.

\begin{lemma}
We have that $\mathcal{P}_{ \Omega}$ is generated by the two Thurston trivalent vertices, and the element $S$.
\end{lemma}

\begin{proof}
Let $\mathcal{P}_{S}$ be the sub-planar algebra of $\mathcal{P}_{ \Omega}$ generated by these three elements, and $\cC_{S}$ the corresponding category. Then we have a chain of embeddings
\[ \mathcal{P}_{ \Lambda_1 + \Lambda_r}   \to   \mathcal{P}_{S} \to \mathcal{P}_{ \Omega}.\]
This gives us dominant monoidal functors
\begin{align*}
\mathcal{F}_1 &:   \cat{sl}{r+1}{k}^\text{ad} \to  \cC_{S}\\
\mathcal{F}_2 &:   \cC_{S} \to \fcat,
\end{align*}
and their adjoints
\begin{align*}
\mathcal{F}^*_1 &:   \cC_{S}\to \cat{sl}{r+1}{k}^{\text{ad}}\\
\mathcal{F}^*_2 &:  \fcat \to  \cC_{S} .
\end{align*}
From \cite{MR3161401}, we have that $\mathcal{F}^*_1(\mathbf{1}_{\cC_{S}})$ is a commutative central algebra object, and that $\cC_S$ is equivalent to the category of  $\mathcal{F}^*_1(\mathbf{1}_{\cC_{S}})$-modules in $ \cat{sl}{r+1}{k}^\text{ad} $.

As these dominant functors $\mathcal{F}_1$ and $\mathcal{F}_2$ are just the inclusions of idempotents, we have that the composition of these two dominant functors is equal on the nose to the dominant functor $ \cat{sl}{r+1}{k}^\text{ad} \to \fcat$ induced by the planar algebra inclusion
\[\mathcal{P}_{ \Lambda_1 + \Lambda_r}  \to \mathcal{P}_{ \Omega} .\]
This induced functor $ \cat{sl}{r+1}{k}^\text{ad} \to \fcat$ is precisely the free module functor $\mathcal{F}_{\Z{m'}}$. Hence we have that
\[  \mathcal{F}_2 \circ \mathcal{F}_1 =  \mathcal{F}_{\Z{m'}},   \]
which implies that
\[  \mathcal{F}^*_1 \circ \mathcal{F}^*_2 =  \mathcal{F}^*_{\Z{m'}} .  \]
From this fact we see
\[  \mathcal{F}^*_1 ( \mathbf{1}_{\mathcal{C}_S}  ) \subseteq \mathcal{F}^*_1\circ \mathcal{F}^*_2 \left( \mathbf{1}_{\fcat} \right ) = \mathcal{F}_{\Z{m'}}^*\left(    \mathbf{1}_{\fcat}\right)  = \Fun(\Z{m'}), \]
as a central commutative algebra in $\cat{sl}{r+1}{k}^\text{ad}$. In particular we get that $ \mathcal{F}^*_1 ( \mathbf{1}_{\mathcal{C}_S}  ) \cong \Fun(\Z{\ell})$ where $\ell \divides m'$. As $\Fun(\Z{m'})$ is the central commutative algebra object in $ \cat{sl}{r+1}{k}^\text{ad}$ corresponding to the de-equivariantisation by the $\Rep(\Z{m'})$ subcategory, the central structure is given by the braiding of $\cat{sl}{r+1}{k}^\text{ad}$. Hence the central structure on $\Fun(\Z{\ell})$ is also given by the braiding. This gives that $\cC_S$ is a de-equivariantisation of  $\cat{sl}{r+1}{k}^\text{ad} $ by $\Rep(\Z{\ell})$, i.e.
\[     \cC_{S} \simeq \cat{sl}{r+1}{k}^\text{ad}_{\operatorname{Rep}(\Z{\ell})}.\]

In $\cC_{S}$ we know that $S$ gives an isomorphism from $\mathbf{1} \to p_{ k\Lambda_{\frac{r+1}{m'}}}$ which implies that $m' \divides \ell$. Thus
\[ \cC_S \simeq \cat{sl}{r+1}{k}^\text{ad}_{\operatorname{Rep}(\Z{m'})}\]
 which gives the desired isomorphism of planar algebras
\[ \mathcal{P}_{S} \cong \mathcal{P}_{ \Omega} .\]

  \end{proof}

In order to study the planar algebra automorphisms of $\mathcal{P}_{ \Omega}$ we need to study the element $S$ further, and deduce further relations that it satisfies.

\begin{rmk}
To simplify notation, we will now draw multiple strands of a planar algebra as a single strand in our graphical diagrams. It will be clear from context how many strands are meant by the diagram.
\end{rmk}

In the category $\cat{sl}{r+1}{k}^\text{ad}$ we have that $k\Lambda_{\frac{r+1}{m'}}^{\otimes m'} \cong \mathbf{1}$. Thus the object $k\Lambda_{\frac{r+1}{m'}}$ generates a subcategory with the fusion rules of $\Z{m'}$. We are given that this subcategory is Tannakian (as it is the subcategory we are de-equivariantating by), so it is braided equivalent to $\Rep(\Z{m'})$. \changeX{For $n\in \Z{m'}$, let $p_{k\Lambda_{r\frac{r+1}{m'}}}\in \mathcal{P}_{\Lambda_1 + \Lambda_r}$ be the unique (by the fusion rules) projection onto $k\Lambda_{n\frac{r+1}{m'}}$ appearing in the smallest possible box-space of $ \mathcal{P}_{\Lambda_1 + \Lambda_r}$. Note that as $\Lambda_1 + \Lambda_r$ is self-dual, we get that $k\Lambda_{n\frac{r+1}{m'}}$ and $k\Lambda_{-n\frac{r+1}{m'}}$ live in the same box space.} We can choose a system of trivalent vertices
\[ \changeX{ t_{n,p}=   \raisebox{-.5\height}{ \includegraphics[scale = .6]{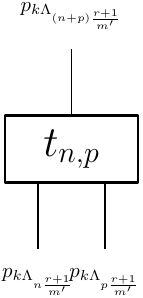}} \quad :  p_{ k\Lambda_{n\frac{r+1}{m'}}}  \otimes p_{k\Lambda_{p\frac{r+1}{m'}}} \to p_{k\Lambda_{(n+p)\frac{r+1}{m'}}}} \]
in $\cat{sl}{r+1}{k}^\text{ad}$ with trivial 6-j symbols, and such that the charge conjugation auto-equivalence $\changeX{ k\Lambda_{n\frac{r+1}{m'}} \mapsto k\Lambda_{-n\frac{r+1}{m'}}}$ maps $t_{n,p}\mapsto t_{-n,-p}$.

We can build an isomorphism
\[ j :=t_{1,1}t_{2,1} \cdots t_{m'-1,1}  =  \raisebox{-.5\height}{ \includegraphics[scale = .6]{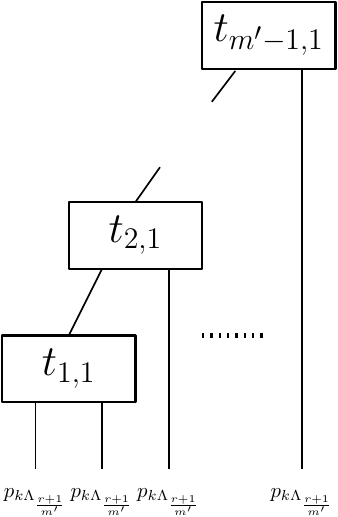}} \quad  : p_{k\Lambda_{\frac{r+1}{m'}}^{\otimes m'}} \to \mathbf{1}.\]
Hence we have that $j$ is a map from the idempotent $p_{k\Lambda_{\frac{r+1}{m'}}}^{\otimes m'}$ to the trivial idempotent in the planar algebra $\mathcal{P}_{\Lambda_1 + \Lambda_r}$. In the planar algebra $\mathcal{P}_{\Omega}$ we have that $S^{\otimes m'}$ is an isomorphism from the idempotent $p_{k\Lambda_{\frac{r+1}{m'}}}^{\otimes m'}$ to the trivial idempotent. Thus we have that
\[    S^{\otimes m'} \circ j\]
lives in the 0-box space of $\mathcal{P}_\Omega$ and is non-zero. This allows us to normalise $S$ so that we get the relation
\begin{equation}\label{eq:rel1}
  \raisebox{-.5\height}{ \includegraphics[scale = .6]{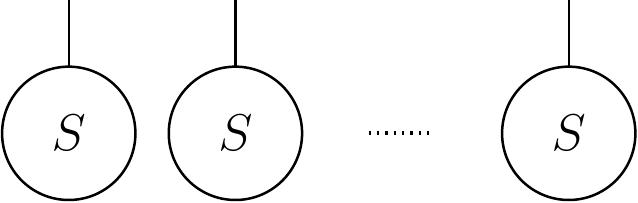}} \quad = \quad   \raisebox{-.5\height}{ \includegraphics[scale = .6]{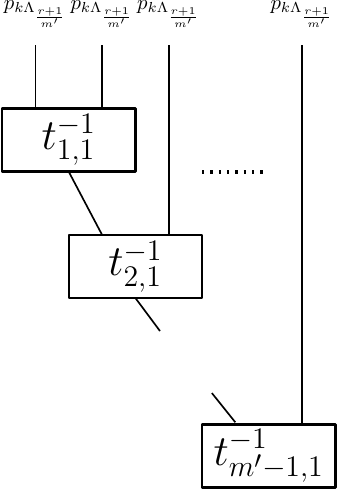}},
\end{equation}
in the planar algebra $\mathcal{P}_{\Omega}$.

This explicit presentation of the planar algebra $\mathcal{P}_{\Omega}$ is sufficient to compute the minimal idempotents up to equivalence, and thus the simple objects of $\fcat$.

\begin{lemma}\label{lem:simps}
The simple objects of $\fcat$ are parametrised (up to isomorphism) by
\[   (X, \chi_X)\]
where $X$ is a simple object of $\cat{sl}{r+1}{k}^\text{ad}$ (up to action by $k\Lambda_{\frac{r+1}{m'}}$) and $\chi_X$ is a character of the group $\operatorname{Stab}_{\Z{m'}}(X)$. The quantum dimension of $(X, \chi_X)$ is equal to $ \frac{\dim(X)}{  |\operatorname{Stab}_{\Z{m'}}(X)|  }$.
\end{lemma}

\begin{proof}
The free module functor $\mathcal{F}_{\Z{m'}}$ is dominant, therefore every simple object of $\fcat$ is a sub-object of $\mathcal{F}_{\Z{m'}}(X)$ for some $X \in \cat{sl}{r+1}{k}^\text{ad}$. Let $p_X$ be the minimal projection in the planar algebra $ \mathcal{P}_{\Lambda_1 + \Lambda_r}$ corresponding to $X$. As $\fcat$ is idempotent complete, each simple sub-object of $X$ will correspond (up to isomorphism) to a minimal sub-idempotent of $p_X$.

\changeX{Assume that $\operatorname{Stab}_{\Z{m'}}(X) \cong \Z{d}$, then} there exists an isomorphism $f_X : p_X \otimes    p_{k\Lambda_{\frac{r+1}{m'}}^{\otimes \frac{m'}{d}}}\changeX{\to p_X}  $ in $\cat{sl}{r+1}{k}^{\text{ad}}$. For each $n\in \Z{d}$ we define isomorphisms $r_{X, n} : p_X \to p_X$ in $\fcat$ by
\[     r_{X, n} :=  f_X^n \circ S^{\otimes n \frac{m'}{d}} =  \changeX{ \raisebox{-.5\height}{ \includegraphics[scale = .6]{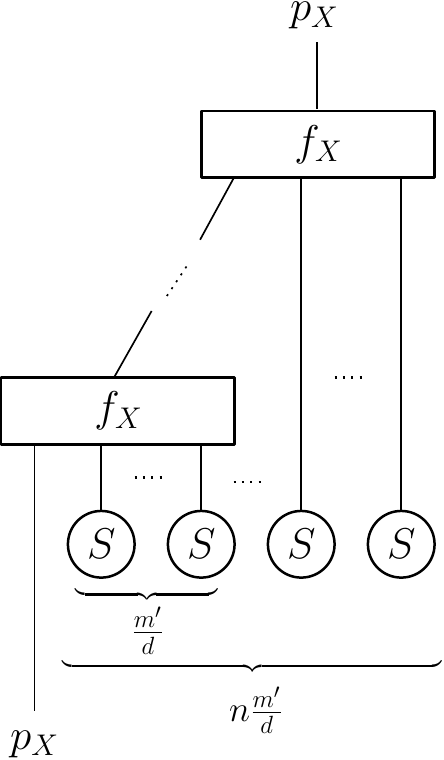}}}.  \]

By design we have that $r_{X, n}r_{X, n'} = r_{X, n + n'}$. Furthermore, by relation~\ref{eq:rel1} we have that $r_{X, d} : p_X \to p_X$ lives in $\mathcal{P}_{\Lambda_1 + \Lambda_r}$. As $p_X$ is simple in $\cat{sl}{r+1}{k}^\text{ad}$, we have that $r_{X, d}$ must be a scaler multiple of $p_X$. We normalise our choice of the isomorphism $f_X$ to ensure that $r_{X, d} = p_X$. Thus we have that $\operatorname{End}( p_X)$ in $\fcat$ is isomorphic to the group algebra $\mathbb{C}[ \Z{d}]$. It is a classical result that the minimal idempotents are indexed by characters $\chi$ of $\Z{d}$ with
\[  p_\chi = \frac{1}{ |\operatorname{Stab}_{\Z{m'}}(X)|}   \sum _{n\in \Z{d}} \chi(n) r_{X,n}.   \]

The quantum dimension of the minimal idempotent $p_\chi$ is given by the trace. Note that the trace of $r_{X,n}$ is $0$, unless $n = 0$, as otherwise we could build a non-trivial morphism
\[k\Lambda_{\frac{r+1}{m'}}^{\otimes \frac{nm'}{d}} \to \mathbf{1}.\]
If $n=0$, then the trace of $r_{X,n}$ is the quantum dimension of $X$. Hence the trace of $p_\chi$ is equal to the quantum dimension of $X$ divided by $ |\operatorname{Stab}_{\Z{m'}}(X)|$.
\end{proof}

   \begin{rmk}\label{rmk:iso}
   For ease of notation, let us fix isomorphisms $\Z{N} \to \widehat{\Z{N}}$ by
   \[  n \mapsto  \chi_n :=   \changeX{a} \mapsto e^{2\pi i \frac{n\changeX{a}}{ N}}.        \]

   \end{rmk}

We can now determine an upper bound for the group $\Aut( \mathcal{P}_{\Omega} )$, and hence also for the group $\BrAut(\fcat ; \Omega)$ .

\begin{lemma}\label{lem:boundo}
We have that
\[     \BrAut( \fcat ; \Omega) \subseteq D_{m'} ,\]
with generators
\[    (X, \chi_n) \mapsto (X, \chi_{n+1})       \]
and
\[     (X,  \chi_n) \mapsto (X^*, \chi_{-n})    .  \]
\end{lemma}

\begin{proof}
Let $\phi \in  \Aut( \mathcal{P}_{\Omega} )$ be a braided automorphism. Then $\phi$ is determined by where it sends the three generators. Recall, we have two trivalent vertices satisfying the Thurston relations, and the generator $S$ which lives in the $n$-box space, where $n$ is the smallest $n$ such that $k\Lambda_{\frac{r+1}{m'}}$ appears in the decomposition of $(\Lambda_1 + \Lambda_r)^{\otimes n }$. By explicitly expanding $(\Lambda_1 + \Lambda_r)^{\otimes 3 }$ we can see that $k\Lambda_{\frac{r+1}{m'}}$ appears as a summand only in the case $\cat{sl}{2}{4}^\text{ad}_{\operatorname{Rep}(\Z{2})}$, $\cat{sl}{3}{3}^\text{ad}_{\operatorname{Rep}(\Z{3})}$, and $\cat{sl}{4}{2}^\text{ad}_{\operatorname{Rep}(\Z{2})}$. These cases have already been excluded and dealt with previously in the paper.

Let us deal with the remaining cases. As $S$ does not live in the three box space, we know that there are scalers $c_1, c_2, c_3, c_4 \in \mathbb{C}$ such that
\begin{align*}
  \phi\left( \raisebox{-.5\height}{ \includegraphics[scale = .06]{UPTLTRIV}} \right) &= c_1 \raisebox{-.5\height}{ \includegraphics[scale = .06]{UPTLTRIV}}  + c_2\raisebox{-.5\height}{ \includegraphics[scale = .08]{UPSV}}\\
    \phi\left( \raisebox{-.5\height}{ \includegraphics[scale = .08]{UPSV}} \right) &= c_3 \raisebox{-.5\height}{ \includegraphics[scale = .06]{UPTLTRIV}}  + c_4\raisebox{-.5\height}{ \includegraphics[scale = .08]{UPSV}}.
 \end{align*}

The coefficients $c_1,c_2,c_3,c_4$ for which $\phi$ preserve the Thurston relations are solved for in \cite[Lemma 3.1]{ABCG}. With the condition that $\phi$ is braided, there are two solutions, which we denote $\phi_{\text{id}}$ and $\phi_{\text{cc}}$. These planar algebra automorphisms on the sub-planar algebra $\mathcal{P}_{\Lambda_1 + \Lambda_r } $  are explicitly identified in the cited paper, where it is found that $\phi_{\text{cc}}$ corresponds to the charge conjugation auto-equivalence of $\cat{sl}{r+1}{k}^{\text{ad}}$.

Now the charge-conjugation auto-equivalence maps $k\Lambda_{\pm \frac{r+1}{m'}} \mapsto k\Lambda_{\mp \frac{r+1}{m'}}$, thus we have the following in the planar algebra $\mathcal{P}_{\Lambda_1 + \Lambda_r} $:
\[  \phi_{\text{id}} \left( p_{k\Lambda_{\pm \frac{r+1}{m'}}}\right)= p_{k\Lambda_{\pm \frac{r+1}{m'}}} \]
and
\[  \phi_{\text{cc}} \left( p_{k\Lambda_{\pm \frac{r+1}{m'}}}\right)= p_{k\Lambda_{\mp \frac{r+1}{m'}}} .\]
As the planar algebra $\mathcal{P}_{\Lambda_1 + \Lambda_r} $ canonically embeds in $\mathcal{P}_{\Omega} $, we also have these relations in the larger planar algebra.

To see when these auto-equivalences $\phi_{\text{id}}$ and $\phi_{\text{cc}}$ extend to the full planar algebra $\mathcal{P}_{\Omega}$ we must determine if (and how) these automorphisms act on the generators $S$.

Let us define isomorphisms in $\mathcal{P}_{\Omega}$ by
\[  S_n :=    S^{\otimes n} \circ t_{1,1}t_{2,1} \cdots t_{n-1,1}  = \changeX{ \raisebox{-.5\height}{ \includegraphics[scale = .6]{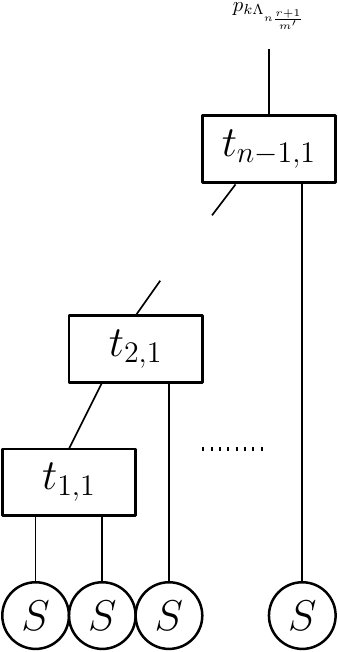}}  \quad : \mathbf{1} \to p_{k\Lambda_{n \frac{r+1}{m'}}}}.\]
Note that trivially we have $S_1 = S$, and by relation \eqref{eq:rel1} we have that $S_{m'} = 1$.

To see when $\phi_{\text{id}}$ extends to $\mathcal{P}_{\Omega}$, observe that $\phi_{\text{id}}( S )$ is an isomorphism from $\mathbf{1} \to p_{k\Lambda_{\frac{r+1}{m'}}}$. As this morphism space is 1-dimensional, we must have that $\phi_{\text{id}}( S ) = \beta S$ for some non-zero scaler $\beta \in \mathbb{C}$. Applying the potential automorphism to relation~\eqref{eq:rel1} gives that $\beta$ must be an $m'$-th root of unity.

To see when $\phi_{\text{cc}}$ extends to $\mathcal{P}_{\Omega}$, observe that $\phi_{\text{cc}}( S )$ is an isomorphism from $\mathbf{1} \to p_{k\Lambda_{-\frac{r+1}{m}}}$. This implies that $\phi_{\text{cc}}( S ) = \hat{\beta} S_{m'-1}$ for some non-zero scaler $\hat{\beta} \in \mathbb{C}$. We apply this potential automorphism to relation~\eqref{eq:rel1} to obtain
\changeX{\[  \hat{\beta}^{m'}   \raisebox{-.5\height}{ \includegraphics[scale = .6]{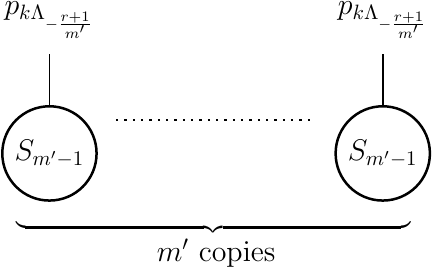}} =  \raisebox{-.5\height}{ \includegraphics[scale = .6]{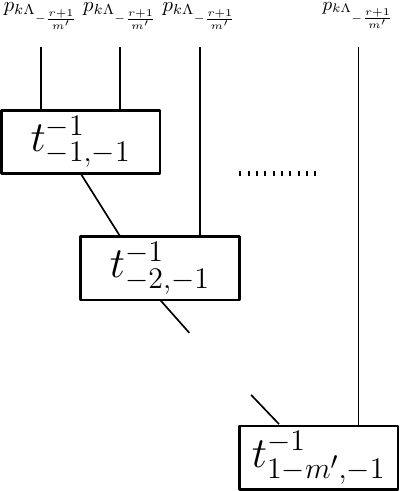}}. \]}
From this equation we expand out the $S_{m'-1}$ terms to obtain an equation with $(m'-1)m'$ of $S$ terms. We then apply relation~\eqref{eq:rel1} to get an equation purely in terms of the trivalent vertices $t$. From here we then use that the trivalent vertices $t$ have trivial 6-j symbols to obtain $\hat{\beta}^{m'} =1$. Thus $\hat{\beta}$ must be an $m'$-root of unity.

With the explicit presentation of how the $2m'$ potential automorphisms act on the generator $S$, it is straight forward to determine that if these automorphisms existed, then they would form a group isomorphic to $D_{m'}$. \changeX{ Note that the two automorphisms corresponding to $\beta = e^{2 \pi i \frac{1}{    m'}}$, and $\hat{\beta} = 1$ are generators for the entire automorphism group.}

We now determine how these $D_{m'}$ worth of potential automorphisms would act on the simple objects of $\fcat$. \changeX{Let $(X, \chi_n)$ be a simple object of $\fcat$, where $\operatorname{Stab}_{\mathbb{Z}_{m'}}(X) \cong \mathbb{Z}_d$ for some $d\mid m'$, and $\chi_n \in \hat{\mathbb{Z}_{d}}$ for some $n\in \mathbb{Z}_d$ (using the isomorphism of Remark~\ref{rmk:iso}).}

 \changeX{For the planar algebra automorphisms sending $S$ to $\beta S$ we pick the generator $\beta = e^{2 \pi i \frac{1}{    m'}}$ to study.We compute that $r_{X,j} \mapsto \beta^{j \frac{m'}{d}}r_{X,j} =e^{2 \pi i \frac{j}{d}}r_{X,j}$. Therefore under this planar algebra automorphism, we have
\[ p_{(X,\chi_n)} \mapsto \frac{1}{d} \sum_{j\in \mathbb{Z}_d} \chi_n(j)e^{2 \pi i \frac{j}{d}}r_{X,j} =\frac{1}{d} \sum_{j\in \mathbb{Z}_d} e^{2 \pi i \frac{nj}{d}}e^{2 \pi i \frac{j}{d}}r_{X,j} = p_{(X,\chi_{n+1})}.\]
Thus the auto-equivalence of $\fcat$ corresponding to the planar algebra automorphism for $\beta= e^{2 \pi i \frac{1}{    m'}}$ maps
\[     (X, \chi_n ) \mapsto (X^*, \chi_{-n}). \]
}

For the planar algebra automorphism sending $S \mapsto \hat{\beta} S_{m'-1}$ we have to work a little harder to determine where it sends the simple object $\changeX{ (X, \chi_n)}$. \changeX{We pick the generator $\hat{\beta} = 1$ to study}. Recall that this planar algebra automorphism restricts to $\phi_{\text{cc}}$ on the sub-planar algebra $\mathcal{P}_{\Lambda_1 + \Lambda_r}$. Thus we know how this automorphism acts on the trivalent vertices $t$, so we can compute that
\[  f_X \mapsto \gamma f_{X^*}^{m'-1} \circ \changeX{( t_{1,1}^{-1} t_{2,1}^{-1}\cdots t_{m'-2,1}^{-1})^{\otimes d}}\]
for some $\gamma \in \mathbb{C}$. By simultaneously rescaling the trivalent vertices $t$, we can ensure that $\gamma = 1$. With this information we compute that
\[  r_{X,j} \mapsto r_{X^*,-j},\]
and hence
\[  p_{\changeX{(X,\chi_n)}} \mapsto \changeX{  \frac{1}{d}\sum _{j\in \Z{d}} \chi_n(j) r_{X^*,-j}=  \frac{1}{d}\sum _{j\in \Z{d}} \chi_n(-j) r_{X^*,j}= \frac{1}{d}\sum _{j\in \Z{d}} \chi_{-n}(j) r_{X^*,j}= p_{(X^*,\chi_{-n})}.} \]
\changeX{Thus the auto-equivalence of $\fcat$ corresponding to the planar algebra automorphism for $\hat{\beta}=1$ maps
\[     (X, \chi_n ) \mapsto (X^*, \chi_{-n}). \]}

\end{proof}

The above lemma gives an upper bound on the braided auto-equivalence group (which fix $\Omega$) for the category $\fcat$. In theory we could determine a complete set of relations for the planar algebra $\mathcal{P}_\Omega$, and verify that the auto-equivalences exist by checking that they preserve all relations. However this requires additional work which is beyond the scope of this paper. Instead we construct $2m'$ worth of braided auto-equivalences of the category $\fcat$ directly, realising the upper bound.

\begin{lemma}
We have
\[      \BrAut( \fcat ; \Omega) \cong \begin{cases}
								\Z{m'}  \text{ if $k =2$ or $r=1$ } \\
									          D_{m'} \text{ otherwise}
									         \end{cases}.\]
\end{lemma}

\begin{proof}
Let us begin by constructing the $\Z{m'}$ worth of braided auto-equivalences. Via construction, we have that $\Z{m'}$ acts on $\fcat$ via the map
\[   (X, \chi_n) \mapsto (X, \chi_{n+1}).\]
To obtain the full $\Z{m}$ worth of auto-equivalences we need to show this action is faithful. This is equivalent to finding an object $X \in  \cat{sl}{r+1}{k}^\text{ad}$ with $\operatorname{Stab}_{\Z{m'}}(X) = \Z{m'}$. \changeX{If $m'$ is odd, then the object 
\[X = \frac{k}{m'}\sum_{i=1}^{m'} \Lambda_{i \frac{r+1}{m'}}\]
satisfies $X \otimes k\Lambda_{\frac{r+1}{m'}} = X$, and lives in $\cat{sl}{r+1}{k}^\text{ad}$ as 
\[\frac{k}{m'}\sum_{i=1}^{m'} i \frac{r+1}{m'}=(r+1) \frac{k(m'+1)}{2m'}\equiv 0 \pmod{r+1}.\] 

When $m'$ is even, the above object does not live in $\cat{sl}{r+1}{k}^\text{ad}$. To choose a suitable object, we observe that $r+1$ is even in this setting (as $m\mid r+1$), and so $2m^2 \mid (r+1)k$. In particular $\frac{(r+1)k}{2(m')^2}$ is an integer. For this case we pick the object
\[  X =  \left(\frac{k}{m'}-1\right) \sum_{i=1}^{m'} \Lambda_{i \frac{r+1}{m'}} +\sum_{i=1}^{m'} \Lambda_{i \frac{r+1}{m'} + \frac{(r+1)k}{2(m')^2} }   . \]
This object satisfies $X \otimes k\Lambda_{\frac{r+1}{m'}} = X$ as desired, and lives in $\cat{sl}{r+1}{k}^\text{ad}$ as 
\[\left(\frac{k}{m'}-1\right)\sum_{i=1}^{m'} i \frac{r+1}{m'}+\sum_{i=1}^{m'} i \frac{r+1}{m'} +\frac{(r+1)k}{2(m')^2} = (r+1)\left( \frac{k(m'+1)}{2m'} +\frac{k}{2m'}  \right)  = (r+1)\left( \frac{k(m'+2)}{2m'} \right)\equiv 0 \pmod {r+1}.  \]
In either case, we have an object $X \in  \cat{sl}{r+1}{k}^\text{ad}$ with $\operatorname{Stab}_{\Z{m'}}(X) = \Z{m'}$ as desired.
}

To construct the remaining auto-equivalences, we observe that the charge-conjugation auto-equivalences exist for $ \cat{sl}{r+1}{k}^\text{ad}$ when $k \geq 3$ and $r \geq 2$ \cite{ABCG}. These auto-equivalences preserve the $\operatorname{Rep}(\Z{m'})$ subcategory, and hence descend to auto-equivalences of $\fcat$.

To finish the proof, we must show that the charge conjugation auto-equivalence never coincides with the $\Z{m'}$ action. Thus we have to find an object $X \in \cat{sl}{r+1}{k}^\text{ad}$ such that $X^*$ is not in the orbit of the action of $k\Lambda_{\frac{r+1}{m'}}$. This object is given by
\[  X = \Lambda_1 + 2\Lambda_{r-1}.\]
This satisfies the required properties when $r \geq 3$ and $k\geq 3$.

If $r = 2$ and $k\geq 3$, then we can use the object
\[ X = 3\Lambda_1.\]

If $k = 2$ or $r=1$, then $m' \in \{1,2\}$ and the result is given in \cite[Theorem 1.2]{MR3808052}.
\end{proof}

Now that we understand the auto-equivalences of the subcategory $\fcat$ which fix $\Omega$, we can leverage this to determine the auto-equivalences of the full category $\dcat{sl}{r+1}{k}$. The idea here is to use the fact that $\dcat{sl}{r+1}{k}$ is a $\Z{\frac{r+1}{m\changeX{m''}}}$-graded extension of $\fcat$. This allows us to apply the results of \cite{MR4192836} to classify the auto-equivalences of $\dcat{sl}{r+1}{k}$ extending a given auto-equivalence of $\fcat$. To convenience the reader, the results of \cite{MR4192836} state for a $G$-graded category $\oplus_G \cC_g$, the number of auto-equivalences extending $\cF \in \TenAut(\cC_e)$ is bounded above by
\[    | \phi \in  \Aut(G) : \cC_g \simeq \cC_{\phi(g)} \text{ as $\cC_e$-bimodules for all $g \in G$} | \cdot | H^1( G , \operatorname{Inv}(\mathcal{Z}( \cC_e)))| \cdot | H^2( G, \mathbb{C}^\times)|.\]

With this bound, we can determine the following result.

\begin{lemma}
The group of auto-equivalences of $\dcat{sl}{r+1}{k}$ extending the identity on the subcategory $\fcat$ is isomorphic to the group
\[    \left\{   a \in \Z{\frac{r+1}{mm''}} : 1 + a\frac{k}{m'} \text{ is coprime to } \frac{r+1}{mm''}\right\},\]
unless $r=1$, $k=2$, and $m = 1$, in which case the group is trivial.
\end{lemma}

\begin{proof}
We begin with the group $\{ \phi \in  \Aut(\Z{\frac{r+1}{mm''}}) : \cC_g \simeq \cC_{\phi(g)} \text{ as $\cC_e$-bimodules for all $g \in  \Z{\frac{r+1}{mm''}}$}\}$. As the $\cC_e$ bimodules form a group, we have that $\cC_g \simeq \cC_{\phi(g)} \text{ as $\cC_e$-bimodules for all $g \in  \Z{\frac{r+1}{mm''}}$}$ if and only if $\cC_{\phi(g)g^{-1}}$ is equivalent to the trivial $\cC_e$-bimodule, and so only if $\cC_{\phi(g)g^{-1}}$ contains an invertible object. Recall that the invertible objects of $\dcat{sl}{r+1}{k}$ are generated by the object $( k \Lambda_{m''}, 1)$, which lives in the graded component $\cC_{\frac{km''}{m}} = \cC_{\frac{k}{m'}}$. Therefore the invertible objects of $\dcat{sl}{r+1}{k}$ live in the graded components $\cC_{N\frac{k}{m'}}$ for $N \in \mathbb{N}$.

Let $c \in \Z{\frac{r+1}{mm''}}^\times$, then $\cC_{cg - g}$ contains an invertible object if and only if $\cC_{c - 1}$ does. For this to happen, we need that $c \equiv 1+N\frac{k}{m'}$ for some $N \in \mathbb{N}$. Using Bezout's identity, this is equivalent to having $c \equiv 1 \pmod {\gcd(\frac{k}{m'} , \frac{r+1}{mm''}    )}$. A direct prime by prime computation reveals that $\gcd(\frac{k}{m'} , \frac{r+1}{mm''} ) = \frac{n'}{m'}$, where we recall that $n' = \gcd(\changeX{r+1},k)$. Thus together we have a bound
\[  | \{ \phi \in  \Aut(\Z{\frac{r+1}{mm''}}) : \cC_g \simeq \cC_{\phi(g)} \text{ as $\cC_e$-bimodules for all $g \in  \Z{\frac{r+1}{mm''}}$}\}| \leq | \{   c \in \Z{\frac{r+1}{mm''}}^\times : c \equiv 1 \pmod {\frac{n'}{m'}}\}|.\]

Now we count the group $H^1\left(\Z{\frac{r+1}{mm''}} , \operatorname{Inv}(\mathcal{Z}( \fcat))\right)$.  As a 1-cocycle is determined by its value on the generator, we have that the size of this group is bounded above by the size of $\operatorname{Inv}(\mathcal{Z}( \fcat))$. As the universal grading group of $\fcat$ is trivial, we can use \cite{MR3354332} to see that every invertible of $\fcat$ has at most one lift to the centre. Further, as $\fcat$ is braided each invertible object has a lift to the centre via the braiding. Therefore
\[\operatorname{Inv}(\mathcal{Z}( \fcat)) \cong \operatorname{Inv}( \fcat) \cong \Z{\frac{n'}{m'}},\]
and so the size of the group $H^1\left(\Z{\frac{r+1}{mm''}} , \operatorname{Inv}(\mathcal{Z}( \fcat))\right)$ is bounded above by $\frac{n'}{m'}$.

It is a classical group theory result that $H^2\left( \Z{\frac{r+1}{mm''}}, \mathbb{C}^\times\right)$ is trivial.

Thus there are at most
\[   \frac{n'}{m'} \cdot \left| \{   c \in \Z{\frac{r+1}{mm''}}^\times : c\equiv 1 \pmod {\frac{n'}{m'}}\}\right|\]
auto-equivalences of $\dcat{sl}{r+1}{k}$ extending the identity on the $\fcat$ subcategory. We now show this bound is sharp by constructing enough distinct auto-equivalences of $\dcat{sl}{r+1}{k}$ to realise the upper bound. We will construct these auto-equivalences as simple current auto-equivalences. For the definition of simple current auto-equivalences we use in this paper, see \cite[Lemma 2.4]{ABCG}.

To construct simple current auto-equivalences, we pick out the invertible object $( k\Lambda_{m''}, 1) \in \dcat{sl}{r+1}{k}$. This object has order $\frac{r+1}{mm''}$, and has self-braiding eigenvalue equal to $e^{2\pi i \frac{ mm''q}{2(r+1)}}$ where $q =
\frac{rk}{m'}$. Thus we get simple current auto-equivalences of $\dcat{sl}{r+1}{k}$ for each element of the set
\[    \left\{ a \in \Z{\frac{r+1}{mm''}} : 1 + a\frac{rk}{m'} \text{ is coprime to }\frac{r+1}{mm''}\right\}.\]

To see that these simple current auto-equivalences are distinct, note that they form a group. Therefore, we need to show that for each $a \neq 0$, the corresponding simple current auto-equivalence acts non-trivially. Consider $( \Lambda_{m'},1) \in \dcat{sl}{r+1}{k}$. Then we have that the simple current auto-equivalence sends
\[ ( \Lambda_{m'},1)\mapsto ( \Lambda_{m'},1)\otimes ( k\Lambda_{m''}, 1) = ( (k-1)\Lambda_{m''} + \Lambda_{m' + m''})    ,1  ). \]
To verify that this action in non-trivial, we have to check that $ (k-1)\Lambda_{m''} + \Lambda_{m' + m''}) $ does not live in the orbit of $\Lambda_{m'}$ under the action of $\Z{m}$. Supposing this was the case, then there would exist a $t \in \mathbb{N}$ such that
\[  (k-1)\Lambda_{m''} + \Lambda_{m' + m''})  = (k-1)\Lambda_{t\frac{r+1}{m}} + \Lambda_{t\frac{r+1}{m} + m'}.  \]
Assuming $k  >  2$, we get the equation
\[   t \frac{r+1}{mm''} \equiv 1 \pmod{r+1},    \]
which is nonsense, as $\frac{r+1}{mm''} $ is clearly not invertible in $\Z{r+1}$.

If $k=2$, then we get that $2m' \equiv 0 \pmod {r+1}$, and thus $m' = \frac{r+1}{2}$. As $m' \divides k$, we see that either $r=1$ and $m' = 1$, or $r=3$ and $m' = 2$. The latter case is one of the excluded cases. For the former case, it is known that the simple current auto-equivalence acts trivially \cite[Theorem 1.2]{MR3808052}.

The same argument used in \cite[Lemma A.2]{ABCG} shows that the set of simple current auto-equivalences, and the set
\[  \left\{b \in \Z{\frac{r+1}{mm''}\frac{n'}{m'}} : b \equiv 1  \pmod  { \frac{n'}{m'}  }\right\}   \]
have the same size. We give a bijection
\[      \frac{n'}{m'} \cdot \left\{   c \in \Z{\frac{r+1}{mm''}}^\times : c\equiv 1 \pmod {\frac{n'}{m'}}\right\} \to \left\{b \in \Z{\frac{r+1}{mm''}\frac{n'}{m'}} : b \equiv 1  \pmod  { \frac{n'}{m'}  }\right\}\]
by sending
\[   (N, c) \mapsto c + N \frac{n'}{m'}.\]
As the simple current auto-equivalences are all distinct (except for $\ecat{sl}{2}{2}{1}$), and the number of them is equal to the upper bound of auto-equivalences extending the identity on the subcategory $\fcat$, we therefore have that every auto-equivalences extending the identity on the subcategory $\fcat$ is isomorphic to the group of simple current auto-equivalences, which is
\[    \left\{ a \in \Z{\frac{r+1}{mm''}} : 1 + a\frac{rk}{m'} \text{ is coprime to }\frac{r+1}{mm''}\right\}.\]

\end{proof}

\textit{A-priori} there should be no reason that the upper bound on the number of auto-equivalences we construct should be tight. We suspect that something deep going on here that deserves to be investigated.

As a corollary, we can determine which auto-equivalence which extend the identity are braided.

\begin{cor}

The group of braided auto-equivalences of $\dcat{sl}{r+1}{k}$ extending the identity on the subcategory $\fcat$ is isomorphic to the group
\[   \Z{2}^{p+t},\]
where
\begin{itemize}
\item $p$ is the number of distinct odd primes dividing $\frac{r+1}{mm''}$ but not $\frac{k}{m'}$, and
\item $t = \begin{cases} 0 \text{ if $\frac{r+1}{mm''}$ is odd, or if $\frac{k}{m'} \equiv 0 \pmod 4$, or if both $\frac{k}{m'}$ is odd, and $\frac{r+1}{mm''} \equiv 2 \pmod 4$} \\
 1 \text{ otherwise,}
\end{cases}$
\end{itemize}
unless $r=1$, $k=2$, and $m = 1$, in which case the group is trivial.

\end{cor}

\begin{proof}
\changeX{We know that a simple current auto-equivalence is braided precisely when $a^2\frac{rk}{m'} - 2a \equiv 0 \pmod {2\frac{r+1}{mm''}}$. It is then immediate that the group of simple currents for $\dcat{sl}{r+1}{k}$ is isomorphic to the group of simple currents for $\cat{sl}{\frac{r+1}{mm''}}{\frac{k}{m'}}$ (with compositions as in \cite[Appendix A]{ABCG}). The claim then follows from \cite[Theorem 1.1]{ABCG}, where we ignore the $\mathbb{Z}_2^c$ factor of $\operatorname{EqBr}( \cat{sl}{\frac{r+1}{mm''}}{\frac{k}{m'}}   )$ which corresponds to a non simple current auto-equivalence.
}
\end{proof}

Now that we completely understand the braided auto-equivalences of $\dcat{sl}{r+1}{k}$ which extend the identity on the subcategory $\fcat$, we can use a torsor argument to fairly easily leverage this information to understand the auto-equivalences extending the charge-conjugation auto-equivalence on the subcategory $\fcat$ which fix the distinguished object $\Omega$. This completes the proof of Theorem~\ref{thm:mainnon}, the main result of this section.
\begin{proof}[Proof of Theorem~\ref{thm:mainnon}]

All that remains to be done is to show that there exists a braided auto-equivalence of $\dcat{sl}{r+1}{k}$ which restricts to give the charge-conjugation auto-equivalence of $\fcat$. This follows from the fact that charge conjugation exists for $\cat{sl}{r+1}{k}$, and it preserves the $\Rep(\Z{m})$ subcategory. Therefore it descends to the category $\dcat{sl}{r+1}{k}$.

\end{proof}

\section{Candidates for exceptional auto-equivalences}\label{sec:excep}

In the previous section we were able to completely determine all non-exceptional braided auto-equivalences of the categories $\dcat{sl}{r+1}{k}$. That is, we could determine all braided auto-equivalences which fixed the distinguished object $\Omega$. For this section we will focus on determining the braided auto-equivalences which move $\Omega$. This section will be combinatorial in nature, making use of the rich combinatorics of the categories $\cat{sl}{r+1}{k}$. Let us outline the arguments of this section.

Our main tool to determine when the category $\dcat{sl}{r+1}{k}$ has an exceptional auto-equivalence will be the following Lemma, which gives very restrictive necessary conditions.

\begin{lemma}\label{lem:objcond}
The category $\dcat{sl}{r+1}{k}$ has a braided exceptional auto-equivalence only if there exists a object $X \in \cat{sl}{r+1}{k}$ such that

\begin{itemize}
\item $X \nin [\Lambda_1 + \Lambda_r]$, and
\item the orbit of $X$ under the action of $\Z{m}$ is closed under charge-conjugation,
\item we have \[     [r]_{r,k}[r+2]_{r,k} = \frac{\dim(X)}{|\operatorname{Stab}_{\Z{m}}(X)|     } \geq \frac{\dim(X)}{|\operatorname{Stab}_{\Z{r+1}}(X)|     },\] and
\item the twist of $X$ is equal to the twist of $\Lambda_1 + \Lambda_r$.

\end{itemize}

\end{lemma}
\begin{proof}
 Suppose $\dcat{sl}{r+1}{k}$ has a braided exceptional auto-equivalence, then by definition there is an object $(X, \chi_X) \in \dcat{sl}{r+1}{k}$ such that $\Omega$ is mapped to $(X,  \chi_X)$ under the exceptional auto-equivalence, and $(X,  \chi_X)$ is not in the orbit of $\Omega$ under simple currents.

As $\Omega$ is self-dual, we have that $(X,  \chi_X)$ is self-dual, and hence the orbit of $X$ under $\Z{m}$ is closed under conjugation.

To obtain the dimension bound for $X$, we note that
\[  \dim(\Omega) = \dim((X, \chi_X)).   \]
From this we can obtain the inequality
\[   \dim(\Omega) =   \frac{\dim(X)}{|\operatorname{Stab}_{\Z{m}}(X)|     }\geq  \frac{\dim(X)}{|\operatorname{Stab}_{\Z{r+1}}(X)|     }. \]
The dimension of the object $\Omega$ is $[r]_{r,k}[r+2]_{r,k}$, hence we have the result.

To get the condition on the twist of $X$, note that a braided auto-equivalence of $\dcat{sl}{r+1}{k}$ will preserve twists by \cite[Lemma 2.2]{ABCG}. The twist of an object $(X,  \chi_X) \in \dcat{sl}{r+1}{k}$ is equal to the twist of $X \in \cat{sl}{r+1}{k}$. The condition is then immediate.
\end{proof}

The key restriction here is the existence of an object $X\in  \cat{sl}{r+1}{k}$ with
\[      [r]_{r,k}[r+2]_{r,k}   \geq \frac{\dim(X)}{|\operatorname{Stab}_{\Z{r+1}}(X)|     }.   \]
If $\operatorname{Stab}_{\Z{r+1}}(X)$ is non-trivial, then we can use Lemma~\ref{lem:boundfix} to bound the dimension of $X$ below by the dimension of a simpler object in $\cat{sl}{r+1}{k}$, say for example $2\Lambda_2$. We can then use the hook formula to write the dimension of this simpler object as a product of quantum integers. For our $2\Lambda_2$ example we would have the dimension is $\frac{[r]_{r,k}[r+1]_{r,k}^2[r+2]_{r,k}}{[3]_{r,k}[2]^2_{r,k}}$. This then gives us an inequality of quantum integers that must be obeyed for there to exist an exceptional auto-equivalence. By suitably bounding this inequality we can then obtain strong restrictions on the rank and level of the category. With this approach we are able to show that there are only a finite number of cases where the inequality may hold. From here we can then directly search for $X$ where the condition
 \[     [r]_{r,k}[r+2]_{r,k} = \frac{\dim(X)}{|\operatorname{Stab}_{\Z{m}}(X)|     }\]
 holds. This yields a very small number of candidates for exceptional auto-equivalences.

When $\operatorname{Stab}_{\Z{r+1}}(X)$ is trivial, we search for objects $X \in \cat{sl}{r+1}{k}$ which satisfy
\[      [r]_{r,k}[r+2]_{r,k} =   \dim(X).   \]
Here there are many candidates for $X$. In particular, any object in $[\Lambda_1 + \Lambda_r]$ will satisfy this condition. However, when paired with the condition that  $X \nin [\Lambda_1 + \Lambda_r]$, we can again reduce the list of candidates down to a finite list via similar techniques as before. This case is a bit more fiddly than the case with non-trivial stabaliser group, as now we have to carefully avoid the objects in the orbit of $\Lambda_1 + \Lambda_r$, however the technical details remain the same.

In order to suitably bound the inequalities of quantum integers, we have to assume that $k \geq r+1$ in order to apply Lemma~\ref{lem:below}. To deal with the $k < r+1$ cases, we use level-rank duality to reduce it to the $k \geq r+1$ case.

All together we can give a complete list of objects $X\in \cat{sl}{r+1}{k}$ such that $X \nin [\Lambda_1 + \Lambda_r]$ and such that
\[     [r]_{r,k}[r+2]_{r,k} = \frac{\dim(X)}{|\operatorname{Stab}_{\Z{m}}(X)|     }. \]
From this finite list we then search for objects which satisfy the remaining conditions of Lemma~\ref{lem:objcond} to obtain an even smaller list.

Finally, we computer search the fusion rings of these remaining candidates, looking for fusion ring automorphisms which preserve the twists of the simples. This yeilds the main theorem of this section.

\begin{theorem}\label{thm:candy}
Let $r \geq 1$ and $k\geq 2$ and $m$ a divisor of $r+1$ satisfying $m^2 \divides k(r+1)$ if $r$ is even, and  $2m^2 \divides k(r+1)$ if $r$ is odd. Then except for the cases
\begin{align*}
 &\ecat{sl}{2}{16}{2},\quad & \ecat{sl}{3}{9}{3},\quad &\ecat{sl}{4}{8}{4},\quad &\ecat{sl}{5}{5}{5},\\
 & \ecat{sl}{8}{4}{4},\quad &  \ecat{sl}{9}{3}{3},\quad & \ecat{sl}{16}{2}{2}, \text{ and } &\ecat{sl}{16}{2}{4}
 \end{align*}
every braided auto-equivalence of $\ecat{sl}{r+1}{k}{m}$ is non-exceptional.

For the first four cases, we have that there are two possibilities for the group of braided auto-equivalences:

\begin{align*}
\BrAut(    \ecat{sl}{2}{16}{2}   ) & \in  \{   \Z{2} , S_3\}  , \qquad &&\BrAut(    \ecat{sl}{3}{9}{3}   )  \in \{D_3, S_4\}\\
\BrAut(    \ecat{sl}{4}{8}{4}   ) & \in   \{D_4, S_4\}, \quad \text{and} &&\BrAut(    \ecat{sl}{5}{5}{5}   ) \in \{D_5,   A_5\}.
\end{align*}

%

For the remaining four cases, we have that
\begin{align*}
\BrAut(\ecat{sl}{8}{4}{4}) &= \BrAut(\ecat{sl}{4}{8}{4}),\\
\BrAut( \ecat{sl}{9}{3}{3}) &= \BrAut(\ecat{sl}{3}{9}{3}) \times \Z{2},\\
\BrAut( \ecat{sl}{16}{2}{2}) &= \BrAut(\ecat{sl}{2}{16}{2}) \times \Z{2}, \text{ and }\\
\BrAut(\ecat{sl}{16}{2}{4}) &= \BrAut( \ecat{sl}{2}{16}{2}).
\end{align*}
\end{theorem}

With the high-level arguments and end goal in mind. Let us proceed with the fine details of the arguments. Let $\dcat{sl}{r+1}{k}$ be a category with an exceptional braided auto-equivalence. Then by Lemma~\ref{lem:objcond} we get an object $X \in \cat{sl}{r+1}{k}$ satisfying the conditions of the lemma. Our goal is show that $r$, $k$, and $m$ are severely constrained. We will have to split into several cases, depending on the stabaliser group of $X \in \cat{sl}{r+1}{k}$, and on the size of $k$ compared to $r+1$.

\subsection*{Case: $\Stab(X) = r+1$} \hspace{1em}

Let us first deal with the case where $X$ has full stabaliser subgroup, i.e. $\Stab(X) = \Z{r+1}$. As $|\Stab(X)|$ divides $ k$, we necessarily have $k \geq r+1$ in this case. From Lemma~\ref{lem:boundfix} we can deduce that
\[  \dim(X) \geq \dim(2\Lambda_2), \quad \dim(X) \geq \dim(6\Lambda_1), \quad \text{ and } \quad \dim(X) \geq \dim(3\Lambda_3).\]
\changeX{These inequalities hold when $k\geq 2$, $k\geq 6$, and $k\geq 3$ respectively. Recalling $\dim(X) = [r]_{r,k}[r+2]_{r,k}$, } we get the inequalities
\[ [r]_{r,k}[r+2]_{r,k}\geq \frac{\dim(2\Lambda_2)}{r+1}, \quad [r]_{r,k}[r+2]_{r,k} \geq \frac{\dim(6\Lambda_1)}{r+1} ,\quad \text{ and } \quad  [r]_{r,k}[r+2]_{r,k} \geq \frac{\dim(3\Lambda_3)}{r+1}.    \]

Let us focus on this first inequality for now. \changeX{As $k \geq r+1$ in this case, we have that $k\geq 2$ for all $r\geq 1$, and so this inequality holds in all cases.} Expanding this inequality with the hook formula, and simplifying gives
\begin{equation}\label{eq:inequal1}
 (r+1)[3]_{r,k}[2]_{r,k}^2 \geq  [r+1]_{r,k}^2.
\end{equation}

The left hand side we can bound above by $(r+1)\cdot 12$ by Lemma~\ref{lem:above}, giving that the above inequality can only hold if the weaker inequality
\begin{equation}\label{eq:inequal2}
 12(r+1) \geq  [r+1]_{r,k}^2
\end{equation}
holds. As $k \geq r+1$, we have that $r+1 \leq \frac{1}{2}(1 + r + 1 + r) \leq  \frac{1}{2}(1 + r + k)$, so we can apply Lemma~\ref{lem:below} to obtain the bound $[r+1]_{r,k} \geq \frac{r+1}{2}$. This gives us the even weaker inequality
\[ 48(r+1) \geq (r+1)^2\]
which only holds if $r\leq 47$.

For each $r\leq 47$ we still have an infinite number of $k$ where the initial inequality may hold. Let us return to the inequality from Equation~\eqref{eq:inequal2}. For each fixed $r \leq 47$, the left hand side is constant, while the right hand side is an increasing function of $k$. Therefore if we can find a smallest $k$ for which this inequality breaks, then we know it will also break for all larger $k$. This leaves us with a finite list of $k$ for which the initial inequality from Equation~\eqref{eq:inequal1} can hold. Finally, we check each of these finite potential solutions against Equation~\eqref{eq:inequal1} to obtain an even smaller list of potential candidates.

We find the following finite list of potential solutions for $r \geq 12$.
\begin{center}
\begin{tabular}{ l| c }
 $r$ & Potential $k$ \\
 \hline
12 & $13\leq k \leq 68$ \\
 13 & $14\leq k \leq 49$\\
 14 & $15\leq k \leq 42$ \\
 15 & $16\leq k \leq 38$ \\
 16 & $17\leq k \leq 35$ \\
 17 & $18\leq k \leq 33$\\
 18 & $19\leq k \leq 32$ \\
 19 & $20\leq k \leq 31$ \\
 $20 \leq r \leq 22$ & $r+1 \leq k \leq 30$\\
 $23 \leq r \leq 29$ & $r+1 \leq k \leq 29$\\
 $r \geq 30$ &$\emptyset$
\end{tabular}
\end{center}

\begin{rmk}\label{rmk:boundtrick}
We will repeatedly use the above trick in order to leverage an inequality of quantum integers, into an explicit list of $r$ and $k$ where the inequality holds. To summarise, we begin with an inequality $\text{left} \leq \text{right}$ of quantum integers. We then use the bound from Lemma~\ref{lem:below} to bound the left equation below, and the bound from Lemma~\ref{lem:above} to bound the right equation above. These bounds remove the quantum integers, and the resulting inequality gives an upper bound on $r$. We now return to the equation $\text{left} \leq \text{right}$, but this time only bound the right hand side above, by a function of $r$. We plug each of our finite $r$ into this inequality, giving a new inequality which states that a product of quantum integers is less than some constant. As quantum integers are an increasing function of $k$ (once $r$ is fixed), we can find the smallest $k$ which breaks the inequality, which tells us it also breaks for all larger $k$. At this point we may find that no $k$ breaks the inequality. When this happens we have to throw away the $r$, and find a different inequality of quantum integers to deal with that particular $r$. This leaves us with a finite number of $r$ where the inequality may hold, and for some subset of these $r$, a finite list of $k$ where the inequality may hold. To further the finite list of $k$, we test each possible solution against the initial inequality of quantum integers.
\end{rmk}

For $r\leq 11$ there is no $k$ where Equation~\eqref{eq:inequal2} breaks. To deal with the case of $r\leq 11$ let us now consider the inequality
\[  \dim(X) \geq \changeX{\frac{ \dim(6\Lambda_1)}{r+1}}.\]
\changeX{Recall this inequality holds if $k\geq 6$. Assuming $k\geq 6$, we expand the above inequality with the hook formula to get the inequality }
\[ [6]_{r,k}[5]_{r,k}[4]_{r,k}[3]_{r,k}[2]_{r,k}[r]_{r,k}(r+1) \geq [r+6]_{r,k}[r+5]_{r,k}[r+4]_{r,k}[r+3]_{r,k}[r+1]_{r,k}.\]
Playing the game from Remark~\ref{rmk:boundtrick} we find this inequality breaks for all $k\geq 41$. \changeX{Hence, the object $X$ can only satisfy the dimension condition of Lemma~\ref{lem:objcond} if $k < \max(6,41) = 41$.}

Together we have a finite list of $r$ and $k$ such that \changeX{$X$ could possibly have the correct dimension}. \changeX{That is, if $r \leq 11$, then $k< 41$,and if $r\geq 12$, then $k$ is one of the finite number of values in the above table.} This is still an unreasonable number of cases to computer search through. For example $\cat{sl}{30}{30}$ has on the order of $10^{16}$ simple objects. To refine our finite list of potential solutions further we run each solution of $r$ and $k$ through the inequality
\[ [r]_{r,k}[r+2]_{r,k} \geq \frac{ \dim(3\Lambda_3)}{r+1}.\]
\changeX{Recall this inequality only holds when $k\geq 3$. As $k\geq r+1$ in this case, the only situation where this inequality doesn't necessarily hold is $r=1$ and $k=2$. However for these values the inequality is still good (as the inequality simplifies to $1\geq \frac{1}{r+1}$.)}

This yields the following list of $r$ and $k$, such that $\cat{sl}{r+1}{k}$ may have an object $X$ with $\Stab(X) = \Z{r+1}$, and with $\frac{ \dim(X) }{\Stab(X)} \leq [r]_{r,k}[r+2]_{r,k} $.
\begin{center}
\begin{tabular}{ l| c }
 $r$ & Potential $k$ \\
 \hline
 1 & $2\leq k \leq 40$\\
 2 & $3\leq k \leq 40$ \\
  3 & $4\leq k \leq 40$ \\
   4 & $5\leq k \leq 15$ \\
    5 & $6\leq k \leq 8$ \\
     6 & 7 \\
 $r \geq 7$ &$\emptyset$
\end{tabular}
\end{center}
 From this small finite list, we can computer search to find all objects $X \in \cat{sl}{r+1}{k}$ such that $\frac{\dim(X)}{r+1 } =    [r]_{r,k}[r+2]_{r,k}$. This yields the following result

\begin{lemma}\label{lem:d=k}
Let $r\geq 1$, and $k \geq r+1$. There exists a object $X$ of $\cat{sl}{r+1}{k}$ with $\frac{ \dim(X)}{r+1} = [r]_{r,k}[r+2]_{r,k}$ if and only if
\begin{enumerate}
\item $r =1$ and $k=16$, in which case $X =  8\Lambda_1$, or
\item $r = 2$ and $k=9$, in which case $X =  3\Lambda_1 + 3\Lambda_2$, or
\item $r=  4$ and $k = 5$, in which case $X =  \Lambda_1 + \Lambda_2+\Lambda_3+\Lambda_4$.
\end{enumerate}
\end{lemma}

\subsection*{Case: $k \geq r+1$ and $|\Stab(X)| \nin \{ 1, r+1\}$} \hspace{1em}

Let us now consider the case where $k \geq r+1$, and $|\Stab(X)| \nin \{ 1, r+1\}$. We can immediately assume that $r\geq 3$, as if $r \in \{1,2\}$, then there are no possibilities for $|\Stab(X)|$ which must divide $r+1$. \changeX{Hence we can also assume that $k\geq 4$.}

As $|\Stab(X)| \nin \{1, r+1\}$ and $r\geq 3$ we have $2 \leq \frac{r+1}{|\Stab(X)|} \leq r-1$. \changeX{As $k\geq 4$, we can use} Lemma~\ref{lem:boundfix}, along with Lemma~\ref{lem:orderR} \changeX{to see that}
\[ \dim(X) \geq \dim\left(3\Lambda_{\frac{r+1}{d}}\right) \geq \dim(3\Lambda_2).\]

We expand this inequality as

\[ |\Stab(X)| [4]_{r,k}[3]^2_{r,k}[2]^2_{r,k}  \geq [r+1]^2_{r,k}[r+2]_{r,k}[r+3]_{r,k} .\]

As $|\Stab(X)| \leq \frac{r+1}{2}$, we can bound the left hand side above to get the weaker inequality
\begin{equation}\label{eq:firstsemi}
 \frac{r+1}{2}[4]_{r,k}[3]^2_{r,k}[2]^2_{r,k}  \geq [r+1]^2_{r,k}[r+2]_{r,k}[r+3]_{r,k} .
 \end{equation}

As $r \geq 3$, we have that
\[r+1, r+2, r+3 \leq \frac{3}{4}(1 + r + 1 + r) \leq \frac{3}{4}(1 + r + k).\]
Thus we can apply Lemma~\ref{lem:below} to get the lower bounds
\[   [r+3]_{r,k} \geq \frac{r+3}{4}, \quad [r+2]_{r,k} \geq \frac{r+2}{4},\quad \text{ and } \quad [r+1]_{r,k} \geq \frac{r+1}{4}.\]

With these bounds, we can use the methods described in Remark~\ref{rmk:boundtrick} to obtain a finite list of solutions. We can ignore the $r = 4$ and $r=6$ cases, as in both these cases $r+1$ is prime, and so $|\Stab(X)|$ must be either $1$ or $r+1$.

This yields the following list of $r$ and $k$, such that $\cat{sl}{r+1}{k}$ may have an object $X$ with $\Stab(X) \nin\{1,r+1\}$, and with $\frac{ \dim(X) }{\Stab(X)} \leq [r]_{r,k}[r+2]_{r,k} $.
\begin{center}
\begin{tabular}{ l| c }
 $r$ & Potential $k$ \\
 \hline
 3 & $4 \leq k \leq 9$\\
 5& $6\leq k \leq 7$ \\
 $r \geq 6$ &$\emptyset$
\end{tabular}
\end{center}
 From this small finite list, we can computer search to find all objects $X \in \cat{sl}{r+1}{k}$ such that $\frac{\dim(X)}{ \operatorname{Stab}_{\Z{m}}(X) } =    [r]_{r,k}[r+2]_{r,k}$. This yields the following result
\begin{lemma}\label{lem:d<r}
Let $r\geq 1$, $k \geq r+1$, and $m$ a divisor of $r+1$ such that $m^2 \divides k(r+1)$ if $r$ is even, or such that $2m^2 \divides k(r+1)$ if $r$ is odd. There exists an object of $X\in \cat{sl}{r+1}{k}$ with $|\operatorname{Stab}_{\Z{m}}(X) |\nin\{1,r+1\}$ and $\frac{ \dim(X)}{|\operatorname{Stab}_{\Z{m}}(X)| } = [r]_{r,k}[r+2]_{r,k}$ if and only if
\begin{enumerate}
\item $r = 3$, $k= 8$, and $m = 2$, in which case $X \in [4\Lambda_2]$, or
\item $r = 3$, $k= 8$, and $m = 4$, in which case $X \in [4\Lambda_2]$.
\end{enumerate}
\end{lemma}

\subsection*{Case: $|\operatorname{Stab}_{\mathbb{Z}_{r+1}}(X)| = 1$}

We now have to deal with the case where the object $X$ has trivial stabilizer group. The difficulty here lies in the fact that many objects close to the corners of the Weyl chamber have trivial stabilizer subgroup, and are of small dimension. In fact, the object $2\Lambda_1$ (nearly always) has trivial stabilizer group, and has dimension smaller that $[r]_{r,k}[r+2]_{r,k}$. The objects $X\in \cat{sl}{r+1}{k}$ such that $\dim(X) = [r]_{r,k}[r+2]_{r,k}$ are classified in Appendix~\ref{app:terry} (Proposition~\ref{prop:samedim}).

\subsection*{Case: $k < r+1$}

We now have to consider the case where the level is small compared to the rank i.e. $k < r+1$. Using \textit{level-rank duality} we can reduce this argument to the $k\geq r+1$ case.

There are many interpretations of level-rank duality (see \cite{MR3162483}). For us, we only need a weak version, which relates the dimensions of objects in the categories $\cat{sl}{r+1}{k}$ and $\cat{sl}{k}{r+1}$. Given an object
\[  X = \sum_{i=0}^r \lambda_i \Lambda_i \in \cat{sl}{r+1}{k}\]
we can form the $r \times k$ Young tableaux
\[   T(X) := \left( \sum_{i=1}^r \lambda_i , \sum_{i=2}^r \lambda_i , \cdots , \lambda_r\right).\]
Taking the transpose of this tableaux gives a $k\times r$ tableaux. Initially this presents a spanner for a level rank-duality connection, as the objects of $\cat{sl}{k}{r+1}$ are identified by $(k-1) \times (r+1)$ tableauxes. Thus level-rank duality at first glance appears to give a connection between $\cat{sl}{r+1}{k}$ and $\cat{sl}{k+1}{r}$. However this connection is superficial at best, and only shows the ranks of the two categories are equal. Instead we will restrict our attention to objects $X \in \cat{sl}{r+1}{k}$ with $\lambda_0 \neq 0$. With this restriction, the the tableaux $T(X)$ can be considered as a $r\times (k-1)$ tableaux, and thus the transpose can be identified with an object of $\cat{sl}{k}{r+1}$. We write $X^T$ for this transposed object of $\cat{sl}{k}{r+1}$. Explicitly we have that
\[  X^T = \sum_{\ell = 1}^k \hat{\lambda_\ell} \Lambda_\ell,\]
where
\[  \hat{\lambda_\ell} := \left|   \changeX{1} \leq j \leq r : \sum_{i=0}^{j-1} \lambda_i = k- \ell  \right  |. \]

 Using the hook formula, along with the fact that $[n]_{r,k} = [1 +r + k - n]_{k,r}$ we see that
 \[  \dim(X) = \dim\left(X^T\right).\]

 In order to apply level-rank duality arguments to study the exceptional auto-equivalences of $\dcat{sl}{r+1}{k}$, we need to understand how the the stabaliser group $\operatorname{Stab}_{\Z{m}}(X)$ is affected by level-rank duality. There is a subtlety here in that $m$ doesn't necessarily divide $k$, and so talking about $\operatorname{Stab}_{\Z{m}}(X^T)$ doesn't make sense. We solve this problem, and resolve the subtlety in the following lemma.
 \begin{lemma}
Let $m$ a divisor of $r+1$, such that $m^2 \divides k(r+1)$ if $r$ is even, or such that $2m^2 \divides k(r+1)$ if $r$ is odd, and set $m' = \gcd(k,m)$. Then we have isomorphisms
 \[ \operatorname{Stab}_{\Z{m}}(X)\cong \operatorname{Stab}_{\Z{m'}}(X)  \cong \operatorname{Stab}_{\Z{m'}}(X^T).   \]
 \end{lemma}

 \begin{proof}
 We will first show that
 \[\operatorname{Stab}_{\Z{m}}(X)\cong \operatorname{Stab}_{\Z{m'}}(X).\]

 Suppose that $\operatorname{Stab}_{\Z{m}}(X)\cong \Z{d}$, then we have that
 \[  X = \sum_{i = 1}^{\frac{r+1}{d}} \lambda_i \sum_{j=1}^d \Lambda_{\changeX{j} \frac{r+1}{d} + \changeX{i}}.\]
 This implies that $d$ divides $k$, and therefore $ \operatorname{Stab}_{\Z{m'}}(X) \cong \Z{d}$.

 For the second isomorphism we want to show that
 \[ \operatorname{Stab}_{\Z{m'}}(X)  \cong \operatorname{Stab}_{\Z{m'}}\left(X^T\right).\]
 Suppose $\operatorname{Stab}_{\Z{m'}}(X) \cong \Z{d}$. Then
 \[   \lambda_i = \lambda_{i +\frac{r+1}{d}}\]
 for all $i \in  \Z{r}$. This implies that for any $j\in \Z{r}$ we have
 \[  \sum_{i = j}^{j -1+ \frac{r+1}{d}} \lambda_i = \frac{k}{d}.\]

 To prove the claim of the lemma, we have to show that
 \[   \hat{\lambda_\ell} = \hat{\lambda}_{\ell  + \frac{k}{d}  },    \]
 for all $\ell \in \Z{k}$.

 Suppose we have $0 \leq j \leq r$ contributes to $\hat{\lambda}_\ell$. That is
 \[  \sum_{i = 0}^{j-1} \lambda_i  = k - \ell.    \]
 Then we have that
 \[   \sum_{i = 0}^{j - \frac{r+1}{d}-1} \lambda_i = \sum_{\ell = 0}^{j-1} \lambda_i - \sum_{i = j - \frac{r+1}{d}}^{j-1} \lambda_i  = k - \ell - \frac{k}{d}, \]
 and so $j - \frac{r+1}{d}$ contributes to $\changeX{\hat{\lambda}_{\ell + \frac{k}{d}}}$. This implies that
 \[   \hat{\lambda_\ell} = \hat{\lambda}_{\ell  + \frac{k}{d} }  \]
and hence we have the result.
 \end{proof}

 Now suppose $k < r+1$, and let $m$ a divisor of $r+1$ such that $m^2 \divides k(r+1)$. Let $X$ be an object of $\cat{sl}{r+1}{k}$ with $\operatorname{Stab}_{\Z{m}}(X) = \Z{d}$ for some $d$, with $\dim(X) =\frac{[r]_{r,k}[r+2]_{r,k}}{d}$. By hitting $X$ with a suitable simple current, we can assume that $\lambda_0 \neq 0$, and thus we can apply level-rank duality to get an object $X^T \in \cat{sl}{k}{r+1}$ with
 \[  \dim(X^T)  =  \frac{[r]_{r,k}[r+2]_{r,k}}{|\operatorname{Stab}_{\Z{m'}}(X^T)|}. \]
 Furthermore, we note that if $X^T \in [\Lambda_1 + \Lambda_k]$, then $X\in  \{  \Lambda_1 + \Lambda_r , (k-2)\Lambda_1 +\Lambda_2, \Lambda_{r-1} + (k-2)\Lambda_r\} \subset [ \Lambda_1 + \Lambda_r]$.

 Together with Lemma~\ref{lem:d=k}, Lemma~\ref{lem:d<r}, and Proposition~\ref{prop:samedim}, which tell us the objects $X\in \cat{sl}{r+1}{k}$ with $\frac{\dim(X)}{\operatorname{Stab}_{\Z{m}}(X) } = [r]_{r,k}[r+2]_{r,k}$ when $k \geq r+1$, we can use level-rank duality to extend this result to all $k\geq 2$. From this finite list, we can directly search for objects that satisfy the remaining conditions of Lemma~\ref{lem:objcond}. Namely, we consider the subset of objects $X$ with the same twist as $\Lambda_1 + \Lambda_r$ and such that $X^* \in [X]$. This yields the following lemma.

\begin{lemma}\label{lem:boundlist}
Let $r\geq 1$, $k  \geq 2$, and $m$ a divisor of $r+1$ such that $m^2 \divides k(r+1)$ if $r$ is even, or such that $2m^2 \divides k(r+1)$ if $r$ is odd. There may exist an exceptional auto-equivalence of $\dcat{sl}{r+1}{k}$ only in the following cases:
\begin{center}
  \begin{tabular}{c|c}
    	\toprule
			$\cC$ &   Image of $\Omega$ \\
	\midrule
	
	$\ecat{sl}{2}{16}{2}$ &  $ (8\Lambda_1 , \chi_n)$ \hspace{.5em}:  $n\in \Z{2}$\\[1em]

	$\ecat{sl}{3}{9}{3}$ &  $ (3\Lambda_1 + 3\Lambda_2, \chi_n)$\hspace{.5em}:  $n\in \Z{3}$\\[1em]

	$\ecat{sl}{4}{8}{2}$  & $ (4\Lambda_2, \chi_n)$\hspace{.5em}:  $n\in \Z{2}$\\
	$\ecat{sl}{4}{8}{2}$  & $ (4\Lambda_1 + 4\Lambda_3, \chi_n)$\hspace{.5em}:  $n\in \Z{2}$\\
	$\ecat{sl}{4}{8}{4}$  & $ (4\Lambda_2, \chi_n)$\hspace{.5em}:  $n\in \Z{2}$\\[1em]

	$\ecat{sl}{5}{5}{5}$  & $( \Lambda_1 + \Lambda_2 + \Lambda_3 + \Lambda_4, \chi_n)$\hspace{.5em}:  $n\in \Z{5}$	 \\[1em]

	$\ecat{sl}{8}{4}{2}$ & $(	2\Lambda_4 , \chi_n)$\hspace{.5em}:  $n\in \Z{2}$       \\
	$\ecat{sl}{8}{4}{4}$ & $(	2\Lambda_4 , \chi_n)$\hspace{.5em}:  $n\in \Z{2}$       \\[1em]
	
	$\ecat{sl}{9}{3}{3}$ & $(\Lambda_3 + \Lambda_6 , \chi_n)$\hspace{.5em}:  $n\in \Z{3}$ \\[1em]

	$\ecat{sl}{16}{2}{2}$  & $(\Lambda_8 , \chi_n)$\hspace{.5em}:  $n\in \Z{2}$ \\
	$\ecat{sl}{16}{2}{4}$  & $(\Lambda_8 , \chi_n)$\hspace{.5em}:  $n\in \Z{2}$\\
    	\bottomrule

	    \end{tabular}
\end{center}
\end{lemma}

Now that we have this extremely small finite list of candidates for exceptional braided auto-equivalences of $\dcat{sl}{r+1}{k}$, we can explicitly search the fusion rings of these candidates to see if the exceptional braided auto-equivalence exists at the level of the fusion ring. Additionally, we check that the fusion ring automorphisms preserve the twists of the simple objects. We obtain the explicit data for these categories from the results of \cite{MR1409292}. \changeY{The idea behind \cite{MR1409292} is to use the free module functor to use the given knowledge of the modular data of $\mathcal{C}$ to determine as much about the S-matrix and twists of $\mathcal{C}_{\operatorname{Rep}(\mathbb{Z}_m)} ^0$ as possible. This functor completely determines the twists with no ambiguities. The ambiguities in the S-matrix regarding objects which split in the de-equivariantisation are then resolved using the standard modular data relations (such as $(ST)^3 = S^2)$. The fusion rules can then be determined via Verlinde. The rank of the categories $\ecat{sl}{2}{16}{2}, \ecat{sl}{3}{9}{3},\ecat{sl}{4}{8}{2}, \ecat{sl}{4}{8}{4}$, and $\ecat{sl}{5}{5}{5}$ are $6, 9, 50, 16$, and $10$ respectively. The remaining relevant data can be found in Mathematica files attached to the arXiv submission of this paper.} 

From the results of the previous section, we know precisely the non-exceptional braided auto-equivalences of all of the categories $\dcat{sl}{r+1}{k}$. This information helps us in two ways. First, we know that for all cases except for the finite exceptions in the above list, that all braided auto-equivalences are non-exceptional, hence we now fully understand their braided auto-equivalence groups. Second, we also know the braided auto-equivalences which fix the object $\Omega$ of the finite number of exceptions in the above list. Via compositional arguments, this allows us to rule out many potential exceptional auto-equivalences of these categories. This allows us to essentially determine the group structure of the braided auto-equivalence groups, up to the exceptional auto-equivalences existing.

\begin{lemma}
We have that
\begin{align*}
\BrAut(    \ecat{sl}{2}{16}{2}   ) & \in  \{   \Z{2} , S_3\}  , \qquad && \BrAut(    \ecat{sl}{3}{9}{3}   )  \in \{D_3, S_4\}\\
 \BrAut(    \ecat{sl}{4}{8}{2}   ) & \in   \{D_2 \}, \qquad &&\BrAut(    \ecat{sl}{4}{8}{4}   )  \in   \{D_4, S_4\}, \\
\BrAut(    \ecat{sl}{5}{5}{5}   ) &\in \{D_5,   A_5\}.
\end{align*}
\end{lemma}
\begin{proof}
From the results of \cite{MR1409292} we have the fusion rings and twists of each of the five above categories. We compute the group of fusion ring automorphisms which preserve the twists. We will refer to these groups as the braided fusion ring symmetries.

For $    \ecat{sl}{4}{8}{2}  $ we see that every braided fusion ring symmetry is non-exceptional, thus $\BrAut(    \ecat{sl}{4}{8}{2}   ) = D_2$ by Theorem~\ref{thm:mainnon}.

For the case of $ \ecat{sl}{2}{16}{2}$ we have that the braided fusion rings symmetries form a group isomorphic to $S_3$, with generators
\[      (8\Lambda_1, \chi_0) \leftrightarrow   (8\Lambda_1, \chi_1)     ,   \]
and
\[     \Omega\leftrightarrow   (8\Lambda_1, \chi_0)     .   \]
From Theorem~\ref{thm:mainnon} we know that the first generator is realised as a braided auto-equivalence of $\ecat{sl}{2}{16}{2}$. Thus $\BrAut(    \ecat{sl}{2}{16}{2}   )$ is an intermediate subgroup of $\Z{2}$ and $S_3$. There are only two such intermediate subgroups which are the end points.

For the case of $ \ecat{sl}{3}{9}{3}$ we have that the braided fusion rings symmetries form a group isomorphic to $\Z{2}\times S_4$, with generators
\[     (3\Lambda_1, 1) \leftrightarrow   (3\Lambda_2, 1) \quad , \quad (3\Lambda_1 + 3\Lambda_2, \chi_1)\leftrightarrow(3\Lambda_1 + 3\Lambda_2, \chi_2),    \]
and
\[   (3\Lambda_1 + 3\Lambda_2, \chi_0)\mapsto (3\Lambda_1 + 3\Lambda_2,\chi_1) \mapsto (3\Lambda_1 + 3\Lambda_2, \chi_2)  ,    \]
and
\[   \Omega \mapsto (3\Lambda_1 + 3\Lambda_2, \chi_0)\mapsto (3\Lambda_1 + 3\Lambda_2, \chi_1) \mapsto (3\Lambda_1 + 3\Lambda_2, \chi_2).\]

From Theorem~\ref{thm:mainnon} we know that the first two generators are realised as braided auto-equivalences of $ \ecat{sl}{3}{9}{3}$, and form a group isomorphic to $D_3$. Further, this theorem tells us that any braided auto-equivalence which fixes $\Omega$ must be in the subgroup generated by the first two generators. Thus $\BrAut(  \ecat{sl}{3}{9}{3}  )$ is an intermediate subgroup of $D_3 \subset \Z{2} \times S_4$ with the property that if $\Omega$ is fixed by an auto-equivalence, then this auto-equivalence lives in the $D_3$ subgroup. With this knowledge, we can study the intermediate subgroups of $D_3 \subset \Z{2} \times S_4$ to see that there are only two such subgroups with this property. These are $D_3$ with the first two generators, and $S_4$ with the first two generators, and the new generator
\[\Omega  \leftrightarrow (      3\Lambda_1 + 3\Lambda_2,\chi_0   ) \quad , \quad   (      3\Lambda_1 + 3\Lambda_2, \chi_1   )  \leftrightarrow (      3\Lambda_1 + 3\Lambda_2,\chi_2 ).\]
Thus $\BrAut(    \ecat{sl}{3}{9}{3}   )$ is isomorphic to either $D_3$ or $S_4$.

The remaining two cases fall to the same argument. For $ \ecat{sl}{4}{8}{4}$ the group of braided fusion ring symmetries is $\Z{2}\times S_4$. We have that the generators
\[       (4\Lambda_2, \chi_0) \leftrightarrow     (4\Lambda_2, \chi_1) \quad , \quad (2\Lambda_1 + 2\Lambda_2 + 2\Lambda_3, \chi_i) \mapsto  (2\Lambda_1 + 2\Lambda_2 + 2\Lambda_3, \chi_{i+1})       \]
and
\[       (2\Lambda_1 +\Lambda_2 , 1) \leftrightarrow  (\Lambda_2+2\Lambda_3 , 1) \quad, \quad (2\Lambda_1 + 2\Lambda_2 + 2\Lambda_3, \chi_1) \leftrightarrow (2\Lambda_1 + 2\Lambda_2 + 2\Lambda_3, \chi_3)       \]
form a $D_4$ subgroup of $\BrAut(  \ecat{sl}{4}{8}{4})$, and that any braided auto-equivalence which fixes $\Omega$ must be in this subgroup. Analysing the subgroup structure between $D_4$ and $\Z{2}\times S_4$ shows that at most there can be one more generator
\[     \Omega \leftrightarrow (4\Lambda_2, \chi_0)  \quad , \quad    (2\Lambda_1 +\Lambda_2 , 1) \mapsto  (2\Lambda_1 + 2\Lambda_2 + 2\Lambda_3, \chi_0) \mapsto (\Lambda_2+2\Lambda_3  , 1) \mapsto (2\Lambda_1 + 2\Lambda_2 + 2\Lambda_3, \chi_2),  \]
in the braided auto-equivalence group, which would form a group isomorphic to $S_4$. Thus $\BrAut(    \ecat{sl}{4}{8}{4}   )$ is either isomorphic to $D_4$ or $S_4$.

Finally for the case $ \ecat{sl}{5}{5}{5}$ we have that the group of braided fusion symmetries is $S_6$. We have that the generators
\[  (\Lambda_1 + \Lambda_2 + \Lambda_3 + \Lambda_4, \chi_i) \mapsto     (\Lambda_1 + \Lambda_2 + \Lambda_3 + \Lambda_4, \chi_{i+1})\]
and
\[  (\Lambda_1 + \Lambda_2 + \Lambda_3 + \Lambda_4, \chi_i) \mapsto     (\Lambda_1 + \Lambda_2 + \Lambda_3 + \Lambda_4, \chi_{-i})\]
form a $D_5$ subgroup of $\BrAut(  \ecat{sl}{5}{5}{5})$, and that any braided auto-equivalence which fixes $\Omega$ must be in this subgroup. Analysing the subgroup structure between $D_5$ and $S_6$ shows that at most there can be one more generator
\begin{align*} \Omega \mapsto (\Lambda_1 + \Lambda_2  +\Lambda_3  +\Lambda_4 , \chi_1) &\mapsto (\Lambda_1 + \Lambda_2  +\Lambda_3  +\Lambda_4 , \chi_0),  \\
(\Lambda_1 + \Lambda_2  +\Lambda_3  +\Lambda_4 , \chi_2)  \mapsto   (\Lambda_1 + \Lambda_2  +\Lambda_3  +\Lambda_4 , \chi_3) &\mapsto   (\Lambda_1 + \Lambda_2  +\Lambda_3  +\Lambda_4 , \chi_4)
\end{align*}
in the braided auto-equivalence group, which would form a group isomorphic to $A_5$. Thus $\BrAut(    \ecat{sl}{5}{5}{5}   )$ is either isomorphic to $D_5$ or $A_5$.
\end{proof}

To finish off the proof of the main theorem of this section, we need to deal with the remaining four cases. These can be dealt with easily by a level-rank duality argument.

\begin{prop}
We have the following braided equivalences:
\begin{align*}
\ecat{sl}{16}{2}{4} &\simeq (\ecat{sl}{2}{16}{2})^\text{rev}\boxtimes \operatorname{Vec}\left( \Z{2},\{1, e^{  2\pi i \frac{3 }{4}   } \}\right ),\\
\ecat{sl}{16}{2}{2} &\simeq (\ecat{sl}{2}{16}{2})^\text{rev}\boxtimes \operatorname{Vec}\left( \Z{8}, \{1,e^{  2\pi i \frac{15}{16}   } ,e^{  2\pi i \frac{3}{4}   }, e^{  2\pi i \frac{7}{16}   }, 1,  e^{  2\pi i \frac{7}{16}   },e^{  2\pi i \frac{3}{4}   }, e^{  2\pi i \frac{15}{16}   }\}\right),\\
\ecat{sl}{9}{3}{3} &\simeq (\ecat{sl}{3}{9}{3})^\text{rev}\boxtimes \operatorname{Vec}\left( \Z{3}, \{1,e^{  2\pi i \frac{1}{3}   },e^{  2\pi i \frac{1}{3}   }\}  \right), \text{ and }\\
\ecat{sl}{8}{4}{4} &\simeq (\ecat{sl}{4}{8}{4})^\text{rev}\boxtimes \operatorname{Vec}\left( \Z{2}, \{1,e^{  2\pi i \frac{3 }{4}   } \}\right ).
\end{align*}
\end{prop}
\begin{proof}
Let us do the computation for $\ecat{sl}{16}{2}{4}$, as the other cases follow in a similar fashion.

From Lemma~\ref{lem:adsub} we know that $\ecat{sl}{16}{2}{4}$ has a subcategory braided equivalent to $(\cat{sl}{16}{2}^\text{ad})_{\Rep(\Z{2})}$. Via level-rank duality \cite{MR3162483}, we have a braided equivalence
\[    \cat{sl}{16}{2}^\text{ad} \to (\cat{sl}{2}{16}^\text{ad})^\text{rev}.  \]
Thus together we have a braided subcategory of $\ecat{sl}{16}{2}{4}$ equivalent to $((\cat{sl}{2}{16}^\text{ad})^\text{rev})_{\Rep(\Z{2})}$. Taking the reverse braiding on the category commutes with taking the de-equivariantization (as $\Rep(\Z{2})^\text{rev} = \Rep(\Z{2})$), thus we have a subcategory braided equivalent to $ (\ecat{sl}{2}{16}{2})^\text{rev}$. As $ (\ecat{sl}{2}{16}{2})^\text{rev}$ is modular, Mugers theorem \cite[Theorem 4.2 ]{MR1990929} give us that the category $\ecat{sl}{16}{2}{4}$ factors as
\[    \ecat{sl}{16}{2}{4} \simeq (\ecat{sl}{2}{16}{2})^\text{rev}\boxtimes \mathcal{D},\]
where $ \mathcal{D}$ is a modular category with global dimension 2. To identify $ \mathcal{D}$ we study the object $( [2\Lambda_2 ] , 1) \in \ecat{sl}{16}{2}{4}$. The twist of this object is $e^{  2\pi i \frac{3 }{4}}$, and the order of the object is $2$. Thus $( [2\Lambda_2 ] , 1)$ generates a modular subcategory equivalent to $\operatorname{Vec}\left( \Z{2}, e^{  2\pi i \frac{3 }{4}   } \right )$. As the category $\ecat{sl}{2}{16}{2}$ has no non-trivial invertible objects, we must have that $( [2\Lambda_2 ] , 1)$ generates $ \mathcal{D}$, which completes the claim.

\end{proof}

\begin{cor}
We have
\begin{align*}
\BrAut(\ecat{sl}{16}{2}{4}) &= \BrAut( \ecat{sl}{2}{16}{2}), \\
\BrAut( \ecat{sl}{16}{2}{2}) &= \BrAut(\ecat{sl}{2}{16}{2}) \times \Z{2},\\
\BrAut( \ecat{sl}{9}{3}{3}) &= \BrAut(\ecat{sl}{3}{9}{3}) \times \Z{2}, \text{ and }\\
\BrAut(\ecat{sl}{8}{4}{4}) &= \BrAut(\ecat{sl}{4}{8}{4}).
\end{align*}
\end{cor}
\begin{proof}
We use the canonical embedding
\[   \BrAut( \cC) \times \BrAut(\mathcal{D}) \to \BrAut(\cC \boxtimes \mathcal{D}).\]
By analysing the fusion rings and twists of the the Deligne products from the previous proposition, we see that this embedding is an isomorphism.
\end{proof}

\section{Realisation of the exceptionals}\label{sec:real}

In the previous section we identified a finite list of the categories $\ecat{sl}{r+1}{k}{m}$ which may have an exceptional braided auto-equivalence, and furthermore, gave upper bounds for the number of such auto-equivalences that may exist. In this section, our goal is to construct all the exceptional braided auto-equivalences of these finite number of categories. That is, we want to compute the braided auto-equivalence groups of the categories
\[
\ecat{sl}{2}{16}{2}, \qquad
\ecat{sl}{3}{9}{3}, \qquad
\ecat{sl}{4}{8}{4}, \quad \text{and} \quad
\ecat{sl}{5}{5}{5}.
\]

While it is not immediate from the above list, there are two situations in play here. The first is for the category $\ecat{sl}{2}{16}{2}$, where the exceptional auto-equivalences come from the coincidence of categories
\[ \cat{so}{8}{3} \simeq(  \ecat{sl}{2}{16}{2})^\text{rev} \boxtimes \cat{so}{8}{1}. \]
The $S_3$ worth of braided exceptional auto-equivalences of $\ecat{sl}{2}{16}{2}$ is then naturally seen due to the triality of the Dynkin diagram $D_4$. This connection was initially discovered in \cite{MR2783128}.

\begin{lemma}  \cite[Theorem 4.3]{MR2783128}
We have
\[   \BrAut\left (\ecat{sl}{2}{16}{2}\right) = S_3.  \]
\end{lemma}

The far more interesting situation occurs with the other three categories. Studying the dimensions of these three examples, one notices that they are all defined over small quadratic fields. For $\ecat{sl}{3}{9}{3} $ the dimensions live in $\mathbb{Q}[ \sqrt{3}]$, for $\ecat{sl}{4}{8}{4} $ the dimensions live in $\mathbb{Q}[ \sqrt{3}]$, and for $\ecat{sl}{5}{5}{5} $ the dimensions live in $\mathbb{Q}[ \sqrt{5}]$. Such behaviour with the dimensions does not hold for general $\ecat{sl}{r+1}{k}{m}$, and suggests that these dimensions may have something to do with the potential exceptional auto-equivalences of these categories.

A large class of categories with objects living in quadratic fields are the quadratic categories, where the simple objects consist of the group of invertibles, and an object $\rho$, along with the orbit of $\rho$ under the action of the invertibles. The natural suspicion to draw, is that the three categories $\ecat{sl}{3}{9}{3} $, $\ecat{sl}{4}{8}{4} $, and $\ecat{sl}{5}{5}{5}$ should in some way be connected to quadratic categories. The naive guess, that these three categories are quadratic categories on the nose, is immediately thwarted by fact that the dimensions in these examples take on more than two values. Further, quadratic categories are almost never modular, where as our three examples are. However, this lack of modularity suggests the next place to look for a connection.

Taking the Drinfeld centre of a quadratic category gives a modular category whose dimensions lie in the same field as the quadratic category. While in the quadratic category, the dimensions of the simples can only have two possible values, the dimensions of the simples in the centre can be any integer combination of $1$ and the dimension of the non-invertible of the quadratic. This provides strong evidence that the three categories $\ecat{sl}{3}{9}{3} $, $\ecat{sl}{4}{8}{4} $, and $\ecat{sl}{5}{5}{5}$ \changeY{may be related to} Drinfeld centres of quadratic categories.

Now that we have an idea of what to look for, we can make educated guesses as to the identity of the quadratic categories. For example, the dimensions of $\ecat{sl}{3}{9}{3}$ are
\[       1,\quad 7  + 4 \sqrt{3},\quad 8 + 4\sqrt{3}, \quad \text{and} \quad 3 + 2 \sqrt{3} \quad \text{(6 times)}.\]
If $\ecat{sl}{3}{9}{3}$ were the Drinfeld centre of a quadratic category, then a natural guess for the dimension of the non-invertible would be $3 + 2 \sqrt{3}$, as all the above dimensions can be constructed as integer combinations of $1$ and $3 + 2 \sqrt{3}$. There is a known quadratic category with an object of this dimension, which is a near-group category with group of invertibles $\Z{3} = \{\mathbf{1}, g, g^2\}$, and a single non-invertible with fusion
\[   \rho \otimes \rho \cong \mathbf{1} \oplus g \oplus g^2 \oplus 6\rho.\]
The global dimension of this category is $24 + 12\sqrt{3}$, so the global dimension of its Drinfeld centre is $1008 + 576\sqrt{3}$. Where as the global dimension of $\ecat{sl}{3}{9}{3}$ is $336 + 192\sqrt{3}$. These global dimensions are off by a factor of three, which suggests a $\Z{3}$ factor is involved. From all this we conjecture that there is a quadratic category $\cC_{3,9,3}$ with fusion as above such that
\[       \mathcal{Z}(      \cC_{3,9,3}          ) \simeq \ecat{sl}{3}{9}{3} \boxtimes \operatorname{Vec}(\Z{3} , \{1,e^{2 i \pi \frac{1}{3}},e^{2 i \pi \frac{1}{3}}\}).\]

Using similar reasoning we conjecture the existence of a fusion category $\cC_{4,8,4}$ with invertibles $\Z{2}\times \Z{2} = \{\mathbf{1}, e, m , em\}$ and non-invertibles $\{ \rho,m\rho\}$ with fusion
\[   \rho \otimes \rho \cong \mathbf{1} \oplus e \oplus 6\rho \oplus 4 m\rho,         \]
such that
\[         \mathcal{Z}(      \cC_{4,8,4}          ) \simeq \ecat{sl}{4}{8}{4} \boxtimes\operatorname{Vec}(\Z{2}\times \Z{2} , \{1,e^{2 i \pi \frac{3}{4}},-1,e^{2 i \pi \frac{3}{4}}\}),\]
and a fusion category $\cC_{5,5,5}$ with invertibles $\Z{2}\times \Z{2} = \{\mathbf{1} , e, m , em\}$ and non-invertibles $\{ \rho, e\rho, m\rho, em\rho\}$ with fusion
\[  \rho \otimes \rho \cong \mathbf{1} \oplus \rho \oplus  e\rho\oplus m\rho \oplus em\rho, \]
such that
\[   \mathcal{Z}(      \cC_{5,5,5}          ) \simeq \ecat{sl}{5}{5}{5} \boxtimes  \operatorname{Vec}(\Z{2}\times \Z{2} ,\{1, -1,-1,-1\}). \]

We make all this precise with the following Theorem.

\begin{theorem}
There exist unitary fusion categories $ \cC_{3,9,3} $, $\cC_{4,8,3} $, and $\cC_{5,5,5}$ with combinatorics as above, such that
\begin{align*}
 \mathcal{Z}(      \cC_{3,9,3}          )    &\simeq \ecat{sl}{3}{9}{3} \boxtimes \operatorname{Vec}(\Z{3} , \{1,e^{2 i \pi \frac{1}{3}},e^{2 i \pi \frac{1}{3}}\})\\
     \mathcal{Z}(      \cC_{4,8,4}          ) &\simeq \ecat{sl}{4}{8}{4} \boxtimes \operatorname{Vec}(\Z{2}\times \Z{2} , \{1,e^{2 i \pi \frac{3}{4}},-1,e^{2 i \pi \frac{3}{4}}\}),\quad \text{ and }\\
      \mathcal{Z}(      \cC_{5,5,5}          ) &\simeq \ecat{sl}{5}{5}{5} \boxtimes \operatorname{Vec}(\Z{2}\times \Z{2} ,\{1, -1,-1,-1\}).
\end{align*}
\end{theorem}

\begin{proof}
Let us begin with the case $\ecat{sl}{3}{9}{3}$. It is well known that there is a conformal inclusion $SU(3)_9 \subset (E_6)_1$ which induces a commutative algebra object $A \in \cat{sl}{3}{9}$. The category $\operatorname{Mod}( \cat{sl}{3}{9} , A)$ is then a unitary fusion category, as described in \cite[Figure 4]{MR2265471}. The category $\operatorname{Mod}( \cat{sl}{3}{9} , A)$ is $\Z{3}$-graded. Let $\cC_{3,9,3}  $ be the trivially graded subcategory of this grading.

By \cite[Corollary 4.8]{MR1815993} we have that
\[  \mathcal{Z}(   \operatorname{Mod}( \cat{sl}{3}{9} , A)      ) \simeq \cat{sl}{3}{9}  \boxtimes (\operatorname{Mod}( \cat{sl}{3}{9} , A)^0)^\text{rev}.\]
For this particular case, we have that $ \operatorname{Mod}( \cat{sl}{3}{9} , A)^0 \simeq \cat{e}{6}{1} \simeq \operatorname{Vec}(\Z{3} , \{1,e^{2 i \pi \frac{2}{3}},e^{2 i \pi \frac{2}{3}}\})$.

Now the results of \cite{MR2587410} allow us to compute the Drinfeld centre of the subcategory $\cC_{3,9,3}  $ in terms of the Drinfeld centre of $ \operatorname{Mod}( \cat{sl}{3}{9} , A)$. This gives that
\[ \mathcal{Z}(      \cC_{3,9,3}          )    \simeq \ecat{sl}{3}{9}{3} \boxtimes \operatorname{Vec}(\Z{3} , \{1,e^{2 i \pi \frac{1}{3}},e^{2 i \pi \frac{1}{3}}\})\]
as desired.

The case of $\ecat{sl}{4}{8}{4}$ is almost identical, except now we use the conformal inclusion $SU(4)_8 \subset Spin(20)_1$. The structure of the associated algebra object, and category of modules, can be found in \cite[Section 2.6]{MR2506168}.

The case of $\ecat{sl}{5}{5}{5}$ should follow in the same manner, where now we work with the conformal inclusion $SU(5)_5 \subset Spin(24)_1$. However, the author was unable to find a suitable description of the category of modules in this case. Instead we have \cite[Theorem 3.2]{MR3764563} which proves the precise statement in this case. They show that $\cC_{5,5,5} $ is the even part of the $3^{\Z{2}\times \Z{2}}$ subfactor.
\end{proof}

With these alternate identifications of the categories $\ecat{sl}{3}{9}{3} $, $\ecat{sl}{4}{8}{4} $, and $\ecat{sl}{5}{5}{5}$ identified, we now have the tools to construct their exceptional braided auto-equivalences, and hence determine their braided auto-equivalence groups. We are able to complete this for the cases $\ecat{sl}{3}{9}{3} $ and $\ecat{sl}{5}{5}{5}$ in this paper. Let us begin with $\ecat{sl}{3}{9}{3} $.

\begin{lemma}
We have
\[ \BrAut(    \ecat{sl}{3}{9}{3}   ) = S_4 \]
\end{lemma}
\begin{proof}
From Theorem~\ref{thm:candy} we have that $ \BrAut(    \ecat{sl}{3}{9}{3}   ) $ is either $D_3$ or $S_4$. By analysing the fusion rings and twists of
\[ \mathcal{Z}(      \cC_{3,9,3}          )    \simeq \ecat{sl}{3}{9}{3} \boxtimes \operatorname{Vec}(\Z{3} , \{1,e^{2 i \pi \frac{1}{3}},e^{2 i \pi \frac{1}{3}}\})\]
we see that
\[  \BrAut(  \mathcal{Z}(      \cC_{3,9,3}          )   ) = \BrAut( \ecat{sl}{3}{9}{3} )   \times \Z{2}.\]
It is proven in \cite[Section 10.6]{MR3635673} that $\Out(    \cC_{3,9,3} ) = D_4$. From \cite{MR3778972} we have an embedding
\[  \Out(    \cC_{3,9,3} ) \to  \BrAut(  \mathcal{Z}(      \cC_{3,9,3}          )       ) . \]
Thus $\BrAut( \ecat{sl}{3}{9}{3} )   \times \Z{2}$ has a subgroup isomorphic to $D_4$. This is only possible if $\BrAut(    \ecat{sl}{3}{9}{3}   ) = S_4$.
\end{proof}

We now deal with the case of $\ecat{sl}{5}{5}{5}$. This case has been examined in the literature previously \cite{MR3764563,MR4001474}.

\begin{lemma}
We have
\[ \BrAut(    \ecat{sl}{5}{5}{5}   ) = A_5. \]
\end{lemma}
\begin{proof}
From Theorem~\ref{thm:candy} we have that $ \BrAut(    \ecat{sl}{5}{5}{5}   ) $ is either $D_5$ or $A_5$. By analysing the fusion rings and twists of
\[     \mathcal{Z}(      \cC_{5,5,5}          ) \simeq \ecat{sl}{5}{5}{5} \boxtimes \operatorname{Vec}(\Z{2}\times \Z{2} ,\{1, -1,-1,-1\}).\]
we see that
\[  \BrAut(  \mathcal{Z}(      \cC_{5,5,5}          )   ) = \BrAut( \ecat{sl}{5}{5}{5} )   \times S_3.\]
It is proven in \cite[Theorem 9.4]{MR3827808} that $\Out(    \cC_{5,5,5} ) = A_4$. Thus $\BrAut( \ecat{sl}{5}{5}{5} )   \times S_3$ has a subgroup isomorphic to $A_4$. This is only possible if $\BrAut(    \ecat{sl}{5}{5}{5}   ) = A_5$.
\end{proof}

Unfortunately we are unable to use the quadratic category $\cC_{4,8,4}$ to construct the exceptional braided auto-equivalence of $\ecat{sl}{4}{8}{4}$. This is because an explicit construction of the quadratic category $\cC_{4,8,4}$ has yet to be given. One way to construct this category is via the Cuntz algebra method, where it will be realised as endomorphisms on the $C^*$-algebra $O_{12} \rtimes \Z{2}$. The large multiplicity spaces of the quadratic category $\cC_{4,8,4}$ means that this method required solving for roughly $1700$ complex variables in $20000$ polynomial equations. This makes the problem too complex, even for modern computer algebra programs.

Instead we construct the exceptional braided auto-equivalence of $\ecat{sl}{4}{8}{4}$ using a coincidence of categories to connect it to $\mathfrak{so}_8$, similar to the $\ecat{sl}{2}{16}{2}$ case. Recall that $\mathfrak{sl}_4$ is isomorphic to $\mathfrak{so}_6$. Hence the category $\cat{sl}{4}{8}$ can be viewed as $\cat{so}{6}{8}$. Now level-rank duality will give a (non-trivial) connection between $\cat{so}{6}{8}$ and $\cat{so}{8}{6}$. Via this connection we can pull back the triality of $\mathfrak{so}_8$ in order to obtain an order 3 auto-equivalence of $\ecat{sl}{4}{8}{4}$, which forces $\BrAut(    \ecat{sl}{4}{8}{4}   ) = S_4$.

Let us expand more on this connection between $\cat{so}{6}{8}$ and $\cat{so}{8}{6}$. These categories have different ranks, so even as abelian categories they are not equivalent. Thus there are several steps we must take to get some sort of equivalence. First let $ \cC^\text{Vec}(   \mathfrak{so}_N,k  )$ be the $\Z{2}$-graded subcategory of $\cat{so}{N}{k}$ generated by the ``vector representation'' $\Lambda_1$. To get ``orthogonal categories'', where level-rank duality applies we have to `'add in'' the determinant representation by taking the $\Z{2}$-equivariantisation via the $D_n$ Dynkin diagram symmetry. This gives us the braided equivalence
\[     \cC^\text{Vec}(   \mathfrak{so}_6,8  )^{\Z{2}} \simeq       [  \cC^\text{Vec}(   \mathfrak{so}_8,6  )^{\Z{2}}     ]^{- \text{rev}}.                \]
Where $- \text{rev}$ means to take the reverse the braiding, and to negate it on the non-trivial piece of the grading. However this equivalence doesn't preserve the determinant representations, so we can't de-equivariantise by a single $\Rep(\Z{2})$ subcategory to obtain a braided equivalence. Instead we must de-equivariantise by the maximal Tannakian subcategory. This gives us a braided equivalence
\[   \ecat{so}{6}{8}{2} =\cC^\text{Vec}(   \mathfrak{so}_6,8  )_{  \langle 8\Lambda_1\rangle  } \simeq       [  \cC^\text{Vec}(   \mathfrak{so}_8,6  )_{\langle 6\Lambda_1\rangle}     ]^{- \text{rev}}=  [\ecat{so}{8}{6}{2}]^{- \text{rev}}.                \]

However triality doesn't preserve the $\langle 8\Lambda_1\rangle$ subcategory of $\cC^\text{Vec}(   \mathfrak{so}_8,6  )$, so it won't descend to $[\ecat{so}{8}{6}{2}]^{- \text{rev}}$. Thus we have to take local modules with respect to the remaining $\Rep(\Z{2})$ subcategory to get
\[   \ecat{so}{6}{8}{4} \simeq [\cat{so}{8}{6}^0_{\Rep(\Z{2}\times \Z{2})}]^{\text{rev}}.              \]
Triality now preserves the $\Rep(\Z{2}\times \Z{2})$ subcategory of $\cat{so}{8}{6}$ and hence descends to give us an order 3 auto-equivalence of $\ecat{so}{6}{8}{4} = \ecat{sl}{3}{8}{4}$.

To make this all precise we begin the following lemma, which formalises type $D$-$D$ level-rank duality.
\begin{lemma}
Let $N$ and $k$ be even integers. We have a braided equivalence
\[    \cC(   \mathfrak{so}_N,k  )^0_{\langle k\Lambda_1\rangle} \simeq [  \cC(   \mathfrak{so}_k,N  )^0_{\langle N\Lambda_1\rangle} ]^{-\text{rev}}   \]
\end{lemma}
\begin{proof}
Let
\[X_{N,k} := (\Lambda_1, +) \in     \cC^\text{Vec}(   \mathfrak{so}_N,k  )^{\Z{2}},\]
and
\[X_{k,N} := (\Lambda_1, +) \in     \cC^\text{Vec}(   \mathfrak{so}_k,N  )^{\Z{2}}.\]
From \cite{MR2132671} we have that
\[   \mathcal{P}_{X_{N,k}} =\overline{ \text{BMW}\left(e^{2\pi i \frac{1}{2(N+k-2)}},e^{2\pi i \frac{N-1}{2(N+k-2)}}\right)},\]
and
\[    \mathcal{P}_{X_{k,N}} = \overline{ \text{BMW}\left(e^{2\pi i \frac{1}{2(N+k-2)}},e^{2\pi i \frac{k-1}{2(N+k-2)}}\right)},   \]
where $\overline{\text{BMW}(q,r)}$ is the semi-simplified planar algebra generated by a single crossing and the Kauffman relation (see \cite[Section 7]{MR2132671} \cite[Definition 2.12]{ABCG}). It is routine to verify that
\[ \raisebox{-.5\height}{ \includegraphics[scale = .5]{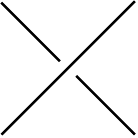}}  \mapsto - \raisebox{-.5\height}{ \includegraphics[scale = .5]{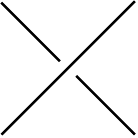}}      \]
gives an isomorphism of planar algebras
\[ \mathcal{P}_{X_{N,k}} \to  \mathcal{P}_{X_{k,N}}.\]
Via \cite[Theorem A]{1607.06041} this induces a braided equivalence
\[  \cC^\text{Vec}(   \mathfrak{so}_N,k  )^{\Z{2}} \to   [   \cC^\text{Vec}(   \mathfrak{so}_k,N  )^{\Z{2}} ]^{- \text{rev}}  .     \]
Lets define the two groups
\[  G_{N,k}:= \Inv\left(     \cC^\text{Vec}(   \mathfrak{so}_N,k  )^{\Z{2}}  \right )= \{   (\mathbf{1}, +), \quad  (\mathbf{1}, -)  , \quad  (k\Lambda_1, +), \quad \text{and} \quad  (k\Lambda_1, -)          \},\]
and
\[  G_{k,N}:= \Inv\left(     [   \cC^\text{Vec}(   \mathfrak{so}_k,N  )^{\Z{2}} ]^{- \text{rev}} \right )= \{   (\mathbf{1}, +), \quad  (\mathbf{1}, -)  , \quad  (N\Lambda_1, +), \quad \text{and} \quad  (N\Lambda_1, -)          \}.\]
The above braided equivalence will preserve the group of invertibles, so we have that it maps $G_{N,k}$ to $G_{k,N}$.

As $k\Lambda_1$ and $N\Lambda_1$ are symmetric objects in their respective categories and have trivial twist (these facts requires both $N$ and $k$ even), we get that
\[ \Rep( G_{N,k}) \subseteq Z_2(      \cC^\text{Vec}(   \mathfrak{so}_N,k  )^{\Z{2}}                ) ,      \]
and
\[  \Rep(G_{k,N}) \subseteq Z_2(        [   \cC^\text{Vec}(   \mathfrak{so}_k,N  )^{\Z{2}} ]^{- \text{rev}}               ).      \]
Thus we can de-equivariantise to obtain a braided isomorphism
\[    \cC^\text{Vec}(   \mathfrak{so}_N,k  )^{\Z{2}}_{ \Rep( G_{N,k})} \to   [   \cC^\text{Vec}(   \mathfrak{so}_k,N  )^{\Z{2}} ]_{ \Rep( G_{k,N})}^{- \text{rev}}  .     \]
From \cite[Theorem 4.9]{MR2609644} we have that $ \langle (\mathbf{1},-)\rangle$ generates a copy of $\Rep(\Z{2}) \subseteq\cC^\text{Vec}(   \mathfrak{so}_N,k  )\changeX{^{\Z{2}}} $  (resp. $  [   \cC^\text{Vec}(   \mathfrak{so}_k,N  )^{\Z{2}} ]^{- \text{rev}}    $), and that de-equivaraiantising $   \cC^\text{Vec}(   \mathfrak{so}_N,k  )^{\Z{2}}    $   (resp. $ [   \cC^\text{Vec}(   \mathfrak{so}_k,N  )^{\Z{2}} ]^{- \text{rev}} $)  by this copy of $\Rep(\Z{2})$ recovers $\cC^\text{Vec}(   \mathfrak{so}_N,k  )$ (resp. $ [   \cC^\text{Vec}(   \mathfrak{so}_k,N  ) ]^{- \text{rev}})$ up to braided equivalence. As $\langle (\mathbf{1}, -)\rangle$ is normal in $G_{N,k}$ and $G_{k,N}$ respectively we can use the above fact, along with \cite[Proposition 4.12]{MR1749250} to get braided equivalences
\[     \cC^\text{Vec}(   \mathfrak{so}_N,k  )^{\Z{2}}_{ \Rep( G_{N,k})}\simeq  \cC^\text{Vec}(   \mathfrak{so}_N,k  )_{ \langle k\Lambda_1 \rangle},      \]
and
\[     [   \cC^\text{Vec}(   \mathfrak{so}_k,N  )^{\Z{2}} ]_{ \Rep( G_{k,N})}^{- \text{rev}}\simeq [   \cC^\text{Vec}(   \mathfrak{so}_k,N  ) ]_{ \langle N\Lambda_1 \rangle}^{- \text{rev}}.    \]
\changeX{ As $\cat{so}{N}{k}$ is modular, and $\cC^\text{Vec}(   \mathfrak{so}_N,k  )$ is the adjoint subcategory with respect to a $\mathbb{Z}_2$-grading on $\cat{so}{N}{k}$, we have that $\cC^\text{Vec}(   \mathfrak{so}_N,k  ) =Z_2(  \cat{so}{N}{k}, H)$ for some $H\subseteq \operatorname{Inv}(\cat{so}{N}{k})$ (this is a consequence of the isomorphism $\operatorname{Inv}(\mathcal{C}) \cong \mathcal{U}(\mathcal{C})$ \cite[Proposition 4.14.3]{EGNO}). From the modular data of $\cat{so}{N}{k}$ we can see that $k\Lambda_1$ is symmetric in $\cC^\text{Vec}(   \mathfrak{so}_N,k  )$, and the invertibles $k\Lambda_N$ and $k\Lambda_{N-1}$ do not centralise $\Lambda_1 \in \cC^\text{Vec}(   \mathfrak{so}_N,k  )$. Thus $H = \langle k\Lambda_1 \rangle$, and so $\cC^\text{Vec}(   \mathfrak{so}_N,k  )\simeq Z_2( \cat{so}{N}{k},     \langle k\Lambda_1 \rangle)$.} This implies that

\[\cC(   \mathfrak{so}_N,k  )^0_{\langle k\Lambda_1\rangle} =\cC^\text{Vec}(   \mathfrak{so}_N,k  )_{ \langle k\Lambda_1 \rangle}.    \]
The same argument works to show that
\[\cC(   \mathfrak{so}_k,N  )^0_{\langle N\Lambda_1\rangle} =\cC^\text{Vec}(   \mathfrak{so}_k,N  )_{ \langle N\Lambda_1 \rangle},    \]
which completes the proof.
\end{proof}
As a corollary we get our desired exceptional braided auto-equivalence.
\begin{cor}
We have
\[   \BrAut(    \ecat{sl}{4}{8}{4}   ) = S_4  .     \]
\end{cor}
\begin{proof}
Via a coincidence of Dynkin diagrams we have
\[      \cat{sl}{4}{8}  \simeq   \cat{so}{6}{8}.    \]
Thus the above Lemma gives a braided equivalence
\[        \ecat{sl}{4}{8}{2} \simeq    [\ecat{so}{8}{6}{2}]^{- \text{rev}}.\]
By taking local-modules with respect to the remaining $\Rep(\Z{2})$ subcategory we obtain
\[        \ecat{sl}{4}{8}{4} \simeq    [\cat{so}{8}{6}^0_{ \Rep( \Z{2}\times\Z{2} )}]^{- \text{rev}} = [\cat{so}{8}{6}^0_{ \Rep( \Z{2}\times\Z{2} )}]^{\text{rev}} .\]
The triality of the $D_4$ Dynkin diagram induces an order 3 braided auto-equivalence of $\cat{so}{8}{6}^\text{- rev}$ by \cite[Corollary 2.7]{ABCG}. This auto-equivalence preserves the $\Rep( \Z{2}\times\Z{2} )$ subcategory, and thus descends to  $[\cat{so}{8}{6}^0_{ \Rep( \Z{2}\times\Z{2} )}]^{\text{rev}}$. To see that this braided auto-equivalence is non-trivial, we observe that triality will send
\[   2\Lambda_1 \to 2\Lambda_3        \]
in $\cat{so}{8}{6}^\text{- rev}$, which implies that the induced braided auto-equivalence will map
\[    \cF_{\Z{2}\times \Z{2}}(   2\Lambda_1  ) \mapsto  \cF_{\Z{2}\times \Z{2}}(   2\Lambda_3 )        \]
in $[\cat{so}{8}{6}^0_{ \Rep( \Z{2}\times\Z{2} )}]^{\text{rev}}$. However the two objects $ 2\Lambda_1$ and $2\Lambda_3$ are not in the same orbit under the simple currents of $\cat{so}{8}{6}^\text{-rev}$ which implies
\[    \cF_{\Z{2}\times \Z{2}}(   2\Lambda_1  ) \ncong  \cF_{\Z{2}\times \Z{2}}(   2\Lambda_3 ).        \]
Thus the induced braided auto-equivalence of $[\cat{so}{8}{6}^0_{ \Rep( \Z{2}\times\Z{2} )}]^{\text{rev}}$ is non-trivial.

Recall from Theorem~\ref{thm:candy} that $  \BrAut(    \ecat{sl}{4}{8}{4}   )$ is either $D_4$ or $S_4$. As we know there exists an order 3 braided auto-equivalence of $ \ecat{sl}{4}{8}{4} $, we can only have the latter option.
\end{proof}
\appendix

\section{Coincidences of Small Dimensions}\label{app:terry}
\begin{center}   \textbf{By Terry Gannon} \end{center}

In this appendix, we prove the following result regarding coincidences of dimensions in the categories $\cat{sl}{r+1}{k}$.

\begin{prop}\label{prop:samedim}
Let $X$ be a simple object of $\cat{sl}{r+1}{k}$ such that $\dim(X) = \dim(\Lambda_1 + \Lambda_r)$. Then either $X \in [\Lambda_1 + \Lambda_r]$, or
\begin{align*}
 (r,k) = (8,3), (8,15) \quad &\text{and} \quad X \in [\Lambda_3], &\text{ or}\\
 (r,k) = (2,9), (14,9) \quad &\text{and} \quad X \in [3\Lambda_1], &\text{ or}\\
 (r,k) = (3,6), (5,4) \quad &\text{and} \quad X \in [2\Lambda_2], &\text{ or}\\
 (r,k) = (7,4), (7,6) \quad &\text{and} \quad X \in [\Lambda_4], &\text{ or}\\
 (r,k) = (3,8), (5,8) \quad &\text{and} \quad X \in [4\Lambda_1].
\end{align*}
\end{prop}

The main technical tool we use to prove this result is the following lemma, which allows us to shuffle around the Dynkin labels of a simple object $X$ to decrease its dimension.

\begin{lemma}\label{lem:key}
Let $X = \sum_{i=0}^r   \lambda_i \Lambda_i$ be a simple object of $\cat{sl}{r+1}{k}$. For any two indices $0\leq j,l \leq r$, and integers $0 \leq c_j \leq \lambda_j$ and $0 \leq c_l \leq \lambda_l$ we have that
\[  \dim(X) \geq \min( \dim(X  - c_j\Lambda_j + c_j\Lambda_l) , \dim(X + c_l \Lambda_j - c_l \Lambda_l)    ),  \]
with equality if and only if $c_j=0$ or $c_l = 0$.
\end{lemma}
\begin{proof}
Note that we can write
\[   X = \frac{c_l}{c_j+c_l}( X  - c_j\Lambda_j + c_j\Lambda_l)  +  \frac{c_j}{c_j+c_l}( X  + c_l\Lambda_j - c_l\Lambda_l).   \]
Thus the result follows immediately from Lemma~\ref{lem:convex}.
\end{proof}

With this tool we can now prove Proposition~\ref{prop:samedim}.
\begin{proof}
Let $X$ be a simple object of $\cat{sl}{r+1}{k}$ such that $\dim(X) = \dim(\Lambda_1 + \Lambda_r)$. First assume that more than three of the labels $\lambda_i$ are non-zero. Pick two of these non-zero labels $\lambda_j, \lambda_l$. Then Lemma~\ref{lem:key} tells us that either $X  - \lambda_j\Lambda_j + \lambda_j\Lambda_l$ or $X + \lambda_l \Lambda_j - \lambda_l \Lambda_l$ has dimension less than or equal to $X$. Hence we get an object $X'\in \cat{sl}{r+1}{k}$ with $\dim(X') \leq \dim(\Lambda_1 + \Lambda_r)$ and with one less non-zero label than $X$. By repeating this process we can assume that $\dim(X) \leq \dim(\Lambda_1 + \Lambda_r)$ and $X$ has at most three non-zero labels.

Now suppose $X$ has exactly three non-zero labels. By applying a simple current symmetry we can assume that $X = \lambda_0 \Lambda_0 + \lambda_a \Lambda_a+\lambda_b \Lambda_b$. By applying Lemma~\ref{lem:key} with $c_0 = \lambda_0 -1$ and $c_a = \lambda_a - 1$ we get that
\[  \dim(X) \geq \min(\Lambda_a + \lambda_b\Lambda_b , (\lambda_0 + \lambda_a - 1)\Lambda_a + \lambda_b\Lambda_b).   \]
By repeating this process with the $\lambda_b$ label, and applying a simple current symmetry we get that $\dim(X) \geq \Lambda_{a'} +\Lambda_{b'}$ for some $0< a' < b'$, with equality if and only if $X \in [ \Lambda_{a'} +\Lambda_{b'} ]$. By level-rank duality we have that $\dim(\Lambda_{a'} +\Lambda_{b'}) = \dim( (b' - a')\Lambda_1 + a'\Lambda_2)$. By applying Lemma~\ref{lem:key} several times we obtain
\begin{align*}
 \dim( (b' - a')\Lambda_1 + a'\Lambda_2) &\geq \min( \dim(\Lambda_1 + \Lambda_2) ,  \dim(\Lambda_1 + (k-2)\Lambda_2) , \dim((k-2)\Lambda_1 + \Lambda_2)  )\\
  &=   \min( \dim(\Lambda_1 + \Lambda_2) , \dim(\Lambda_1 + \Lambda_r)).
  \end{align*}
  Hence all together we have
  \[   \dim(\Lambda_1 + \Lambda_r) \geq    \min( \dim(\Lambda_1 + \Lambda_2) , \dim(\Lambda_1 + \Lambda_r))   \]
  with equality only if $X \in [ \Lambda_{a'} +\Lambda_{b'} ]$.

  Using \cite[Equation (2.1c)
]{MR1887583} we compute that
  \[  \frac{\dim(\Lambda_1 + \Lambda_2)}{\dim(\Lambda_1 + \Lambda_r)} = \frac{\ssin{\pi (r+1)}{(k+r+1)}}{\ssin{3\pi} {(k+r+1)}}.  \]
  This is always $\geq 1$, with equality only if $r=2$ or $k=3$. Hence if $X$ has exactly three non-zero labels then $r =2$ or $k=3$ and $X \in [\Lambda_1 + \Lambda_2]$. However in these cases we get that $\Lambda_1 +\Lambda_2 \in [\Lambda_1 + \Lambda_r]$ and so $X\in [\Lambda_1 +\Lambda_r]$.

  Finally suppose $X$ has exactly two non-zero labels (if $X$ has one non-zero label, then $X = \mathbf{1}$). Then we can write $X = a\Lambda_b$ with $a,b \geq 2$. We can assume that $r\geq 3$ and hence $k\geq 4$ by level-rank duality. By applying Lemma~\ref{lem:key} with $c_0 = k-a-2$ and $c_b = a-2$ to get
  \[ \dim(a\Lambda_b) \geq \min ( \dim(2\Lambda_b ), \dim((k-2)\Lambda_b)   ) =  \dim(2\Lambda_b ).  \]
  By applying level-rank duality and using the same trick we find that $\dim(a\Lambda_b) \geq \dim(2\Lambda_2)$. We compute
  \[ 1\geq  \frac{\dim(2\Lambda_2)}{\dim(\Lambda_1 + \Lambda_r)}  =  \frac{\ssin{\pi}{k+r+1}\ssin{(r+1)\pi}{k+r+1}^2}{\ssin{2\pi}{k+r+1}^2\ssin{3\pi}{k+r+1}} \geq  \frac{\ssin{\pi}{k+r+1}\ssin{4\pi}{k+r+1}^2}{\ssin{2\pi}{k+r+1}^2\ssin{3\pi}{k+r+1}}, \]
  as $r\geq 3$ and $k\geq 4$. By level-rank duality we can assume that $r+1 \leq k$. Simple calculus shows that $\frac{\sin(x)\sin^2(4x)}{\sin^2(4x)\sin(3x)} > 1$ when $x < \frac{\pi}{10}$. Hence $r + k + 1 \leq 10$. For these finite possible cases, we can directly search to find when $\dim(a\Lambda_b) = \dim(\Lambda_1 + \Lambda_r)$. The only solutions are $(r,k) = (3,6), (5,4)$ where $X = 2\Lambda_2$.

  Finally (by level-rank duality) it suffices to consider $X = \Lambda_b$ for $b\leq \frac{r+1}{2}$. We have
  \[ \frac{\dim(\Lambda_b)}{\dim(\Lambda_1 + \Lambda_r)} = \frac{\ssin{\pi(k+1)}{1+r+k}\ssin{\pi(k-1)}{1+r+k}}{\ssin{\pi} {1 + r +k}^2}\prod_{j=1}^b \frac{\ssin{\pi(r+2 - j)}{1+r+k}}{\ssin{\pi j} {1 + r +k}}.  \]
  In particular this shows
  \[\dim(\Lambda_1) < \dim(\Lambda_2) < \cdots < \dim(\Lambda_{\frac{r+1}{2}}).   \]

  Let's now study when the terms $\frac{\dim(\Lambda_b)}{\dim(\Lambda_1 + \Lambda_r)}$ are equal to $1$. We will start by studying the case $b=1$, and will increase $b$ until be can show that this term is always strictly bigger than $1$.

    For the case of $b=1$, we find that $ \frac{\dim(\Lambda_1)}{\dim(\Lambda_1 + \Lambda_r)}=1$ if and only if $\Lambda_1 \in [\Lambda_1 + \Lambda_r]$. With the $b=1$ case done, we can now assume $r\geq 3$.

    For the case of $b=2$, we can use the inequality coming from the concavity of $\ln | \sin(x)|$:
    \begin{equation}\label{eq:concav}
    \sin(a) \sin(b)  < \sin(a -x)\sin(b + x)
    \end{equation}
     for $0 < b< a< \pi$ and $0 < x \leq \frac{a-b}{2}$, to get
    \[    \frac{\dim(\Lambda_2)}{\dim(\Lambda_1 + \Lambda_r)} < 1.  \]
    With the $b=2$ case done, we can now assume $r\geq 5$.

    For the case of $b=3$ we have to consider several subcases. If $r \in \{5,6,7\}$ and $k > r+1$, then we get $ \frac{\dim(\Lambda_3)}{\dim(\Lambda_1 + \Lambda_r)} <1$ from $\sin(\pi(r+1)/(1 + r + k)) < \sin(\pi(r+2)/(1 + r + k))$ and the fact that $\frac{\sin(x)\sin((r-1)x)}{\sin(2x)\sin(3x)} $ is decreasing for $0 <x  < \frac{\pi}{4}$. If $r \in \{5,6,7\}$ and $k \leq r+1$ then there are just a small number of cases to check. If $r = 8$ then $ \frac{\dim(\Lambda_3)}{\dim(\Lambda_1 + \Lambda_r)}$ is a strictly increasing function of $k$, which equals $1$ at $k=15$. When $r\geq 9$, $k \geq 4$, and $1 + r + k\geq 19$ we can use Equation~\eqref{eq:concav} again to obtain
    \[   \frac{\dim(\Lambda_3)}{\dim(\Lambda_1 + \Lambda_r)} > 1.   \]
    When $r\geq 9$ and $k=3$ we get
    \[  \frac{\dim(\Lambda_3)}{\dim(\Lambda_1 + \Lambda_r)} =\frac{ \ssin{\pi}{1 + r + k} \ssin{5\pi}{1 + r + k}}{\ssin{2\pi}{1 + r + k}^2}>1.  \]
    With the $b=3$ case done we can now assume $r\geq 7$.

    Finally for the case of $b=4$ we have that when $r = 7$ and $k\geq 3$, we have $\frac{\dim(\Lambda_4)}{\dim(\Lambda_1 + \Lambda_r)} > 1$.

    Thus all we have remaining is a finite list of pairs $(r,k)$ where we could possibly have $\dim(\Lambda_b) = \dim(\Lambda_1 + \Lambda_r)$. Searching these pairs, and applying level-rank duality gives the statement of the proposition.
\end{proof}

\bibliography{bibliography}
\bibliographystyle{plain}
\end{document}